\documentclass[twoside]{article}

\usepackage[accepted]{aistats2021}
\usepackage[round]{natbib}

\usepackage[utf8]{inputenc} % allow utf-8 input
\usepackage[T1]{fontenc}    % use 8-bit T1 fonts
\usepackage{hyperref}       % hyperlinks
\usepackage{url}            % simple URL typesetting
\usepackage{booktabs}       % professional-quality tables
\usepackage{amsfonts}       % blackboard math symbols
\usepackage{nicefrac}       % compact symbols for 1/2, etc.
\usepackage{microtype}      % microtypography
\usepackage{color,soul}
\usepackage[dvipsnames]{xcolor}
\usepackage{hyperref}
\usepackage{url}
\usepackage{booktabs}
\usepackage{subfigure}
\usepackage{float}
\usepackage{indentfirst}
\usepackage{smile}
\usepackage{appendix}
\usepackage{algorithm}
\usepackage{algorithmic}
\usepackage{amssymb}
\usepackage{color}
\usepackage{enumitem}
\usepackage{diagbox}
\usepackage{multicol}
\usepackage{wrapfig}
\usepackage{chngcntr}
\usepackage{ulem}

\counterwithout{equation}{section}

\newenvironment{steps}
 {\begin{enumerate}[label=Step \Roman*,leftmargin=*,align=left]}
 {\end{enumerate}}

\newtheorem{prop}{Proposition}

\newtheorem{thm}{Theorem}

\newtheorem{defi}{Definition}

\newtheorem{lem}{Lemma}

\begin{document}


\runningauthor{Wei, Zhu, Zhang, Xie}

\twocolumn[

\aistatstitle{Goodness-of-Fit Test for Mismatched Self-Exciting Processes}

% \aistatstitle{Goodness-of-fit test for mismatched self-exciting process models}

\aistatsauthor{ Song Wei \quad Shixiang Zhu \quad Minghe Zhang \quad Yao Xie}

\aistatsaddress{School of Industrial and Systems Engineering, \\ Georgia Institute of Technology, Atlanta, Georgia, USA} ]

\begin{abstract}

Recently there have been many research efforts in developing generative models for self-exciting point processes, partly due to their broad applicability for real-world applications. However, rarely can we quantify how well the generative model captures the nature or ground-truth since it is usually unknown. The challenge typically lies in the fact that the generative models typically provide, at most, good approximations to the ground-truth (e.g., through the rich representative power of neural networks), but they cannot be precisely the ground-truth. We thus cannot use the classic goodness-of-fit (GOF) test framework to evaluate their performance. In this paper, we develop a GOF test for generative models of self-exciting processes by making a new connection to this problem with the classical statistical theory of Quasi-maximum-likelihood estimator (QMLE). We present a non-parametric self-normalizing statistic for the GOF test: the Generalized Score (GS) statistics, and explicitly capture the model misspecification when establishing the asymptotic distribution of the GS statistic. Numerical simulation and real-data experiments validate our theory and demonstrate the proposed GS test's good performance.

% we provide goodness-of-fit tests for generative models by leveraging a new connection of this problem with the classical statistical theory of mismatched maximum-likelihood estimator (MLE). We present a non-parametric self-normalizing test statistic for the goodness-of-fit test based on Generalized Score (GS) statistics. We further establish asymptotic properties for MLE of the Quasi-model (Quasi-MLE), asymptotic $\chi^2$ null distribution and power function of GS statistic. Numerical experiments validate the asymptotic null distribution as well as the consistency of our proposed GS test.
\end{abstract}

%\vspace{-0.2in}
\section{Introduction}
%\vspace{-0.1in}

Self- and mutual- exciting point processes, as known as the Hawkes processes, are introduced by the original papers by \citet{hawkes1971point,hawkes1971spectra,hawkes1974cluster}. They become popular in machine learning due to their wide applicability in modeling triggering effect in discrete event data, which is ubiquitous in modern applications ranging from seismology \citep{ogata1988statistical,ogata1999seismicity,zhuang2011next}, infectious disease modeling \citep{meyer2014power,schoenberg2019recursive}, crime events \citep{doi:10.1198/jasa.2011.ap09546}, wildfire occurrence \citep{peng2005space}, civilian deaths in Iraq \citep{lewis2012self}, terrorist activity forecasting \citep{porter2012self}, social network analysis and so on.

% Self- and mutual- exciting point processes, as known as the Hawkes processes, introduced by the original papers by Hawkes \cite{hawkes1971point,hawkes1971spectra,hawkes1974cluster}, become popular in machine learning. Recently, there have been many research advances in generative models for self-exciting point processes, partly due to their broad applicability for real-world applications to model the triggering effect in discrete event data, ranging from seismology \cite{ogata1988statistical,ogata1999seismicity,zhuang2011next}, infectious disease modeling \cite{meyer2014power,schoenberg2019recursive}, crime events \cite{doi:10.1198/jasa.2011.ap09546}, wildfire occurrence \cite{peng2005space}, civilian deaths in Iraq  \cite{lewis2012self}, terrorist activity forecasting \cite{porter2012self}, social network analysis and so on.

Classical Hawkes processes are largely parametric, which focus on modeling the conditional intensity function of the point process (since the conditional intensity function completely specifies the distribution of the process). Hawkes process assumes that the intensity function consists of the sum of a deterministic background intensity (which can be time-varying) and a stochastic term, which captures the influence from the past events. It is common to assume that the influence from past events is additive, and the so-called {\it triggering function} measures an individual event's influence. One key problem in the Hawkes process is to specify the triggering kernel. Popular parametric triggering functions include exponential kernel, power kernel, and Matérn kernel \citep{reinhart2018review}. 

When facing more complex data with complex temporal triggering patterns, parametric models can become too restrictive and even mis-specified. Thus, recently, there have been many efforts in developing more general generative models for point processes, including probability weighted kernel estimation with adaptive bandwidth \citep{zhuang2002stochastic}, probability weighted histogram estimation \citep{marsan2008extending} and with inhomogeneous spatial background rate \citep{fox2016spatially} and neural Hawkes process \citep{mei2017neural}.

Since the specified models (including those generative models) are very likely to be incorrect due to the ignorance of the ground-truth, a natural and important question yet to be answered is which model to select in practice. Here, we proposed to use how well those models capture the data, i.e. {\it goodness-of-fit} of these Hawkes process models, as the metric to rank models in practice. As well-said in \citet{engle1984wald}:
{\it "At any stage in the specification search, it may be desirable to determine whether an adequate representation of the data has been achieved."} For generative models, since they tend to be further away from the probabilistic framework of Hawkes processes, it is more difficult to evaluate their GOF to the real data. For these generative models, the classic statistical GOF test framework
may not apply.

There are two major difficulties in utilizing existing GOF tests for the self-exciting point processes. (1) Generative models typically provide, at most, good approximations to the ground-truth (e.g., through the rich representative power of neural networks), but they cannot be precisely the ground-truth. 
For instance, it is unlikely that neural networks truly specify the data distribution; rather, the neural networks are being used because of their universal approximation power and can generate a good approximation to the ground-truth \citep{mei2017neural}.
Typically, it is impractical to access the GOF via testing with unknown ground-truth $g^*$, as illustrated in the left panel in Figure~\ref{fig:set_up_mismatch}. In both theory and practice, the best we can do is to test how close our fitted model $\hat g$ is to the approximation $g_0$, as illustrated in the middle panel in Figure~\ref{fig:set_up_mismatch}. 
Nevertheless, we still consider model misspecification explicitly since it is vital in establishing the asymptotic performance of our proposed GOF test.
(2) When we fit conditional intensities from various families, direct comparison between GOF measures from different families is not reasonable; we need to find a unifying space to access comparable GOF measures for all considered families, as illustrated by red lines in the right panel in Figure~\ref{fig:set_up_mismatch}.
This space, or rather function family, for GOF, should be carefully chosen such that it is both expressive enough and not too complex to develop a valid, consistent, and tractable test statistic thereon.

\begin{figure}[!htp]
%\vspace{-0.1in}
\centering
\includegraphics[scale=0.33]{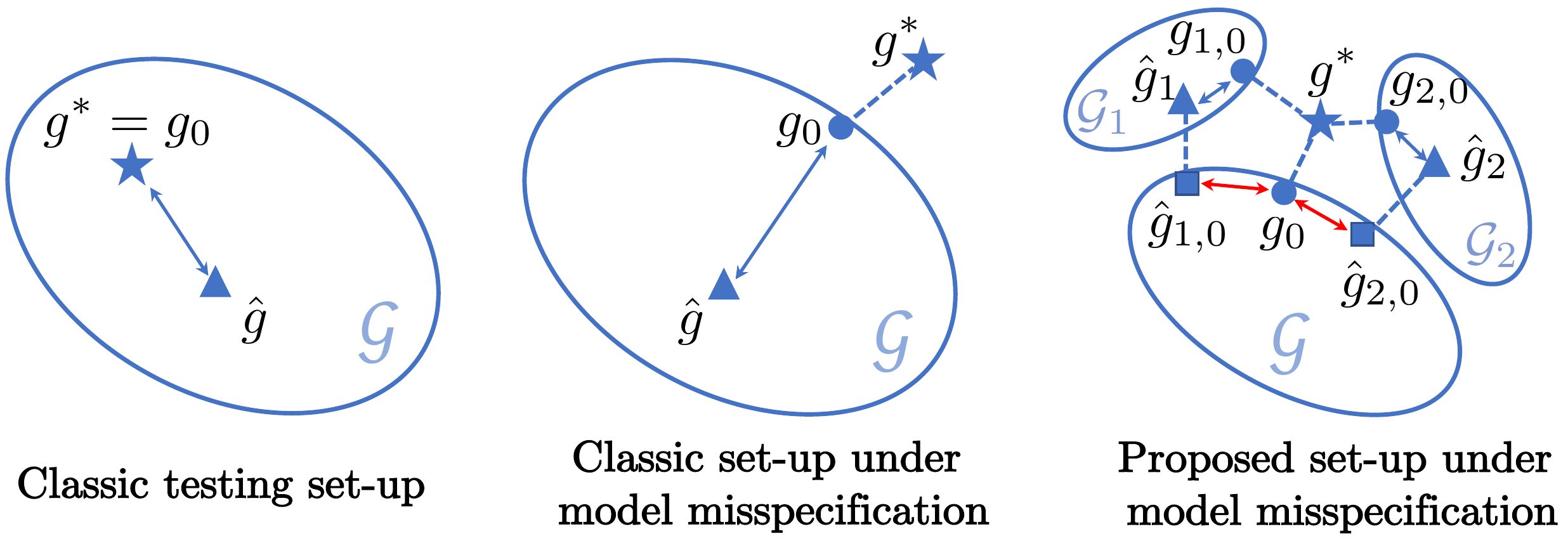}
 \vspace{-.2in}
\caption{The ground-truth is $g^*$; the assumed/specified family of candidate models is $\mathcal G$ in the left two panels and $\mathcal G_1$ and $\mathcal G_2$ in the right panel. GOF addresses how close the fitted model $\hat g$ is to the unknown true one $g^*$. In classic set up, one assumes there exists a $g_0 \in \mathcal G$ such that $g^* = g_{0}$. Under a more general model misspecification case where $g^*$ may not be contained in $\mathcal G$, classic GOF measures the distance between $\hat g$ and a good approximate $g_0 \in \mathcal G$ to $g^*$ (in K-L divergence sense). When we want to rank models, we need to find the approximation of $\hat g_1$, $\hat g_2$, and $g^*$ in a unifying space $\mathcal G$ and compare models $\hat g_1$, $\hat g_2$ therein.}\label{fig:set_up_mismatch}
%\vspace{-0.1in}
\end{figure}

GOF test for the whole conditional intensity has been developed by \citet{ogata1988statistical,schoenberg2003multidimensional}, but the theory therein is established under the classic set-up and may fail to generalize to model misspecification setting. 
Moreover, the triggering effect is the main effect-of-interest in many Hawkes process models since (1) it characterizes the dynamics between events and (2)
the background rate can be separately estimated from the well-established declustering procedure \citep{zhuang2002stochastic,marsan2008extending,fox2016spatially}.
However, the background rate usually dominates the conditional intensity, and the existing tests may not detect subtle triggering function differences. Thus, a principled method to quantify the {\it goodness-of-fit for triggering effect in Hawkes processes under model misspecification} is essential. 

\textbf{Contribution.}  
In this paper, we present a non-parametric goodness-of-fit (GOF) test statistic, called the {\it Generalized Score} (GS), which can be broadly applied to evaluating the self-exciting part in Hawkes process generative models. The GS test is constructed by translating the GOF test into a two-sample test: whether the real data and synthetic data from the generative model have the same distribution? Based on this, we derive the likelihood score statistic with estimated piecewise constant kernels, which is flexible and has little model restrictions. We further establish asymptotic properties for MLE of the Quasi-model (QMLE), asymptotic $\chi^2$ null distribution, as well as the power function of GS statistic. The main ingredients of our analysis include (1) making a connection between GS test and the classic theory on MLE under model misspecification (QMLE) \citep{white1982maximum} and (2) generalizing the asymptotic properties of MLE of Hawkes process in \citet{ogata1978asymptotic} to model misspecification case. Our GS test provides a tool for model diagnosis and comparison of the self-exciting part in Hawkes process generative models. We demonstrate the effectiveness of our proposed test via numerical simulation and real-data examples.  

Several features of our GS test include: (1) We develop the test for generative models considering their inherent ``model misspecification nature''; (2) we focus on GOF of the triggering effect in Hawkes process models; (3) due to its construction, the GS statistic enjoys simple asymptotic distribution specified by $\chi^2$ distribution and analytical form of the power function, which enables us to 
%select a critical value to
calibrate the test without sampling.

\textbf{Related Work.}  
The one-sample goodness-of-fit problem is closely related to the two-sample test problem. For independent and identically distributed (i.i.d.) observations, two-sample test is well studied (e.g. energy statistic \citep{szekely2004testing,baringhaus2004new} and maximum mean discrepancy (MMD) \citep{gretton2012kernel}) and so is the GOF based on it. 
\citet{chwialkowski2016kernel} developed Stein operator based MMD (which they call squared Stein discrepancy) and changed the two-sample test statistic to a one-sample GOF test metric. \citet{bounliphone2015test} reformulated the one-sample GOF problem into a two-sample test problem and developed a model selection tool based on MMD. Extension of those methods to point process is missing until \citet{pmlr-v89-yang19a} proposed a kernel goodness-of-fit test by defining a Stein discrepancy for generic point process; However, a common drawback of a kernel-based test is that the null distribution is hard to evaluate (since they depend on infinite series involving the eigenvalues of the kernel). In contrast, our GS statistic follows a simple $\chi^2$ null distribution and is easy to calibrate. Our proposed method allows the distribution under the null to be flexible and estimated from data by comparing the data to the generative model via the test statistic. Other model diagnostics include likelihood of fitted model and the observed data \citep{schorlemmer2007earthquake} and Information Criterion (IC) \citep{chen2018performance}. The likelihood is the most commonly used, but overfitting makes it less convincing and even questionable (as discussed via numerical simulation). \citet{chen2018performance} assumed correct model specification, which typically does not hold in the real study, and the consistency result of IC is restricted to exponential triggering function case. For more on the kernel-based two-sample test as well as model diagnosis and selection method of the point process, one can refer to \citet{harchaoui2013kernel} and \citet{bray2013assessment}.

%\vspace{-0.1in}
\section{Problem set-up}\label{background}
%\vspace{-0.1in}

We first introduce some necessary mathematical preliminaries, and then formulate the one-sample goodness-of-fit problem into a two-sample test problem. 
%\vspace{-0.1in}
\subsection{Mathematical background}
%\vspace{-0.1in}
Consider a counting process $\{N(t ) : t \geq 0\}$, with associated history
 $\cH_{0,t} = \{t_i: 0 < t_i < t\} \  (t \geq 0)$ indicating the occurrence time of a sequence of discrete events. For simplicity, we use $\cH_{t}$ instead.  A point process is characterized by its conditional intensity, which is defined as:
$$
\lambda\left(t | \cH_{t}\right) 
=\lim _{\Delta t \downarrow 0} \EE\left[N\{(t, t+\Delta t) \} | \cH_{t}\right]/\Delta t.
$$
Hawkes process is a self-exciting point process with conditional intensity takes the following form:
 \begin{equation}
  \lambda\left(t | H_{t}\right)=\mu+\sum_{\left\{i: t_{i}<t\right\}} \phi\left(t-t_{i}\right), 
  \label{eq:hawkes}
\end{equation}
where $\mu$ is called the background intensity and $\phi : (0, \infty) \rightarrow [0, \infty)$ is called the triggering function. 
% To simplify the notation, we can define $\phi$ on $\RR$ but it takes value zero on $(-\infty, 0]$. 

We assume the separability of triggering function into components for magnitude and time:
 $   \phi\left(t-t_{i}\right) = \alpha g(t-t_i),$
where temporal triggering function $g$ is a probability density function (p.d.f.) and $\alpha$ represents the magnitude of triggering effect, i.e. how many subsequent events one event can trigger on average.
Given the past trajectory $\cH_T$ with $N$ events, the log-likelihood over time interval $[0,T]$ can be expressed as:
$$\ell(\theta)=\sum_{i=1}^{N} \log \left(\lambda(t_{i}|\cH_{t_{i}})\right)-\int_{0}^{T}  \lambda(u |\cH_u) \mathrm{d} u.$$
One can refer to \citet{laub2015hawkes} and \citet{reinhart2018review} for a more comprehensive introduction of Hawkes process and a detailed deviation of its (log-)likelihood function.

%\vspace{-0.1in}
\subsection{Problem formulation}
%\vspace{-0.1in}

% In this section, we propose a diagnosis procedure for goodness-of-fit for fitted Hawkes process generative models based on Generalized Score test under model misspecification. We first reformulate the one-sample goodness-of-fit problem into a two-sample test problem. 

% %\vspace{-0.1in}
% \subsection{Problem set-up}
% %\vspace{-0.1in}

Suppose we have two data sequences $D
_z = (t_1^{(z)}, \ldots t_{N_z}^{(z)})$, $(z = 1, 2)$, which represent the arrival times of a sequence of events. Here, $D_1$ is from real world and $D_2$ is generated from the fitted generative model. Assume $D_1 \sim \lambda^*$ and $D_2 \sim \lambda$, where $\lambda^*$ is the unknown true conditional intensity and $\hat \lambda$ is the fitted one. Further assume both conditional intensities take form in \eqref{eq:hawkes}. We aim to test\begin{equation*}
    H_0': \lambda^* = \hat \lambda, \quad \mbox{versus} \quad H_1': \lambda^*\neq \hat \lambda.
\end{equation*}
Note that $\lambda^*$ in the above formulation is unknown. As illustrated in Figure~\ref{fig:set_up_mismatch}, we cast the problem above into testing $H_0: \lambda^*_0 = \hat \lambda_0$ by projecting the unknown ground-truth onto a piecewise constant function family $\cG$, on which we can develop a tractable goodness-of-fit test statistic. Empirically, this projection is done by mixing $D_1$ and $D_2$ and fitting a piecewise constant triggering function to the mixed data. Most importantly, when we have several candidate models, this statistic serves as a quantitative metric to compare models.

We calculate this test statistic in the following three steps: Mix the two data sequences up to get an aggregated sequence; Estimate $\theta_0$, maximizer of the Quasi-likelihood, from a Quasi-parameter space $\Theta$ for the aggregated sequence; Compute a test statistic $\Hat{GS}_T$ based on the estimation in the last step.

\textbf{Remark 1} (Singleton null). In our setting, the triggering function's unknown parameter is infinite-dimensional, so the null hypothesis $H_0$ is an uncountable set. To make the problem tractable, we cast $H_0'$ to $H_0$ by representing the unknown triggering function using some basis function (in our case, we use indicator function on mutually disjoint intervals \eqref{piecewiseconst}) such that we reduce this into a finite-dimensional problem. Besides, testing with unknown $\lambda^*$ is impractical, and we can only handle the projected problem $H_0$ to draw the inference for $H_0'$ anyways.

\textbf{Remark 2} (Model mismatch). We use the term "Quasi" here since commonly speaking, there will be a mismatch between a machine learning algorithm class we specify and the unknown true intensity, i.e., this class is misspecified as illustrated in Figure~\ref{fig:set_up_mismatch}. We add a prefix "Quasi-" for everything under this class, e.g., Quasi-conditional intensity and Quasi-likelihood function. Since conditional intensity characterizes a point process and we assume the triggering function $\phi^* = \alpha g^*$, 
we only need to specify the approximate class for $g^*$.  We choose a piecewise constant function class as $\cG$. The reason is three-fold: (i) a piecewise constant function can approximate any integrable function arbitrarily well by reducing the size of the discretization bin; (ii) there exists $g_0 \in \cG$, which corresponds to our estimand $\theta_0$, serving as a good approximation to $g^*$ and it is identifiable; (iii) most importantly, we can develop an easy-to-calibrate hypothesis test on this family. We will elaborate on these in the next section.

%\vspace{-0.1in}
\section{Proposed goodness-of-fit test} \label{testprocedure}
%\vspace{-0.1in}

The idea behind this test comes from a critical observation that under $H_0$ (or $H_0'$), mixing two sequences will lead to a Hawkes process with scaled intensity function. Based on this observation, we can derive a Generalized Score (GS) test, which is known to be locally most powerful (Neyman–Pearson lemma).

\textbf{Step 1: Mix two data sequences and model the aggregated sequence.} 

In this step, we derive the Quasi-log-likelihood function for the aggregated sequence. The proof is deferred to Appendix \ref{mixing}.
\begin{prop}[(Log-likelihood of mixing of two Hawkes processes)]\label{sum_2_hawkes}
\it Suppose we have two Hawkes processes with conditional intensities
\[\lambda_z(t |\cH_{z,t})=\mu^{(z)} +\sum_{\{i: t_{i}^{(z)}<t\}} \phi^{(z)}(t-t_{i}^{(z)} ) \ (z=1, 2).\] Define their mixing to be $N(t)= N_1(t)+N_2(t).$ Then it has background intensity $\mu = \mu^{(1)}+\mu^{(2)}$. Denote $T = \max\{t_{N_1}^{(1)},t_{N_2}^{(2)}\}$ and $\Phi^{(z)} (t) = \int_0^t \phi^{(z)} (u)\mathrm{d}u$. Given the past trajectory: $\cH_t=\cH_{1,t}\cup  \cH_{2,t}$, where $\cH_{z,t} =\{t_1^{(z)},\dots,t_{N_z}^{(z)}\},\ z=1,2,$ we  have that:
(i) Under $H_1$, let $z' = 2 (\text{or }1) \ \text{when} \ z = 1 (\text{or }2)$, the full model log-likelihood $\ell_1(\mu, \phi^{(1)},  \phi^{(2)}|\cH_{t})$ is
%\vspace{-0.05in}
\begin{equation*}
\begin{split}
     - \mu T + \sum_{z=1}^2 & \sum_{i=1}^{N_z} \log\Big(\mu+
    \sum_{j<i}\phi^{(z)}(t_i^{(z)}-t_j^{(z)})
    \\ &+\sum_{j=1}^{N_{z'}}\phi^{(z')}(t_i^{(z)}-t_j^{(z')})
    \Big)-\Phi^{(z)} (T-t_i^{(z)}).
\end{split}
\end{equation*}
%\vspace{-0.1in}

(ii) Under $H_0: \phi^{(1)} =  \phi^{(2)} = \phi$, the sub-model log-likelihood is 
$\ell_0(\mu, \phi|\cH_{t}) = \ell_1(\mu, \phi,  \phi|\cH_{t}).$
\end{prop}
%\vspace{-0.1in}
Note that the triggering function takes value zero on $(-\infty,0]$ and thus we did not consider the triggering effect of events to its own history. By this proposition, we can model the aggregated data via a univariate Hawkes process with the same triggering function under $H_0$. For each event in process $z\ (z=1,2)$, it does not only dependent its original own history, but also depends on the history of another process $z'\ (z'=2,1)$. See an illustration of this in Figure~\ref{fig_sum_hawkes}.

\begin{figure}[htb]
  \centering
    \subfigure{\includegraphics[scale = 0.27]{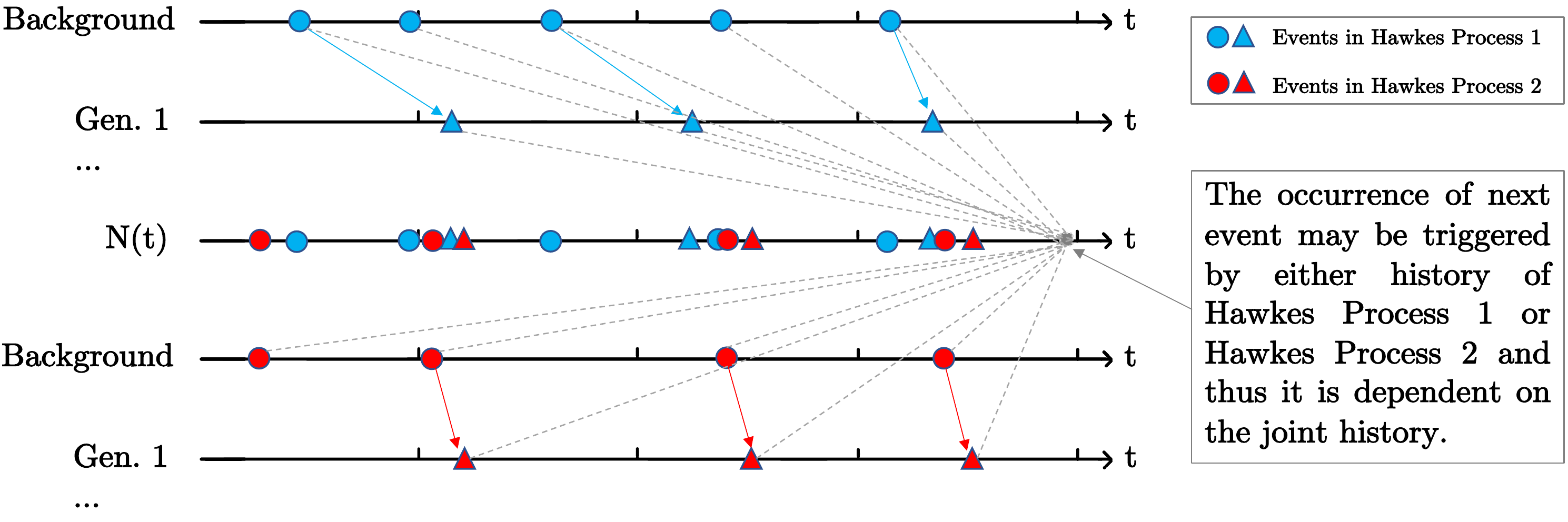}}
        \vspace{-.15in}
    \caption{Illustration of mixing of two Hawkes processes $N(t)$. Given the past sample trajectory, the upcoming event of $N(t)$ may (1) come from background poisson process of Hawkes process 1 or 2 OR (2) be a offspring of history $\cH_{1,t}$ or $\cH_{2,t}$. The grey dashed line in the figure illustrated scenario (2).}\label{fig_sum_hawkes}
\end{figure}

\textbf{Step 2: Discretize triggering function and learn quasi-conditional intensity.} 

In this step, we choose piecewise constant function as the approximation to the true triggering function for the aggregated sequence. This means we will discretize the time horizon into small intervals (which we call bins) and estimate a "weight" on each interval. In practice, the time horizon we discrete is truncated on $[0,T_0]$ and discretized into finitely many bins, since it is unnecessary to estimate infinite number of weights on infinite time horizon. More specifically, we assume $g_0(t) = \sum_{k=1}^{n_0} g_{k} \mathbf{1}_{B_{k}}(t)$ and estimate it from the following class:
%\vspace{-0.05in}
\begin{equation}\label{piecewiseconst}
 \cG \triangleq  \Big\{g(t) \  \Big|  \ 0\leq g_{k} <\infty \  \text{and} \ \sum_{k=1}^{n_0} g_{k} \Delta t_{k} = 1 \Big\}.
\end{equation}
Here, $0 = \delta t_{0} <  \delta t_{1} < \dots<\delta t_{n_0}=T_0 $ and each bin $B_{k}=\left(\delta t_{k-1}, \delta t_{k}\right]$ has length $\Delta t_{k}=\delta t_{k}-\delta t_{k-1} (k=1,2,\dots,n_0)$. 

We apply Probability Weighted Histogram Estimation \citep{marsan2008extending,fox2016spatially} to learn the weights $g_k$ on each bin $B_k$, triggering magnitude $\alpha$ and background intensity $\mu$. 
Most importantly, our Quasi-conditional intensity defined in \eqref{piecewiseconst} satisfies the model assumption in \citet{fox2016spatially}, which guarantees the non-parametric stochastic declustering algorithm as an EM algorithm. It maximizes a lower bound on the Quasi-log-likelihood function, which is in fact the complete-data Quasi-log-likelihood function derived by \citet{veen2008estimation}. Thus, it outputs the MLE of Quasi-log-likelihood function (QMLE). See Appendix \ref{PWHE} for further details.

Before moving on, we need to formally define the estimand $\theta_0$ we want to learn. It is the parameter of Quasi-conditional intensity which maximizes the expected Quasi-log-likelihood.
\begin{defi}[(Estimand)]
\it The estimand $\theta_0$ is\begin{equation}\label{def_theta0}
    \theta_0 = \arg \max_{\theta \in \Theta} \EE [\ell_1(\theta|\cH_T)],
\end{equation}where the expectation is w.r.t. all trajectories $\cH_T$ and the expression of $\ell_1(\theta|\cH_T)$ is given in Proposition~\ref{sum_2_hawkes}.
\end{defi}
%\vspace{-0.1in}
\textbf{Remark} (Information theoretic interpretation). %Here $\theta_0$ defined above is given by the maximum expected likelihood principle, which is an extension of the classic maximum likelihood principle, and has an information theoretic interpretation \cite{Akaike1998}. 
Here, $\theta \in \Theta$ in (\ref{def_theta0}) has an information theoretic interpretation \citep{Akaike1998}. It parameterizes the Quasi-conditional intensity $\lambda = \lambda_{\theta}$ and $\theta_0$ defined above corresponds to $\lambda_0 = \lambda_{\theta_0}$, which minimizes Kullback-Leibler (K-L) divergence to the unknown ground-truth $\lambda^*$:\begin{equation*}
    \theta_0 = \arg \min_{\theta \in \Theta} \EE [\ell^*-\ell_1(\theta|\cH_T)]= \arg \min_{\theta \in \Theta} KL(\lambda^*||\lambda_\theta),
\end{equation*}where $\ell^*$ is the true log-likelihood.
That's why we call $\lambda_0$ ``the best approximation to $\lambda^*$'' or ``projection onto the user-specified space'' (as illustrated in Figure~\ref{fig:set_up_mismatch}).
 
\begin{prop}[(Global identifiability)]
\it $\theta_0$ defined by \eqref{def_theta0} is globally identifiable.
\end{prop}
%\vspace{-0.1in}
We prove global identifiability by showing \eqref{def_theta0} is a (strictly) concave program. Most importantly, when the fitted model is actually the same as the unknown true one, $\theta_0$ will lie in $\Theta_0$, i.e. $H_0$ holds under $H_0'$. This justifies our projected test $H_0$, indicating that the difference between mismatched models represents the difference between true models. The detailed proof is deferred to Appendix~\ref{all_asym}. We should make a mild assumption that $\theta_0$ is interior to the convex Quasi-parameter space $\Theta$. This makes sure that we have $\nabla_{\theta} \EE [ \ell_1 (\theta_0|\cH_T)] = 0,$ which guarantees that $\theta_0$ is the estimand which our QMLE is consistent for. We will show this in detail later in the Appendix~\ref{all_asym}.

\textbf{Step 3: Compute GS statistic.} 

Here, we call the singleton that we want to test a sub-model. We call the Quasi-parameter space under $H_0$ sub-model Quasi-parameter space and denote it by $\Theta_0$. Similarly, $\Theta$ is the full model Quasi-parameter space, or rather, Quasi-parameter space under $H_1$. Under our proposed approximation class \eqref{piecewiseconst}, the Quasi-conditional intensity has a parameterization $$\theta = (\mu,\phi_1^{(1)},\dots,\phi_{n_0}^{(1)},\phi_1^{(2)},\dots,\phi_{n_0}^{(2)})^\intercal,$$ where $\mu \triangleq \mu^{(1)}+\mu^{(2)}$ and $\phi_k^{(z)} = \alpha^{(z)} g_k^{(z)}$. The full model Quasi-parameter space is given by
%\vspace{-0.05in}
\begin{equation*}
    \begin{split}
        \Theta = \Big\{\theta \ \big |\ & \mu > 0, 0\leq \alpha^{(z)} < 1, \\ & g^{(z)}=\sum_{k=1}^{n_0} g^{(z)}_{k} \mathbf{1}_{B_{k}} \in \cG \ \  (z=1,2)\Big\}\subset \RR^{d},
    \end{split}
\end{equation*}
%\vspace{-0.05in}
where $d=1+2 n_0$. Note that the second constraint guarantees the stationarity and ergodicity. We further denote
$$\phi^{(z)} = (\phi_1^{(z)},\dots,\phi_{n_0}^{(z)})^\intercal =  (\alpha^{(z)}g_1^{(z)},\dots,\alpha^{(z)}g_{n_0}^{(z)})^\intercal$$ to be the Quasi-parameter of the triggering function of Hawkes process $z \ (z=1,2)$. The sub-model Quasi-parameter space is
$$\Theta_0 = \{\theta \in \Theta \ | \   \phi_k^{(1)}-\phi_{k}^{(2)}=0, \ k=1,\dots,n_0\}\subset \RR^{1+n_0}.$$
Denote the number of constraints (we'll see later it's in fact degree-of-freedom of our test statistic) $r=n_0 = \dim \Theta - \dim \Theta_0,$ the null hypothesis $H_0: \theta_0 \in \Theta_0$ can be re-expressed as $H_0: h(\theta_0) = \phi^{(1)}-\phi^{(2)} =  0$, where $h: \RR^{d} \rightarrow \RR^{r}$. We consider a test: \[H_0:h(\theta_0)=0, \quad \mbox{versus} \quad H_1: h(\theta_0)\neq0,\] and the following test statistic:

\begin{defi}[(GS statistic)] \it Suppose the past sample trajectory is $\cH_T$. Denote 
\begin{align*}
    &S_T(\theta)=   \frac{\partial \ell_1(\theta|\cH_T)}{\partial \theta}\in  \RR^{ d},A_T(\theta)=   S_T(\theta) S_T^\intercal(\theta) \in \RR^{d\times d},
    \\&H(\theta)= \frac{\partial h(\theta)}{\partial \theta} \in \RR^{r\times d}, \ 
    B_T(\theta)= - \frac{\partial^2 \ell_1(\theta| \cH_T)}{\partial \theta \partial \theta^\intercal}  \in \RR^{d\times d},
\end{align*}where $H$ exists and has full row rank $r$, and log-likelihood $\ell_1$ is given in Proposition~\ref{sum_2_hawkes}. Then, the Generalized Score (GS) test statistic is given by
\begin{align*}
     \Hat{GS}_T &= S_T^\intercal(\Hat{\theta}_{QMLE}) \Hat{\Sigma}^{-1} S_T(\Hat{\theta}_{QMLE}),
\end{align*}where $\Hat{\theta}_{QMLE} \in \Theta_0$ is QMLE under null hypothesis and $ \Hat{\Sigma}^{-1}$ is given by: \begin{equation*}
    \begin{split}
        \Hat{\Sigma}^{-1} =& B_T^{-1}(\theta)  H(\theta)^\intercal \Big(H(\theta)  B_T^{-1}(\theta) \\
        &A_T(\theta)B_T^{-1}(\theta)H(\theta)^\intercal\Big)^{-1} H(\theta)B_T^{-1}(\theta)\bigg|_{\theta=\Hat{\theta}_{QMLE}}.
    \end{split}
\end{equation*}
\end{defi}
%\vspace{-0.1in}
Later, we will show $T \Hat{\Sigma}^{-1}$ is a consistent estimator of inverse of covariance matrix of $S_T(\Hat{\theta}_{QMLE})/\sqrt T$. Closed-form expression for $\Hat{GS}_T$ is given in Appendix \ref{explicitGS}.

% \subsection{A Non-Parametric Goodness-of-Fit Test for Generative Self-exciting Point Process Model}

%\textcolor{blue}{Woody: Do we really need this? We first re-state the testing procedure in Section \ref{testprocedure} for : (i) ; (ii)  and ; (iii) . For simplicity, we will call this test statistic score in the following. This procedure yields  test statistics.} 

Based on our testing procedure for two single data sequences above (steps $1\sim 3$), we state a more general version for two sets of data sequences in Algorithm~\ref{algo_goodness}.

%\vspace{-0.05in}

\begin{algorithm}[htp]
    \caption{Non-parametric goodness-of-fit test for self-exciting point processes}\label{algo_goodness}
    {\bf Input:} Two set of i.i.d. data sequences $D_1 = \{D_{1,1},\dots,D_{1,L}\}$ and $D_2 = \{D_{2,1},\dots,D_{2,L}\}$. \\
    {\bf Initialization:} $n_0$ bins on time horizon $[0,T_0]$; repeat times $K$; number of sequences $N$ to calculate one GS statistic $\Hat{GS}_T$. \\
    {\bf Output:} $K$ i.i.d. GS statistics .
    \begin{steps}
        \item Mix $D_{1,i}$ and $D_{2,i}$ to get the aggregated sequence $D_i^{agg}$ ($i=1,\dots,L$). %as in Step I.            
        \item Apply Probability Weighted Histogram Estimation to learn QMLE. %as in Step II.    
        \item Repeat the procedure for $K$ times: randomly shuffle the order of sequences in the $D_1$ and repeat step I to get a different set of aggregated sequences, from which we randomly choose $N$ sequences to calculate one $\Hat{GS}_T$. %as in Step III.

    \end{steps}
\end{algorithm}

%\vspace{-0.05in}

The stationarity of a stochastic process means the unconditional probability distribution does not change when shifted in time. More specifically, for a stochastic process $N(t)$, for all $t \in \RR$, $N(t,t+\delta]$ follows a same probability distribution as long as $\delta > 0$ is fixed. Thus, when $T \rightarrow \infty$, we will have $$\frac{\EE[\ell_1(\theta|\cH_T)]}{T \  \EE[\ell_1(\theta|\cH_1)]} \rightarrow 1.$$ This shows that the estimand defined by maximum expected log-likelihood principle will not vary with different time horizon $T$ (otherwise, $\theta_0$ is not well-defined). Most importantly, this also shows that learning with $L$ short sequences on time horizon $[0,T_0]$ is equivalent to learning with one long sequence on time horizon $[0,LT_0]$, which justifies our generalization to the testing on two sets of data sequences in Algorithm~\ref{algo_goodness}.

% %\textbf{computation cost?}

%\vspace{-0.1in}
\section{Theoretical Analysis}\label{theoreticalresult}
%\vspace{-0.1in}

Here, we will prove the asymptotic performance of our GS statistics by establishing a novel connection with classic results in statistics for QMLE and the GS test based on it \citep{white1982maximum}. We provide a generalization of the asymptotic properties of MLE for Hawkes process \citep{ogata1978asymptotic} to model mismatch case, based on which we get the asymptotic behaviors of testing procedure such as score test and Wald test. The proofs and numerical illustration on why we choose score test over Wald test are deferred to Appendices~\ref{all_asym} and \ref{add_exp}.

We use $\theta_0$ to denote the projection of ground-truth and test $H_0: \theta_0 \in \Theta_0$ against $H_1: \theta_0 \not \in \Theta_0$. Apparently, under different hypothesis, $\theta_0$ cannot be the same. To avoid confusion, we say the projection is $\theta_0 = \theta_1 \in \Theta_0$ under $H_0$ and $\theta_0 = \theta_2 \not\in \Theta_0$ under $H_1$.

\begin{lem}[(Asymptotic properties of Quasi-MLE)]\label{asym_qmle}\it Let $\Hat{\theta}_{QMLE}$ and $\Tilde{\theta}_{QMLE}$ be QMLE under $H_0$ and $H_1$.
For piecewise constant triggering function family \eqref{piecewiseconst}, 
% denote $\Bar{\theta}_{QMLE}$ to be $\Hat{\theta}_{QMLE}$ under $H_0$ and $\Tilde{\theta}_{QMLE}$ under $H_1$, then 
QMLE satisfies the following asymptotic properties: 

(i) Convergence to $\theta_0$ almost surely. When $T \overset{}{\to} \infty$,
%\vspace{-0.05in}
\begin{align*}
    \text{under $H_0$: }  \ \Hat{\theta}_{QMLE} \overset{a.s.}{\to} \theta_1;  \  \ \text{under $H_1$: } \ \Tilde{\theta}_{QMLE} \overset{a.s.}{\to} \theta_2 ;
\end{align*}
(ii) Asymptotic normality. Define 
$A(\theta)=\EE \left[ A_T(\theta)\right]/T$ and $B(\theta)=\EE \left[ B_T(\theta)\right]/T,$ when $T \overset{}{\to} \infty$, we will have:
\begin{align*}
    \text{Under $H_0$: } \ \ \sqrt{T} (\Hat{\theta}_{QMLE} - \theta_1) \overset{d}{\to} N (0,\Sigma^{-1}(\theta_1));\\
    \text{Under $H_1$: }  \ \ \sqrt{T} (\Tilde{\theta}_{QMLE} - \theta_2) \overset{d}{\to} N (0,\Sigma^{-1}(\theta_2)),
\end{align*} where $\Sigma^{-1}(\theta) = B^{-1}(\theta)A(\theta)B^{-1}(\theta)$.

(iii) We also have asymptotically normality of the Quasi-score function, no matter under $H_0$ or $H_1$:$$\frac{1}{\sqrt{T}} \frac{\partial  \ell_1 (\theta|\cH_T)}{\partial \theta}\bigg|_{\theta = \theta_0} \overset{d}{\to} N (0,A(\theta_0)) \ \ \ \text{  as  } \ \ \ T \overset{}{\to} \infty.$$
\end{lem}
%\vspace{-0.1in}
\textbf{Remark.}
The score function should have the Fisher Information Matrix (FIM) $I(\theta^*)$ as its asymptotic covariance matrix when the model is correct. Using FIM will break the asymptotic $\chi^2$ distribution in the model mismatch case. That's why we need to consider the model mismatch explicitly. Even though we cannot correctly specify the function family for unknown ground-truth, using $A(\theta_0)$ instead of FIM as the covariance matrix will still yield correct asymptotics for our proposed test. Moreover, by Theorem 1 in \citet{ogata1978asymptotic}, one can verify that Information Matrix Equivalence Theorem in \citet{white1982maximum} still holds for stationary point process, i.e. $\theta_0 = \theta^*$ and $A(\theta_0) = B(\theta_0) = I(\theta_0)$ hold if and only if the model is correctly specified. Thus, our results simplify to the form in \citet{ogata1978asymptotic} in the absence of model mismatch. Though the asymptotic covariance matrix of QMLE is no longer inverse of the FIM $I^{-1}(\theta^*)$, we can still estimate it consistently. 
% Besides, our results can also be viewed as an extension of \cite{white1982maximum} to a stationary point process case.

\begin{thm}[(Asymptotic null distribution of $\Hat{GS}_T$)]\label{asym_null}
\it Under $H_0$, the Generalized Score (GS) test statistic has an asymptotic $\chi^2$ distribution. More specifically,$$\Hat{GS}_T \overset{d}{\to} \chi^2_r  \ \ \ \text{  as  } \ \ \ T \overset{}{\to} \infty.$$
\end{thm} 

%\vspace{-0.1in}
Note that here the degree of freedom is $r=n_0$, which is exactly the number of bins we discretize $[0,T_0]$ into. 

\begin{thm}[(Power function of GS test)]\label{asym_power}
\it Under $H_1$, the GS statistic follows a asymptotic noncentral $\chi^2$ distribution with degree of freedom $r$ and noncentrality parameter $T \norm{\phi^{(1)}-\phi^{(2)}}_2^2$. For any critical value $c>0$, when $T\rightarrow \infty$, the test power is:
%\vspace{-0.05in}
$$\frac{\PP_{H_1} (\Hat{GS}_T > c)}  {Q_{r/2}(\sqrt{T }\norm{\phi^{(1)}-\phi^{(2)}}_2,\sqrt{c})} \rightarrow 1,$$
where $\|\cdot\|_2$ is the vector $\ell_2$ norm and $Q_{M}(a,b)$ is the Marcum-Q-function. 
\end{thm}

% \begin{wrapfigure}{r}{0.25\textwidth}
% %\centering
% %\vspace{-0.25in}
% \subfigure{\includegraphics[scale=0.2]{power.pdf}}
% %\vspace{-0.3in}
% \caption{Illustration of asymptotic power of GS test.}\label{fig_power}
% %\vspace{-0.15in}
% \end{wrapfigure}

\begin{figure}[!htp]
%%\vspace{-0.1in}
\centering
\includegraphics[scale=0.35]{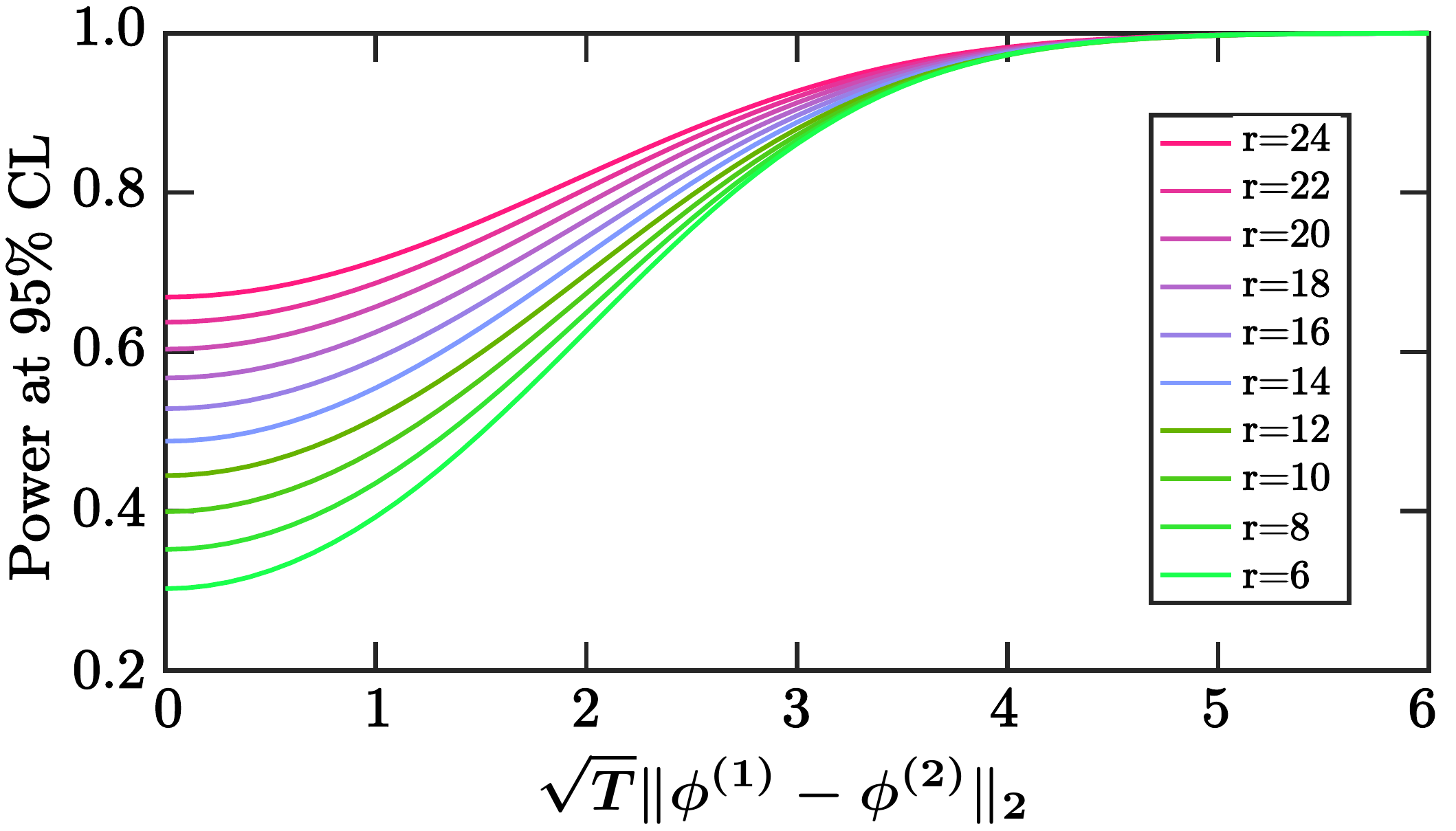}
 \vspace{-.1in}
\caption{Illustration of asymptotic power of GS test.}\label{fig_power}
%\vspace{-0.2in}
\end{figure}
%\vspace{-0.1in}
The asymptotic power function with the critical value chosen to be the upper 95\% quantile of the null distribution is shown in Figure~\ref{fig_power}. $Q_{M}(a,b) \rightarrow 1$ as $a \rightarrow \infty$, indicating our proposed test is consistent. See Appendix \ref{specialfunc} for more on $Q_{M}(a,b)$.

% \begin{figure}[!htp]
%   \centering
% \subfigure{\includegraphics[scale=0.3]{power.pdf}}
   
%      \caption{Asymptotic power of GS test when critical value is chosen to be the upper 95\% quantile of the null distribution $\Delta = \sqrt{T }\norm{\phi^{(1)}-\phi^{(2)}}_2$.}\label{fig_power}
% \end{figure}

% \begin{figure}[H]
%   \centering
% \subfigure{\includegraphics[scale=0.33]{power.pdf}}
   
%     \caption{Asymptotic power of GS test when critical value is chosen to be the upper 95\% quantile of the null distribution $\Delta = \sqrt{T }\norm{\phi^{(1)}-\phi^{(2)}}_2$.}\label{fig_power}
% \end{figure}
%\end{multicols}
%\textbf{Remark.} 
%Based on the asymptotic power function, we can expand the null hypothesis from a singleton to $H_0': \norm{\phi^{(1)}-\phi^{(2)}}_2 < \epsilon$. The $p-$value for this test is $Q_{r/2}(\sqrt{{T} }\epsilon,\hat{GS}_T^{1/2})$ and reject $H_0'$ when $p-$value is smaller than $\alpha$ will lead to a test with confidence $1-\alpha$. 

%\vspace{-0.1in}
\section{Numerical experiments}\label{experiment}
%\vspace{-0.1in}

In this section, we present numerical simulation to (1) validate the asymptotic property of our method by three simulation experiments; (2) demonstrate the GOF test for synthetic and real data.

\subsection{Validation of asymptotic properties}\label{expsec:validation}
%\vspace{-0.1in}
To validate Theorems \ref{asym_null} and \ref{asym_power} presented in Section~\ref{theoreticalresult}, we conduct three simulation experiments on a synthetic data set. We repeat our experiments on five sub-data sets generated from Hawkes process defined in \eqref{eq:hawkes} with 1,000 sequences, where $\mu = 20$ and an exponential triggering function $\phi(t - t_i) =  \alpha e^{- 10 (t - t_i)}, t_i < t$ is adopted; $\alpha$ in each sub-data set is from $\{1.25,1.5,1.75, \dots, 3.75\}$.

\begin{figure}[!htp]
%\vspace{-0.1in}
\centering
\includegraphics[scale=0.3]{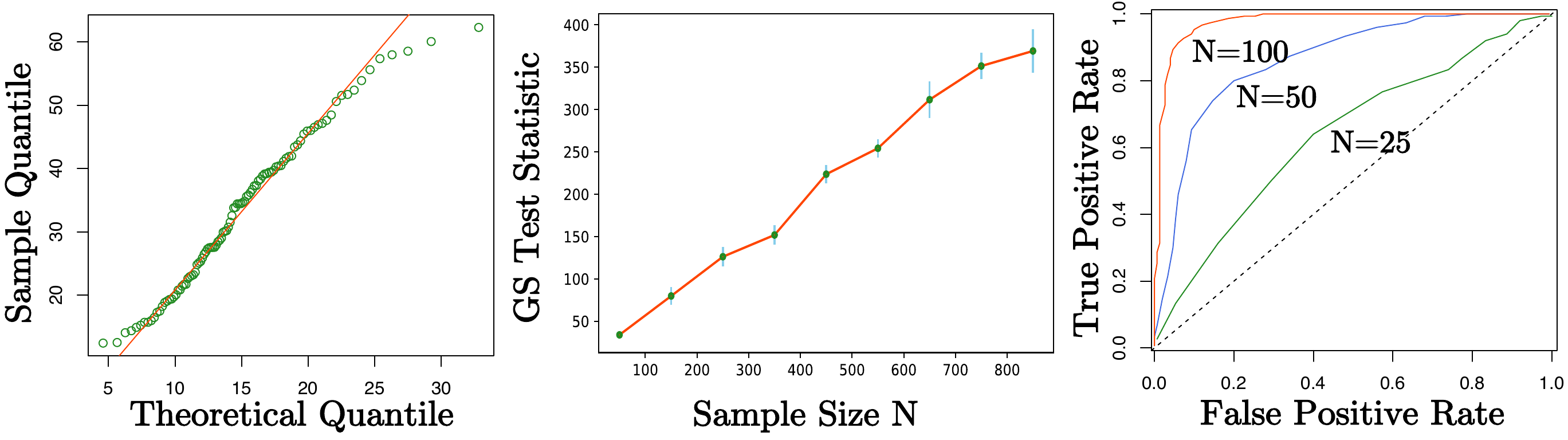}
     \vspace{-.2in}
    \caption{Simulation results: Left (a): quantiles of calculated GS statistics against theoretical quantiles of $\chi^2_{n_0}$ distribution under $H_0$; Middle (b): mean and variance of GS statistics with increasing $N$ under $H_1$; Right (c): ROC curve for different $N$.}\label{exp1}
%\vspace{-0.1in}
\end{figure}

The Q-Q plot in Figure~\ref{exp1} (a) shows that the GS statistic follows the $\chi^2$ distribution, which is consisent with Theorem \ref{asym_null};
Figure~\ref{exp1} (b) visualizes the mean (red line) and the error bar (green bars) of each testing point for the GS statistics over different sample size $N$. Clearly, the GS statistics tend to be linear in sample size under $H_1$, which matches the theoretical results shown in our power study in Theorem \ref{asym_power} and shows that our asymptotic distribution analysis is reasonably accurate. The ROC Curve in Figure~\ref{exp1} (c) %where the true positive means we correctly accept the null hypothesis $H_0$ (identify two sequences are from the same distribution). The false-positive means we mistakenly accept the null hypothesis $H_0$ (identify two sequences are from the same distribution which they are not). We plot the same ROC plots with a different number of samples $N$ used in calculating the GS statistics. 
shows that the GS statistics has good performance when $N = 100$ (AUC is approximately 1); 
%The testing procedure follows Section~\ref{testprocedure}. 
We choose $K$ to be $20, 5, 150$ for three experiments, respectively. Details on testing procedure can be found in Appendix \ref{add_exp}.

In short, we have confirmed 
(a) the $\chi^2$ null distribution; 
(b) the score is linear in sample size under $H_1$;
(c) the consistency of the proposed test.
We also conduct similar experiments for power triggering functions to validate our method is model free. Results are deferred to Figure~\ref{exp1_2} in Appendix~\ref{add_exp} due to space limitation. 

%\vspace{-.15in}
\subsection{Effects of number of Bins $n_0$}\label{expsec:diff_n0} 
%\vspace{-.1in}
We use exponential synthetic data sequence $D_1$ and $D_2$ with $\mu_1=\mu_2=20$, $\beta_1 = \beta_2 = 10$, $\alpha_1 = \alpha_2 = 1.5$ under $H_0$ and $\alpha_1 = 1.5, \alpha_2 = 5$ under $H_1$. The histogram estimate under $H_0$ is given in Figure~\ref{exp1_diffn0}. We perform GS test (confidence level $95\%$) under $H_0$ and $H_1$ 100 times for each $n_0$ and report Type I \& II errors in Table~\ref{table:exp1_diffn0}.

\begin{figure}[!htp]
  \centering
  %\vspace{-.15in}
  \subfigure{\includegraphics[scale=0.25]{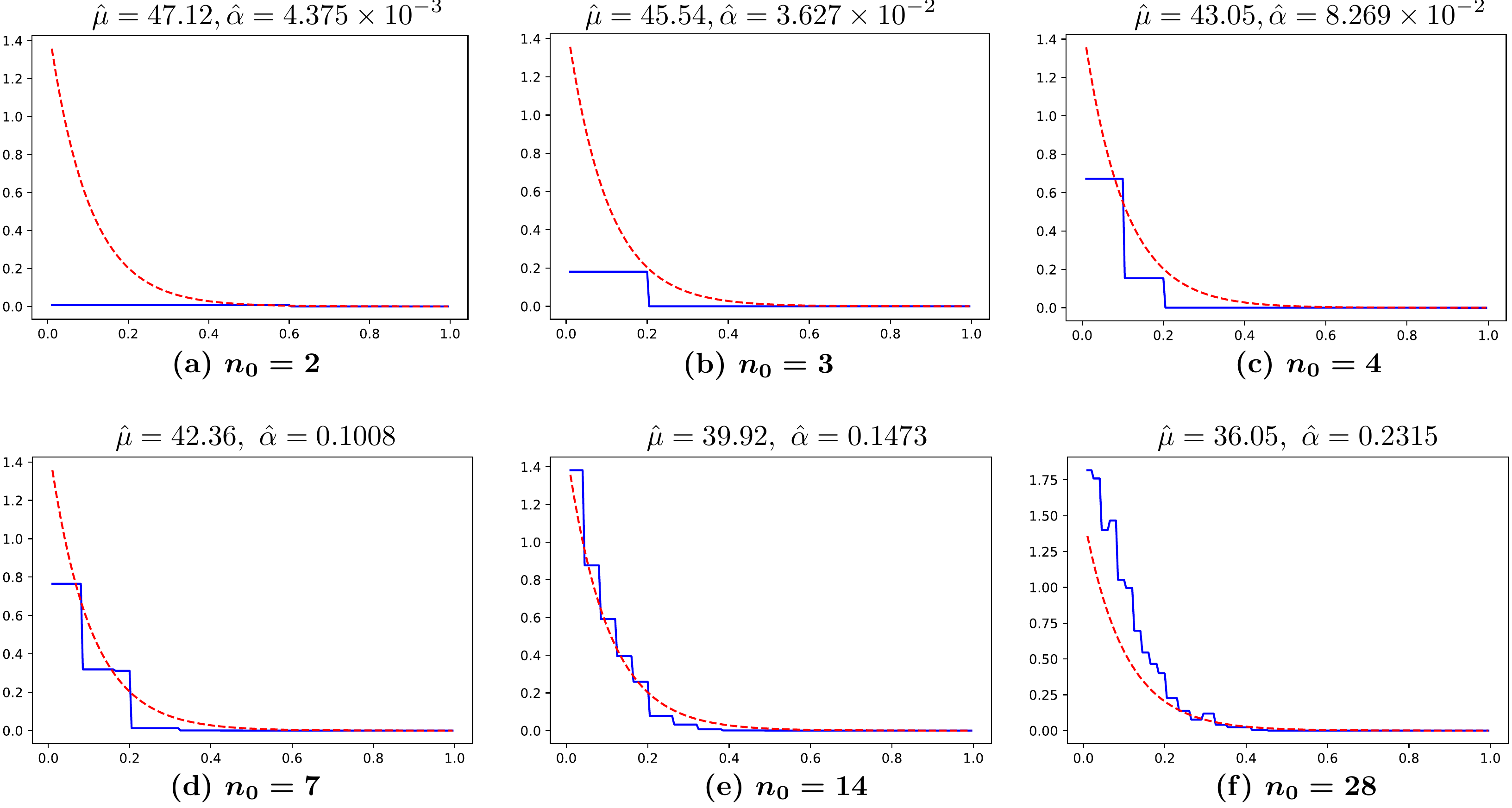}}  
   \vspace{-.2in}
\caption{Histogram estimation of exponential kernel with $\mu = \mu_1+\mu_2 = 40$ and $\alpha=0.15$ with different $n_0$. The red dashed line is ground-truth $\alpha e^{-\beta t}$ and the blue solid line is the histogram estimate. The bottom middle panel ($n_0=14$) is the most accurate one.}\label{exp1_diffn0}
%\vspace{-.1in}
\end{figure}

From Figure~\ref{exp1_diffn0}, we can observe that with too many bins, the histogram will overfit the data (panel (f)), whereas with fewer bins it underfits (panels (a)$\sim$(d)). However,  Table~\ref{table:exp1_diffn0} shows that $n_0=3$ is most powerful in capturing the difference in triggering function. By comparing panel (a) and (b) in Figure~\ref{exp1_diffn0}, even though underfitting still exists, it captures the triggering function, which seems to be sufficient for our setting.

\begin{table}[!htp]
\vspace{-0.1in}
\caption{Empirical Type I and Type II error (over 100 trials) for different number of Bins.}\label{table:exp1_diffn0}
 \vspace{-.1in}
\begin{center}
\begin{small}
\begin{sc}
\resizebox{0.5\textwidth}{!}{%
\begin{tabular}{lcccccccccr}
\toprule[1pt]\midrule[0.3pt]
Number of Bins & 2 & 3 & 4 & 7 & 14  & 28  \\ % Column names row
\cmidrule(l){2-7}
 % In-table horizontal line
Type I error   &    0.04  &  0.05  &  0.06  &  0.02  &  0.03   & 0.02\\ 
Type II error &    0.59  &  \textbf{0.09}  &  0.19  &  0.34  &  0.79  &  0.83\\ 
\midrule[0.3pt]\bottomrule[1pt]
\end{tabular}
}
\end{sc}
\end{small}
\end{center}
%\vspace{-0.25in}
\end{table}

% Discussion via comprehensive experiments and the rule of thumb on how to choose bins can be found in the papers that brought up histogram estimation \cite{marsan2008extending,fox2016spatially}.}

% \begin{figure}[!htp]
%   \centering
% \subfigure{\includegraphics[scale=0.45]{epx1_1.pdf}}
%       % %\vspace{-.15in}
%     \caption{Simulation results: (a) Quantiles of calculated GS statistics against theoretical quantiles of $\chi^2_{n_0+1}$ distribution under $H_0$; (b) mean and variance of GS statistics with increasing $N$ under $H_1$; (c) ROC curve for different $N$.}\label{exp1}
%   % %\vspace{-.15in}
% \end{figure}

% FPR = FP/(FP+TN), TPR = TP/(TP+FN), where TP, FN, TN, FP are number of correctly rejected $H_0$, wrongly accepted $H_0$, wrongly rejected $H_0$, correctly accepted $H_0$ for a given critical value. 

\vspace{-.05in}
\subsection{Comparison with existing methods}\label{expsec:comparison}
%\vspace{-.1in}
The basic idea of existing GOF test due to \citet{ogata1988statistical} is to (i) transform the original process to a residual process by keeping point $t_i$ with probability $\hat \mu/\hat \lambda(t_i|\cH_{t_i})$; (ii) test if the residual is a homogeneous Poisson process with rate $\hat \mu$. Commonly used homogeneity test statistic is Ripley’s $K$ function \citep{ripley_1976} and we use $\hat K(t) = \sum_{i=1}^N \sum_{j\not =i} \mathbf{1}_{\{|t_j - t_i|\leq t\}}/\hat \mu N$ as its estimate.

We apply both tests to exponential synthetic data with $\beta=10$. We still use histogram estimation to estimate the conditional intensity. We calculate the GS statistics with $N=50, K=100$ and the average of $\hat K(t)$ over $L = 100$ sequences for time span $t \in \{1,\dots,10\}$ but only report $t=1,10$ cases since the difference is not large when $t$ doesn't change a lot. The rest is plotted in Figure~\ref{exp1_comparison} in Appendix~\ref{add_exp} due to space limitation.

\begin{figure}[!htp]
  \centering
  %\vspace{-.1in}
  \subfigure{\includegraphics[width=\linewidth]{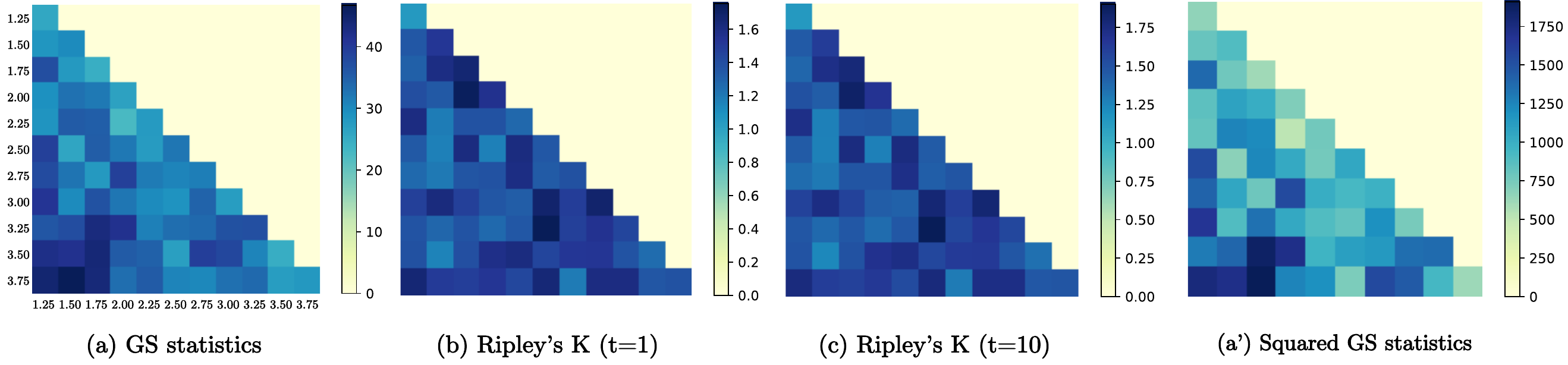}}
  \vspace{-.2in}
\caption{Heat map of (a) GS statistics, (b) $\hat K(1)$, (c) $\hat K(10)$ and (a') Squared GS statistics. For each pixel, the data sequence $D_1$ and $D_2$ are exponential synthetic data with $\alpha_1$ and $\alpha_2$ specified by the $x$-axis and the $y$-axis in (a). Squared GS statistics makes the gradual changing pattern more obvious.}\label{exp1_3}
%\vspace{-.1in}
\end{figure}

Figure~\ref{exp1_3} visualizes the GS statistics and $\hat K(t)$ when $D_1$ and $D_2$ are generated according to different $\alpha$'s, and show our method has more power in detecting the subtle difference in triggering part over existing methods. This is evident as in (a), the colors of the diagonal pixels are lighter whereas the colors of pixels on the bottom left are darker. This gradual changing pattern shows that GS statistic is larger when two generating distributions (i.e. $\alpha$'s) are further away whereas is smaller when those two distributions are closer, i.e. our proposed test can detect the subtle difference in triggering function accurately. However, in (b) and (c) we do not observe this gradual changing pattern, indicating Ripley’s K function values are approximately the same when the true data generation mechanisms of two data sequences vary within a small set. This is because background intensity dominates the conditional intensity and most of the events comes from the background. Thus, testing of whole intensity will fail to detect the subtle triggering function difference.

%\vspace{-0.1in}
\subsection{Demonstration for model comparison}\label{expsec:GOF}
%\vspace{-0.1in}

% \noindent{\bf Generative model comparison.}
We perform our proposed test procedure on various synthetic and real data sets to compare four commonly used models. For synthetic experiments, we generate $5,000$ sequences for each data sets, which come from the Hawkes process ($\mu=10$) defined in \eqref{eq:hawkes} with different types of triggering functions: 
% Each data set contains $L=5,000$ sequences. 
% Specifically, the triggering functions we use are 
% (1) exponential (\texttt{Exp}): 
% $\mu=10,\alpha=1,\beta=3$; 
(a) exponential (\texttt{Exp}): $\phi(t - t_i) = e^{-3(t-t_i)}$;
(b) Matern kernel (\texttt{Matern}): 
$\phi(t - t_i) = 0.2 \times C_{0.2, 2}(t-t_i)$, where $C_{\rho, \nu}(d)=\sigma^{2} (2^{1-\nu}) /\Gamma(\nu) (\sqrt{2 \nu} d / \rho)^{\nu} K_{\nu}(\sqrt{2 \nu} d / \rho)$, where $\Gamma(\cdot)$ is the gamma function, $K_{\nu }(\cdot)$ is the modified Bessel function of the second kind.
For real data experiments, we select a wide range of real data sets including:
(c) MIMIC-III \citep{mimiciii} (\texttt{MIMIC}): 2,246 sequences with average sequence length 4.09;
(d) MemeTracker \citep{Leskovec_2009} (\texttt{MEME}): randomly-picked 5,000 sequences with average sequence length 24.41. 
There are 2,500 sequences in (a), (b), (d), and 1,746 sequences in (c) are used for fitting the model and generating new sample sequences. The rest serves as testing data to calculate our GS statistics.

The models we are testing/comparing include 
(1) exponential triggering function fitted by gradient descent (\texttt{Exp GD}); 
(2) histogram estimation of triggering function fitted by EM algorithm (\texttt{Hist EM}) \citep{marsan2008extending,fox2016spatially};
(3) Long Short Term Memory (\texttt{LSTM}) \citep{doi:10.1162/neco.1997.9.8.1735}; 
(4) Neural Hawkes Process (\texttt{NHP}) \citep{mei2017neural};
% , where its conditional intensity takes form $\lambda(t)=\text{softplus}\left(\boldsymbol{\omega}^{\top} \boldsymbol{h}_{t}\right)$;
(5) Homogeneous Poisson process with random average intensity (\texttt{Random}) as sanity check.

\begin{table}[!htp]
\vspace{-0.1in}
\caption{GS statistic and Log-Likelihood;  lower GS value is better, higher likelihood is better.}\label{exp2}
\vspace{-0.1in}
\begin{center}
\begin{small}
\begin{sc}
\resizebox{0.5\textwidth}{!}{%
\begin{tabular}{lcccccccccr}
\toprule[1pt]\midrule[0.3pt]
& \multicolumn{5}{c}{\textbf{GS statistic}}& \multicolumn{3}{c}{\textbf{Log-Likelihood}} \\ 
Data & Exp GD & Hist EM & LSTM & NHP & Random  & Exp GD & Hist EM & NHP \\ % Column names row
\cmidrule(l){2-6}
\cmidrule(l){7-9} % In-table horizontal line
Exp   &  18.25 &\textbf{11.63}& 88.54& 14.83& 31.78 &   \textbf{21.27} & 21.10  & 20.03\\ 
Matern &  21.01& \textbf{18.40}& 81.37& 21.86& 26.11  & 19.09 &  \textbf{19.49}   & 14.91\\ 
MIMIC &  29.52& 27.90& 41.34& \textbf{25.24}& 31.04  &  \textbf{10.46} & 8.605  & 8.973 \\ 
MEME & 36.92& 34.29& 56.04& \textbf{29.98}& 39.37 &  69.51 & 62.66  & \textbf{73.15} \\ 
\midrule[0.3pt]\bottomrule[1pt]
\end{tabular}
}
\end{sc}
\end{small}
\end{center}
%\vspace{-0.15in}
\end{table}

We follow the exact testing procedure in Algorithm~\ref{algo_goodness} with $N=200$, $K = 5$; we choose $n_0 = 15$ for \texttt{Exp} and \texttt{Matern} data and $n_0 = 13$ for \texttt{MIMIC} and \texttt{MEME} data. We report the mean of scores and the likelihood of fitting the model in Table \ref{exp2}. 
We observe that our proposed GOF test can differentiate models under different settings. In particular, the GS statistics can be used as a ranking criterion.
% some structural knowledge of parametric models such as \texttt{Exp GD} and \texttt{Hist EM} are consistent with the synthetic Hawkes process (e.g., the additivity in triggering effects)
More specifically, 
the parametric models \texttt{Exp GD} and \texttt{Hist EM} achieve lower scores (better performance) on synthetic data sets comparing to \texttt{NHP} and \texttt{LSTM}, since the parametric assumptions of the parametric models (e.g., the additivity in triggering effects) are consistent with the Hawkes process used in generating synthetic data. 
In the contrast, \texttt{NHP} performs better on real data sets, including \texttt{MIMIC} and \texttt{MEME}, where dynamics between events are more complex and difficult to be captured using parametric models.
We also present the corresponding likelihood in Table~\ref{exp2}, which is commonly used to measure how well the data are fitted by the model (higher likelihood the better data is fitted). It shows that the likelihood result generally agrees with our GS statistics. Moreover, we also show that as a deterministic time series model, \texttt{LSTM} is difficult to compete with other baselines.

We should mention \texttt{Exp} data and \texttt{Exp GD} method case in particular, where the model is correctly specified. We use GD to maximize the likelihood to obtain MLE of the parameters. We observe that the estimates are further away from ground-truth while the likelihood keeps growing larger (see Figure~\ref{fig:overfitting} in Appendix~\ref{add_exp}). This means overfitting occurs and therefore likelihood may be a questionable model comparison metric.

We next show that our proposed test can select the best model. We use the ground-truth to generate the "fitted" sequence, since it is hard to learn the parameters correctly (potentially due to the overly short sequences), and compare it with \texttt{Hist EM}. We adopt the same experimental setting with the first row in Table 2 (\texttt{Exp} data) and report the result in Table~\ref{wrap-tab:1}.

\begin{table}[!htp]
\vspace{-0.1in}
\caption{Comparison of ground truth and \texttt{Hist EM} on \texttt{Exp} data. The GS statistic of \texttt{Hist EM} is different from that in Table~\ref{exp2} since we use different synthetic data.}\label{wrap-tab:1}
\vspace{-0.2in}
\begin{center}
\begin{small}
\begin{sc}
\resizebox{0.5\textwidth}{!}{%
\begin{tabular}{ccc}\\
\begin{small}
\begin{sc}
\resizebox{0.4\textwidth}{!}{%
\begin{tabular}{lcccccccccr}
\toprule[1pt]\midrule[0.3pt]
Method & GS statistic & log-likelihood \\ % Column names row
\cmidrule(l){2-3}
 % In-table horizontal line
Ground truth  &    \textbf{13.24}  & 21.64 \\ 
\texttt{Hist EM} &    17.38  &  \textbf{21.65}  \\ 
\midrule[0.3pt]\bottomrule[1pt]
\end{tabular}
}
\end{sc}
\end{small}
\end{tabular}
}
\end{sc}
\end{small}
\end{center}
%\vspace{-0.1in}
\end{table}
% \begin{wraptable}{r}{0.3\textwidth}
% %\vspace{-0.3in}
% \caption{Comparison of ground truth and \texttt{Hist EM} on \texttt{Exp} data. The GS statistic of \texttt{Hist EM} is different from that in Table~\ref{exp2} since we resample the synthetic data.}\label{wrap-tab:1}
% %\vspace{-0.1in}
% \begin{tabular}{ccc}\\
% \begin{small}
% \begin{sc}
% \resizebox{0.26\textwidth}{!}{%
% \begin{tabular}{lcccccccccr}
% \toprule[1pt]\midrule[0.3pt]
% Method & GS statistic & log-likelihood \\ % Column names row
% \cmidrule(l){2-3}
%  % In-table horizontal line
% True  &    \textbf{13.24}  & 21.64 \\ 
% \texttt{Hist EM} &    17.38  &  \textbf{21.65}  \\ 
% \midrule[0.3pt]\bottomrule[1pt]
% \end{tabular}
% }
% \end{sc}
% \end{small}
% \end{tabular}
% %\vspace{-0.1in}
% \end{wraptable} 

From this table, we can see that log-likelihood cannot differentiate those two methods and is even misleading, whereas our proposed GS statistic suggests the ground truth is a lot better than the \texttt{Hist EM} method. Together with the numerical results in the past experiments, we demonstrate that our proposed GOF test can select the best model in the sense that how well the model captures the self-exciting part in the data.

\paragraph{Goodness-of-fit for 911 call data.}
\label{police data}

%\vspace{-0.1in}

% The experiment configurations are as follows: $N=20$, $K = 1$, $n_0 = 12$ endpoints for bins are $(0, .02, .04, .06, .08, .1, .12, .14, .16, .18, .2, .5, 1)$

To demonstrate the use of our test statistic as a diagnosis tool for the GOF of generative models, we test on 911 call data in 2017 provided by the Atlanta Police. The Atlanta Police Department divides its operation region into 78 beats, so we use this to partition the spatial region and consider a non-homogeneous point process generates sequences in each beat.

We first consider police events data in each beats in one day as a sequence, and for each beat fit generative model using \texttt{NHP} and \texttt{Exp GD}. 
Then we calculate the value of the test statistic for each beat. The experiment configurations are as follows: $N=20$, $K = 1$, $n_0 = 12$.
The results are presented in Figure~\ref{police}.

%\begin{wrapfigure}{r}{0.6\textwidth}
\begin{figure}[H]
%\vspace{-0.15in}
\centering
\includegraphics[scale=0.4]{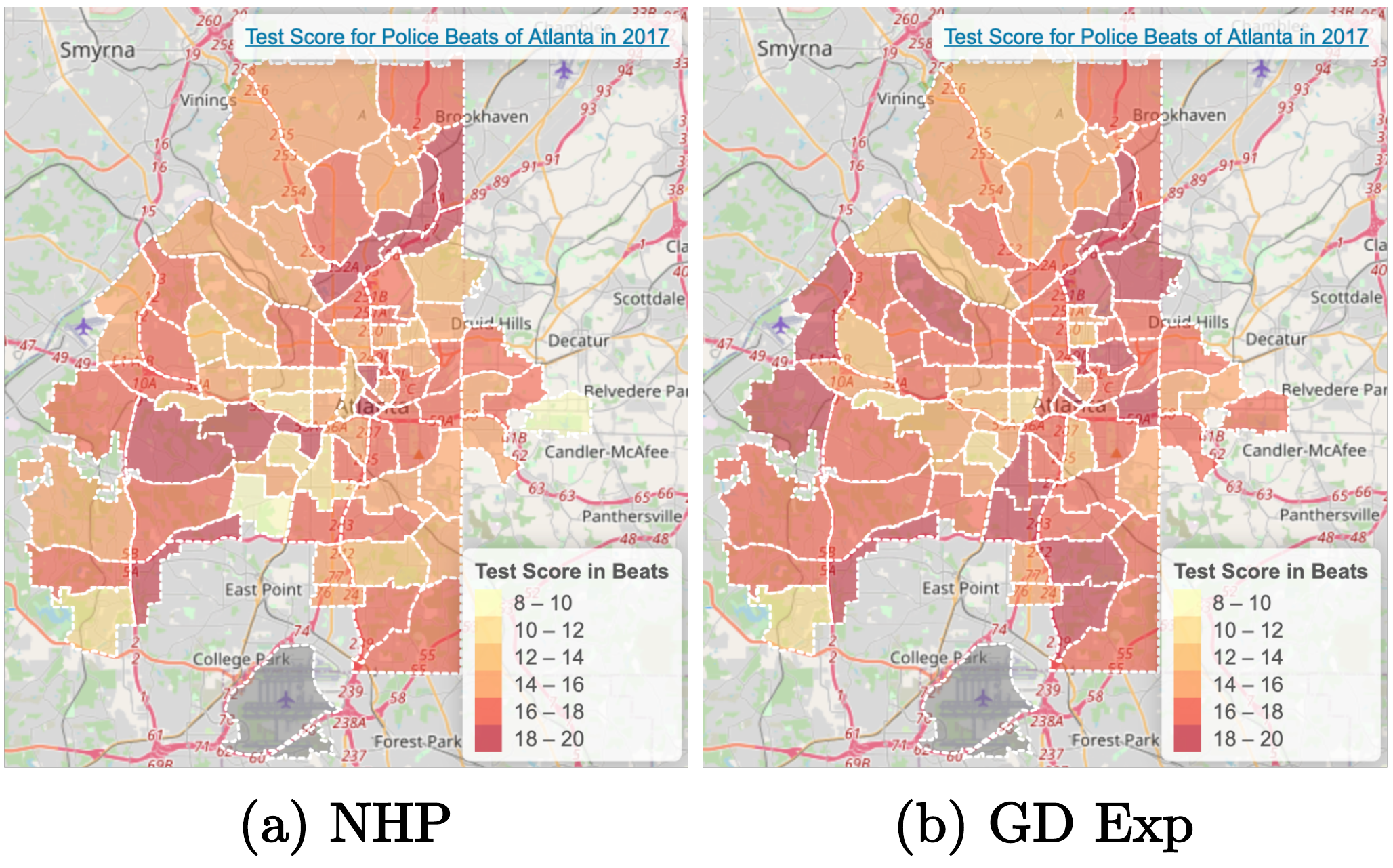}
 \vspace{-.15in}
\caption{Goodness-of-fit test for Atlanta 911 call data: (a) for \texttt{NHP}; (b) for \texttt{Exp GD}. Each polygon in the map represents a police beat in Atlanta. The color depth represents the level of the test score. Lighter color: smaller discrepancy between the generated data and the real data. Overall speaking, we can see \texttt{NHP} has better GOF than \texttt{Exp GD}, especially in populated area. }\label{police}
%\vspace{-0.1in}
\end{figure}
 %\end{wrapfigure}

Clearly, the generative model has different GOF in each beat. Also, the two generative models have different patterns in their GOF over space. Note that we do not know the ground-truth. This example demonstrates that our tools provide a convenient and flexible diagnosis tool for the GOF for generative models in practice.

% \section{Summary}

% In conclusion, we propose a novel approach to quantify the goodness-of-fit of self-exciting point process generative models. As demonstrated by our theoretical results as well as our numerical experiments on synthetic data, the GS statistic can be used to perform model diagnosis and model selection. The intuition of our GS test under model mismatch is given in Section~\ref{theoreticalresult}, which can help us understand that the power of our proposed GS test largely depends on the expressive power of the Quasi-triggering function class.

\section{Acknowledgement}

The work is supported by the NSF CAREER Award CCF-1650913, and NSF CMMI-2015787, DMS-1938106, DMS-1830210. The authors would like to thank the Editor and the anonymous referees for the thoughtful comments and suggestions, which led to an improvement of the presentation. 

% in the additional experiments in Appendix \ref{add_exp}

% Last but not least, our results can also be extended to spatiotemporal Hawkes process(STHP). We only need to (i) extend PWHE to STHP, which only requires additional discretization of the spatial space and spatial inhomogeneous background intensity; (ii) utilize the asymptotic results in \cite{rathbun1996asymptotic} and generalize them to model mismatch case. However, this requires much stronger assumption on the Quasi-parameter space and we can deal with the STHP as is shown in \ref{police data}.

\newpage

\normalem
\bibliographystyle{apalike}
\bibliography{ref}

\newpage
\onecolumn
\begin{appendices}

\section{Mixing of two Hawkes processes}\label{mixing}
%\vspace{-.1in}
We first present a useful lemma, which provides the proof for full model case (i.e. under $H_1$).  Another equivalent definition of conditional intensity $\lambda(t|\cH_t)$ for a counting process $\{N(t) : t \geq 0 \}$ with history
 $\cH_t \ (t \geq 0)$ is
\begin{align*}
    \mathbb{P}(N(t+h)-N(t)=m | \mathcal{H}_t)=\left\{\begin{array}{ll}
\lambda(t|\cH_t)h+o(h), & m=1 \\
o(h), & m>1 \\
1-\lambda(t|\cH_t)h+o(h), & m=0
\end{array}\right.
\end{align*}
We will make use of this definition to prove the following lemma.
\begin{lemma}\label{sum_intensity}\it
Suppose we have $n$ Hawkes processes $\{N_z(t) : t \geq 0 \} \ (z=1,2,\dots,n)$ with conditional intensity specified by \eqref{eq:hawkes}. Define the mixing to be $N(t)=\sum_{z=1}^n N_z(t)$. The conditional intensity of mixing of $n$ Hawkes processes is sum of those $n$ conditional intensities. That is,
$$\lambda(t |\cH_{t})=\sum_{z=1}^n \lambda_z(t|\cH_{z,t}),$$where $\cH_t=\cup_{z=1}^n \ \cH_{z,t}$.
\end{lemma}
\begin{proof}[Proof of Lemma~\ref{sum_intensity}]

We prove by the definition of conditional intensity. For any non-negative integer $m\in\mathbb{Z}_+$, denote $ \textbf{\uwave{m}} = (m_1,\dots,m_n)$ and $M = \{\textbf{\uwave{m}}  \ |\ m_1+\dots+m_n=m, \  m_i\in\mathbb{Z}_+\}$,
\begin{equation*}
    \mathbb{P}(N(t+h)-N(t)=m | \mathcal{H}_t) = \sum_{\textbf{\uwave{m}} \in M}  \ \ \prod_{i=1}^n \mathbb{P}(N_i(t+h)-N_i(t)=m_i | \mathcal{H}_{i,t}).
\end{equation*}

\textbf{Case 1:} When $m>1$, it is easy to see $\mathbb{P}(N(t+h)-N(t)=m | \mathcal{H}_t)=o(h)$, since either there are at least two $m_i$'s $\geq1$ or at least one $m_i\geq2$.

\textbf{Case 2:} When $m=1$, there will be one and only one of all $m_i$'s taking value 1 and the rest will be all zeros. Thus, we have \begin{equation*}
\begin{split}
    &\mathbb{P}(N(t+h)-N(t)=1 | \mathcal{H}_t) \\
    =&~  \sum_{j=1}^n \mathbb{P}(N_j(t+h)-N_j(t)=1 | \mathcal{H}_{j,t})\prod_{i\not=j}^n \mathbb{P}(N_i(t+h)-N_i(t)=0 | \mathcal{H}_{i,t}) \\=& ~ \sum_{j=1}^n (\lambda_j(t|\cH_{j,t})h+\mathrm{o}(h))\prod_{i\not=j}^n (1-\lambda_i(t|\cH_{i,t})h+o(h))= \sum_{j=1}^n \lambda_j(t|\cH_{j,t})h+\mathrm{o}(h).
\end{split}
\end{equation*}

\textbf{Case 3:} When $m=0$, all $m_i$'s will be zeros and we will have \begin{equation*} \begin{split}\mathbb{P}(N(t+h)-N(t)
=0 | \mathcal{H}_t) =& \prod_{i=1}^n \mathbb{P}(N_i(t+h)-N_i(t)=0 | \mathcal{H}_{i,t}) \\= &\prod_{j=1}^n (1-\lambda_i(t|\cH_{i,t})h+o(h))=1 - \sum_{i=1}^n \lambda_j(t|\cH_{j,t})h+\mathrm{o}(h).
\end{split}\end{equation*}

Let $\lambda(t|\cH_{t})=\sum_{i=1}^n \lambda_i(t|\cH_{i,t}),$ and we will find out this is the conditional intensity for $N(t)$.
\end{proof}

\begin{proof}[Proof of Proposition~\ref{sum_2_hawkes}]
We can see under the alternative hypothesis, the result directly follows Lemma~\ref{sum_intensity}. Under null hypothesis, by Lemma \ref{sum_intensity}, it is easy to show $N(t)$ defined in Proposition~\ref{sum_2_hawkes} has intensity 
\begin{equation*} \begin{split}
   \lambda(t|\cH_{t})&=\mu^{(1)}+\mu^{(2)}+\int_{0}^{t} \phi(t-u) \mathrm{d}(N_1(u)+N_2(u))=\mu+\int_{0}^{t} \phi(t-u) \mathrm{d}N(u),
\end{split}\end{equation*}
where $\cH_t=\cH_{1,t}\cup  \cH_{2,t}$. 

By the definition of Hawkes Process in Section~\ref{background},  we can see the mixing of two Hawkes processes under $H_0$ is still a Hawkes process. We complete the proof.
\end{proof}

\section{A non-parametric estimation of the Quasi-conditional intensity}\label{PWHE}
%\vspace{-.1in}
\subsection{Probability Weighted Histogram Estimation under null hypothesis}
%\vspace{-.1in}

Here, we redefine the Quasi-parameter as $\theta = (\mu,\alpha^{(1)},g_1^{(1)},\dots,g_{n_0}^{(1)},\alpha^{(2)},g_1^{(2)},\dots,g_{n_0}^{(2)})$, where $\mu \triangleq \mu^{(1)}+\mu^{(2)}$. This is because we will estimate the triggering magnitude and the temporal triggering function separately.

The full model Quasi-parameter space is given by
$$\Theta = \Big\{\theta \ \big |\  \mu > 0 \ \text{and}\ \int_0^{\infty} \phi^{(z)}(u ) du =\int_0^{\infty} \alpha^{(z)} g^{(z)}(u ) du = \alpha^{(z)}<1 \ \ (z=1,2)\Big\}.$$ Under $H_0$, we have
$$\theta_0 \in \Theta_0 = \{\theta \in \Theta \ | \ \alpha^{(1)} = \alpha^{(2)}= \alpha \ \ \ \text{and} \ \ \ g_k^{(1)} = g_{k}^{(2)}=g_k \ (k=1,\dots,n_0)\}.$$

Denote $\cH_t=\cH_{1,t}\cup  \cH_{2,t}=\{t_1,\dots,t_{N}\}$. Define the branching structure as follows:
\begin{equation*} \begin{split}
    p_{i j}=\left\{\begin{array}{ll}
\text { probability event } i \text { is triggered by event } j, & i>j \\
\text { probability event } i \text { comes from background, } & i=j \\
\ 0, & i<j
\end{array}\right.
\end{split}\end{equation*}
Apparently, we want to estimate the Quasi-background intensity from background process only and Quasi-triggering function from the triggered events only. Instead of using a hard-threshold indicator, \citet{zhuang2002stochastic} used a stochastic declustering procedure to separate the background events from triggered ones by assigning each event a weight, or rather the probability that this event comes from background or is direct offspring from an individual ancestor. Then, we can use a probability weighted estimator to estimate Quasi-background intensity and Quasi-triggering function. The algorithm is as follows:

Assume we have estimated branching structure $p_{i i}^{(v)}$ at iteration $v$, then we can estimate the Quasi-background intensity as follows:
\begin{equation}\label{algo1_background}
    \mu^{(v)}=\frac{1}{T} \sum_{i=1}^N p_{i i}^{(v)}.
\end{equation}
For the Quasi-triggering component, as we assume $g$ to be a p.d.f., we can estimate the magnitude of triggering effect from triggered events only:\begin{equation}\label{algo1_trigger_mag}
    \alpha^{(v)} = 1 - \sum_{i=1}^N p_{ii}^{(v)}/N.
\end{equation}
For the temporal component in the Quasi-triggering function, for each bin (as we discretize in \eqref{piecewiseconst}), we estimate its parameter from those triggered events which falls into that bin, i.e.
\begin{equation}\label{algo1_trigger_time}
    g_{k}^{(v)}=\frac{\sum_{B_{k}} p_{i j}^{(v)}}{\Delta t_{k} \sum_{i=1}^{N} \sum_{j=1}^{i-1} p_{i j}^{(v)}},  \ (k=1, \ldots, n_{0}).
\end{equation}
After estimating the Quasi-conditional intensity function, we update the branching structure. More specifically, for $i>j$:\begin{align}
	    p_{i j}^{(v+1)}&=\PP(\text{i-th event is triggered by j-th event}|\cH_{t_i}) \label{algo1_trigger_prob}=\frac{ \alpha^{(v)}g^{(v)}\left(t_{i}-t_{j}\right) }{\mu^{(v)}+\sum_{j=1}^{i-1}  \alpha^{(v)}g^{(v)}\left(t_{i}-t_{j}\right) },\\
p_{i i}^{(v+1)}&=\PP(\text{i-th event comes from background}|\cH_{t_i}) \label{algo1_background_prob}=\frac{\mu^{(v)}}{\mu^{(v)}+\sum_{j=1}^{i-1}  \alpha^{(v)}g^{(v)}\left(t_{i}-t_{j}\right). }
\end{align}
We summarize the algorithm as follows:
%\vspace{-.1in}
\begin{algorithm}[H]
    \caption{Probability Weighted Histogram Estimation of Quasi-log-likelihood under $H_0$}\label{algo1}
    \begin{algorithmic}
	\STATE{\textbf{Initialize}: choose stopping critical value $\epsilon$ (e.g. $10^{-3}$), initialize $p_{ij}^{(0)}$ and set $p_{ij}^{(-1)} = \epsilon + p_{ij}^{(0)}$ and iteration index $v=0$.}
	\WHILE{$\max _{i > j}\left|p_{i j}^{(v)}-p_{i j}^{(v-1)}\right|<\varepsilon$}
	\STATE{1. Estimate Quasi-background rate $\mu$ as in \eqref{algo1_background}.}
	\STATE{2. Estimate Quasi-triggering components magnitude $\alpha$ and temporal $g(t)$ as in \eqref{algo1_trigger_mag} and \eqref{algo1_trigger_time}.}
	\STATE{3. Update probabilities $p_{ij}^{(v+1)}$'s as in \eqref{algo1_trigger_prob} and \eqref{algo1_background_prob}.}
	\STATE{4. $v \leftarrow v + 1$}
	\ENDWHILE
    \end{algorithmic}
\end{algorithm}
%\vspace{-.2in}
\subsection{EM-type algorithm derivation}\label{algo1EM}
%\vspace{-.1in}
In \citet{fox2016spatially}, they assumed the ground-truth takes piecewise constant form \eqref{piecewiseconst} and demonstrated that algorithm \ref{algo1} is an EM-type algorithm under \eqref{piecewiseconst} by using complete data log-likelihood. However, they did not explicitly show the E-step also maximizes the complete data log-likelihood (or rather complete data Quasi-log-likelihood in our setting). We will first lower bound the Quasi-log-likelihood and then show that the algorithm iterates between maximizing this lower bound w.r.t. branching structure ($p_{ij}$'s) and w.r.t. the Quasi-conditional intensity (Quasi-background rate $\mu$, Quasi-triggering magnitude $\alpha$ and temporal Quasi-triggering function $g$). 

First recall the Quasi-log-likelihood function under $H_0$: 
\begin{equation*} \begin{split}  \ell_0(\theta)=- \mu T +\sum_{i=1}^{N}  \log \left(\mu+\sum_{i>j}  \phi\left(t_{i}-t_{j} \right)\right)  
-\sum_{j=1}^{N} \int_{t_{j}}^{T}  \phi\left(t-t_{j}\right)   d t
\end{split}\end{equation*}
We can simplify the last term above by using integral approximation of \citet{schoenberg2013facilitated}:
\begin{equation*} \begin{split} \sum_{j=1}^{N} \int_{t_{j}}^{T}   \phi\left(t_{i}-t_{j} \right)  d t=\sum_{j=1}^{N} \int_{t_{j}}^{T} \alpha g\left(t-t_{j}\right) d t \approx\sum_{j=1}^{N}  \alpha \int_{t_{j}}^{\infty} g\left(t-t_{j}\right) d t     =\alpha N
\end{split}\end{equation*}
Thus we can ignore the last term when maximizing the log-likelihood function. Next, we lower bound the first term in the Quasi-log-likelihood function by Jensen's inequality:
\begin{equation*} \begin{split}  \sum_{i=1}^{N}  \log \left(\mu +\sum_{i>j}  \phi\left(t_{i}-t_{j} \right)\right)
= &\sum_{i=1}^{N}  \log \left(p_{ii}\frac{\mu }{p_{ii}}+\sum_{i>j} p_{ij} \frac{\phi\left(t_{i}-t_{j} \right)}{p_{ij}}\right)\\ 
\geq &\sum_{i=1}^{N}  p_{ii}\log \left(\frac{\mu }{p_{ii}}\right) +\sum_{i>j} p_{ij} \log \left(\frac{\phi\left(t_{i}-t_{j} \right)}{p_{ij}}\right),
\end{split}\end{equation*}
where $p_{ij}$'s satisfy $\sum_{i\geq j} p_{ij} = 1$. Then we can get a lower bound on the approximation of Quasi-log-likelihood under the piecewise constant parameterization:
\begin{equation*} \begin{split}  - \alpha N -T \mu +
\sum_{i=1}^{N} \Bigg[ p_{i i} \log \left(\mu \right)  + \sum_{i>j} p_{i j}\bigg(\log \alpha + \log \Big(\sum_{k=1}^{n_{0}} g_{k} \mathbf{1}_{B_{k}}\left(t_{i}-t_{j}\right)\Big)\bigg) - \sum_{i\geq j} p_{i j} \log(p_{i j})\Bigg] 
\end{split}\end{equation*}
Denote this lower bound by $\Tilde{\ell}(\theta)$. We maximize this lower bound under the following constraints:
\begin{equation*} \begin{split}
     \sum_{k=1}^{n_{0}} g_{k} \Delta t_{k} = 1, & \quad (g(t) \ \text{is a p.d.f.}) \\
    \sum_{i\geq j} p_{ij} = 1, & \quad (p_{ij}\text{'s are probability weights}) 
    \end{split}\end{equation*}
By adding Lagrange multipliers, this is equivalent to maximizing the following objective:
\begin{equation*} \begin{split}
\Tilde{L} =  \ &\sum_{i=1}^{N} \Bigg[ p_{i i} \log \left(\mu \right)  + \sum_{i>j} p_{i j}\bigg(\log \alpha + \log \Big(\sum_{k=1}^{n_{0}} g_{k} \mathbf{1}_{B_{k}}\left(t_{i}-t_{j}\right)\Big)\bigg) - \sum_{i\geq j} p_{i j} \log(p_{i j})\Bigg]
\\ &- \alpha N -T \mu-c_{1}\left(\sum_{k=1}^{n_{0}} g_{k} \Delta t_{k}-1\right)-\sum_{i=1}^{N}  c_{2}^{(i)}\left(\sum_{i\geq j} p_{ij} - 1\right).
\end{split}\end{equation*}

\textbf{M-step:} By taking first order derivative w.r.t. $\mu$ and setting it to zero, we will have:

$$
\begin{aligned}
\frac{\partial \Tilde{L}}{\partial \mu} =\sum_{i=1}^{N}\left(\frac{p_{i i}   }{\mu }\right)-T = 0.
\end{aligned}
$$
Solving for $\mu$ and we will get 
$$
\mu =\frac{\sum_{i=1}^{N} p_{i i}   }{T  },
$$ which is the same as the update in step 1 in Algorithm~\ref{algo1}. This means when we have $p_{ij}^{(v)}$'s at iteration $v$, the update in step 1 in Algorithm~\ref{algo1} leads to a larger Quasi-log-likelihood value. Similarly taking derivative w.r.t. $\alpha$ and setting it to zero leads to the update in step 2: $\alpha^{(v)} = 1 - \sum_{i=1}^N p_{ii}^{(v)}/N$.
$$
\begin{aligned}\frac{\partial \Tilde{L}}{\partial g_{k}} &=\sum_{i=1}^{N} \sum_{i>j}\left(\frac{p_{i j} \mathbf{1}_{B_{k}}\left(t_{i}-t_{j}\right)}{g_{k}}\right)-c_{1} \Delta t_{k}=0 \\
\frac{\partial \Tilde{L}}{\partial c_{1}} &=1-\sum_{k=1}^{n_{0}} g_{k} \Delta t_{k}=0
\end{aligned}
$$

We can solve for $g_k$ and $c_1$ by some simple algebra and then get the update for $g_k$ at iteration $v$ (given $p_{ij}^{(v)}$'s) :
$$
g_{k}^{(v)}=\frac{\sum_{i=1}^{N} \sum_{i>j} p_{i j}^{(v)} \mathbf{1}_{B_{k}}\left(t_{i}-t_{j}\right)}{\Delta t_{k} \sum_{j=1}^{N} \sum_{i>j} p_{i j}^{(v)}}.
$$

\textbf{E-step:} As for $p_{ij}$'s, denote $$ \log \phi_{ij} = \log \alpha + \log \left(\sum_{k=1}^{n_{0}} g_{k} \mathbf{1}_{B_{k}}\left(t_{i}-t_{j}\right)\right). $$ Repeat the similar procedure, we will get:
$$
\begin{aligned}\frac{\partial \Tilde{L}}{\partial p_{ij}} &=-\log(p_{ij}) -1 -  c_{2}^{(i)} +\log \phi_{ij}=0 \\
\frac{\partial \Tilde{L}}{\partial p_{ii}} &=-\log(p_{ii}) -1 -  c_{2}^{(i)} +\log \mu=0\\
\frac{\partial \Tilde{L}}{\partial  c_{2}^{(i)}} &=\sum_{i\geq j} p_{ij} - 1=0
\end{aligned}
$$
By the first two equations we have $$\frac{p_{ii}}{p_{ij}} = \frac{\mu}{\phi_{ij}}.$$ Plug this back into the last equation and we will get the update in step 3 in Algorithm~\ref{algo1}. Thus, we validate Algorithm~\ref{algo1} as an EM-type algorithm.

%\vspace{-.1in}
\section{Explicit form of GS statistic}\label{explicitGS}
%\vspace{-.1in}
Note that $\phi^{(z)}(u) = \sum_{k=1}^{n_0} \phi_k^{(z)}\mathbf{1}_{B_{k}}(u)$. To simplify the explicit expressions, we first define the following notations:
\begin{equation*} \begin{split}
    G(i,z';z) &= \sum_{j=1}^{N_{z'}} \phi^{(z)}(t_i^{(z')}-t_j^{(z)}),\\
    G_k'(i,z';z) &= \sum_{j=1}^{N_{z'}} \mathbf{1}_{B_{k}}(t_i^{(z')}-t_j^{(z)})\\
    \Delta_{z,i} &= \mu+
      G(i,z;z)+
     G(i,z;z')
\end{split}\end{equation*}

Here, $G(i,z';z)$ represents the triggering effect of events in process $z$ to $i-$th event in process $z'$. $G_k'(i,z';z)$ is the partial derivative of $G(i,z';z)$ w.r.t. $\phi_k^{(z)}$. 

Note that $\phi^{(z)}(\cdot)$ and $\mathbf{1}_{B_{k}}(\cdot)$ $(k=1,2,\dots,n_0)$ take value zero on $(-\infty,0]$. Thus we have $$\sum_{j<i} \phi^{(z)}(t_i^{(z)}-t_j^{(z)}) = \sum_{j=1}^{N_{z}} \phi^{(z)}(t_i^{(z)}-t_j^{(z)}), $$which can be denoted by $G(i,z;z)$ we just defined. By our notations, the Quasi-log-likelihood takes the following form: \begin{equation*} \begin{split}
   \ell_1(\mu, \phi^{(1)},  \phi^{(2)}|\cH_{t}) =  - \mu \ T + \sum_{z=1}^2 \sum_{i=1}^{N_z} \log \Delta_{z,i} -  \int_0^{T-t_i^{(z)}} \phi^{(z)} (u)du,
\end{split}\end{equation*}where $(\mu, \phi^{(1)},\phi^{(2)}) = (\mu,\phi_1^{(1)},\dots,\phi_{n_0}^{(1)},\phi_1^{(2)},\dots,\phi_{n_0}^{(2)})$. Those parameters are denoted by $\theta$ to simplify the notations. To get the explicit form of GS statistic, we only need to calculate the first two order partial derivative of $\ell_1(\mu, \phi^{(1)},  \phi^{(2)}|\cH_{t})$ w.r.t. $\theta$.

\textbf{First order partial derivatives:} 
\begin{equation*} \begin{split}
    \frac{\partial \ell_1(\mu, \phi^{(1)},  \phi^{(2)}|\cH_{t})}{ \partial \mu} =& \sum_{z=1}^2 \sum_{i=1}^{N_z} \frac{1}{\Delta_{z,i}
   }-  T,\\
    \frac{\partial \ell_1(\mu, \phi^{(1)},  \phi^{(2)}|\cH_{t})}{ \partial \phi_k^{(z)}} =&  \sum_{i=1}^{N_z} \frac{   G_k'(i,z;z)}{\Delta_{z,i}
   }+\sum_{i=1}^{N_{z'}} \frac{ G_k'(i,z';z)}{\Delta_{z',i}
   }-\sum_{i=1}^{N_z} \int_0^{ T - t_i^{(z)}}   \mathbf{1}_{B_{k}}(u) du.
\end{split}\end{equation*}Here, we get the explicit expression for $S_T(\theta)$ and $A_T(\theta)$.

\textbf{Second order partial derivatives:} 
\begin{equation*} \begin{split}
    \frac{\partial^2 \ell_1(\mu, \phi^{(1)},  \phi^{(2)}|\cH_{t})}{ \partial \mu^2} =& - \sum_{z=1}^2  \sum_{i=1}^{N_z} \frac{1}{\Delta_{z,i}^2
   }\\
    \frac{\partial^2 \ell_1(\mu, \phi^{(1)},  \phi^{(2)}|\cH_{t})}{ \partial (\phi_k^{(z)})^2} =&  - \sum_{i=1}^{N_z} \left(\frac{   G_k'(i,z;z)}{\Delta_{z,i}
   }\right)^2-\sum_{i=1}^{N_{z'}} \left(\frac{ G_k'(i,z';z)}{\Delta_{z',i}
   }\right)^2
 \end{split}\end{equation*}
\begin{equation*} \begin{split}
   \frac{\partial^2 \ell_1(\mu, \phi^{(1)},  \phi^{(2)}|\cH_{t})}{ \partial \mu \partial \phi_k^{(z)}} =&  -  \sum_{i=1}^{N_z}  \frac{  G_k'(i,z;z)}{\Delta_{z,i}^2
   }+\sum_{i=1}^{N_{z'}}   \frac{G_k'(i,z';z)}{\Delta_{z',i}^2
   }\\
    \frac{\partial^2 \ell_1(\mu, \phi^{(1)},  \phi^{(2)}|\cH_{t})}{\partial \phi_k^{(z)} \partial \phi_l^{(z')} } =&-  \sum_{i=1}^{N_z} \frac{  G_l'(i,z;z')  G_k'(i,z;z)}{\Delta_{z,i}^2
   }+\sum_{i=1}^{N_{z'}}  \frac{  G_l'(i,z';z') G_k'(i,z';z)}{\Delta_{z',i}^2
   }\\
    \frac{\partial^2 \ell_1(\mu, \phi^{(1)},  \phi^{(2)}|\cH_{t})}{\partial \phi_k^{(z)} \partial \phi_l^{(z)}} =&- \sum_{i=1}^{N_z} \frac{  G_k'(i,z;z)G_l'(i,z;z)}{\Delta_{z,i}^2
   }+\sum_{i=1}^{N_{z'}}  \frac{  G_k'(i,z';z)G_l'(i,z';z)}{\Delta_{z',i}^2
   }
\end{split}\end{equation*}
%\vspace{-.1in}
\section{Asymptotic properties of QMLE and GS test} \label{all_asym}
%\vspace{-.1in}
\subsection{Identifiability of the estimand and justification of our testing framework}
%\vspace{-.1in}
\begin{proof}[Proof of Identifiability]

One can verify that for each specific sample trajectory $\cH_T$, $\ell_1 (\theta|\cH_T)$ is composed of two parts: a linear function of $\theta$ plus several logarithm of a linear function of $\theta$. This means that $\ell_1 (\theta|\cH_T)$ is concave in $\theta$. We will use a very simple example to elaborate on this. 

Suppose we only have 3 events $\cH_T = \{t_1^{(1)},t_1^{(2)},t_2^{(1)}\}$, where $T = t_2^{(1)}$. Then\begin{align*}
    \ell_1 (\theta|\cH_T) =& -\mu T 
- \sum_{k=1}^{n_0} \phi_k^{(1)} \int_0^{T - t_1^{(1)}} \mathbf{1}_{B_k}(u)du
- \sum_{k=1}^{n_0} \phi_k^{(2)} \int_0^{T - t_1^{(2)}} \mathbf{1}_{B_k}(u)du\\
&+\log \mu
+ \log (\mu + \phi_{k_1}^{(1)})
+ \log (\mu + \phi_{k_2}^{(1)}+\phi_{k_3}^{(2)}),
\end{align*}where $k_1,k_2,k_3$ are the indices of the bins which $t_1^{(2)} - t_1^{(1)},t_2^{(1)}- t_1^{(1)},t_2^{(1)}-t_1^{(2)}$ fall into, respectively.
% $$-\mu T 
% - \sum_{k=1}^{n_0} \phi_k^{(1)} \int_0^{T - t_1^{(1)}} \mathbf{1}_{B_k}(u)du
% - \sum_{k=1}^{n_0} \phi_k^{(2)} \int_0^{T - t_1^{(2)}} \mathbf{1}_{B_k}(u)du .
% $$
Note that $\int_0^{t} \mathbf{1}_{B_k}(u)du$ is the length of the intersection of $[0,t]$ and $B_k$, which is a constant.
% We denote it by $|[0,t]\cap B_k|$. 
Thus, it is easy to see from the example that $\ell_1 (\theta|\cH_T)$ is concave in $\theta$ for any fixed trajectory.

% The rest is
% $$ \log \mu
% + \log (\mu + \phi_{k_1}^{(1)})
% + \log (\mu + \phi_{k_2}^{(1)}+\phi_{k_3}^{(2)}),
% $$

Next, we can show that $\ell_1 (\theta|\cH_T)$ will remain the same for sample trajectories that are "close" to each other. In the simple example above, as long as $t_1^{(2)} - t_1^{(1)},t_2^{(1)}- t_1^{(1)},t_2^{(1)}-t_1^{(2)}$ remain in bins $B_{k_1},B_{k_2},B_{k_3}$, the value of the corresponding $\ell_1 (\theta|\cH_T)$ will not change. For fixed number of events $N$, we call all trajectories with $N$ events that corresponds to the same Quasi-log-likelihood value a case. It is easy to see the number of all cases for fixed number of events $N$ is countable. Then the expectation taken w.r.t. all possible trajectories will reduce to a countablely infinite summation. That is $$\EE [\ell_1(\theta|\cH_T)] = \sum_i \ell_{1,i}(\theta) p_i,$$where $p_i$ is the probability of all sample trajectories such that $\ell_1 (\theta|\cH_T) \equiv \ell_{1,i}(\theta)$. Note that we just show $\ell_{1,i}(\theta)$ is concave in $\theta$. Thus the objective is a linear combination of concave functions. This means $\theta_0$ actually solves a concave program. It is a unique maximizer of the expected Quasi-log-likelihood, i.e. globally identifiable. We have a well-defined estimand here.
\end{proof}

% Since we need the exact and explicit form of the objective to answer whether $\theta_0 \in \Theta_0$ under $H_0$, we should first answer what is $p_i$. This probability can be either obtained by using the idea considering the aggregated sequence as a whole, $p_{agg}$, or two independent processes, $p_1 \times p_2$. 

\begin{proof}[Justification of our testing framework]
By adopting the view in \citet{Akaike1998}, in \eqref{def_theta0} we are actually trying to find a $\theta_0 \in \Theta$ whose corresponding Quasi-likelihood has a minimum K-L divergence with the unknown ground-truth $\lambda^*$. 

As is suggested in \citet{Akaike1998}, we can view this as a statistical decision problem where the loss function is $\log \lambda^* / \lambda_{\theta}$. For the simple example above, the loss function can be expressed by\begin{align*}
   & (\mu - \mu^*) T +
\sum_{k=1}^{n_0}  \int_0^{T - t_1^{(1)}} (\phi_k^{(1)} - \phi^{*(1)}(u))\mathbf{1}_{B_k}(u)du
+ \sum_{k=1}^{n_0}  \int_0^{T - t_1^{(2)}} (\phi_k^{(2)} - \phi^{*(2)}(u)) \mathbf{1}_{B_k}(u)du\\
+&\log (\mu^*/\mu)
+ \log \frac{\mu^* + \phi^{*(1)}(t_1^{(2)} - t_1^{(1)})}{\mu + \phi_{k_1}^{(1)}}
+ \log \frac{\mu^* + \phi^{*(1)}(t_2^{(1)}- t_1^{(1)}) + \phi^{*(2)}(t_2^{(1)}-t_1^{(2)})}{\mu + \phi_{k_2}^{(1)}+\phi_{k_3}^{(2)}},
\end{align*}

Taking all possible sample trajectories into account, when $ \phi^{*(1)} =  \phi^{*(2)}$, apparently we will achieve minimum risk when $\theta_0 \in \Theta_0$.
% for a fixed sample trajectory $\cH_T$ However, note that in our setting, modelling the aggregated sequence as a whole is itself a Quasi-model { \color{blue} (Not sure about this. If so, we should use true model structure, i.e. $p_1 \times p_2$ and the analysis will be really complicated... )}. 
% However, this does not affect anything since for us $p_i$ is just a unknown weight which sums up to 1 and how to calculate it is not our concern and infeasible, since we do not know the unknown true model. 
\end{proof}

%\vspace{-.1in}
\subsection{Proof of Lemma~\ref{asym_qmle}: consistency and asymptotic normality of QMLE}\label{consistency}
%\vspace{-.1in}
\begin{proof}

We will provide a generalization of the asymptotic properties MLEs under correct model specification for temporal Hawkes process in \citet{ogata1978asymptotic} to model misspecification (or model mismatch) case.

We first show that the assumptions in \citet{ogata1978asymptotic} hold for our Quasi-conditional intensity function.

\textbf{(A)} Since under our parameterization \eqref{piecewiseconst}, we have $\int_0^{\infty} \alpha g(t) dt = \alpha <1$, our point process model is stationary and ergodic. It is easy to check assumptions (A1) $\sim$ (A3).

\textbf{(B)} The Quasi-conditional intensity function we consider here is actually linear w.r.t. the parameters, then it is arbitrarily order continuous differentiable (i.e. smooth) and bounded within any compact set in the Quasi-parameter space. Assumptions (B1) $\sim$ (B7) hold trivially.

\textbf{(C)} By \eqref{piecewiseconst}, the Quasi-temporal triggering function is truncated on $[0,T_0]$, which means and complete data conditional intensity function $\lambda(t|\cH_{-\infty,t})$ will be exactly the same as $\lambda(t|\cH_{0,t})$ as long as $t>T_0$. Since Assumptions (C1) $\sim$ (C4) only require stochastic approximations of $\lambda(t|\cH_{0,t})$ to $\lambda(t|\cH_{-\infty,t})$ when $t$ goes to infinity, it is easy to see those assumptions are satisfied.

Next, since our parametric form \eqref{piecewiseconst} is only approximation to the true one, we need to slightly modify the theoretical results in \citet{ogata1978asymptotic} for our QMLE. Here we will not mention theorems or lemmas that we do not need to modify under model mismatch (except that we should keep in mind that the "true" parameter in \citet{ogata1978asymptotic} is understood as the maximizer of Quasi-likelihood) and it is easy to verify those theoretical results (from the beginning to Theorem 5) by just following the proof therein.

Before we proceed to the proof, we should note that the QMLE is $\Hat{\theta}_{QMLE}$ under $H_0$ and $\Tilde{\theta}_{QMLE}$ under $H_1$. Under $H_1: \theta_0 \not \in \Theta_0$, the estimator $\Tilde{\theta}_{QMLE}$ is obtained using the full model conditional intensity $\ell_1$ instead of $\ell_0$. The estimation is given in Algorithm~\ref{algo2} in Appendix~\ref{add_exp}.

For simplicity, we denote $\Bar{\theta}_{QMLE}$ to be $\Hat{\theta}_{QMLE}$ and $\Tilde{\theta}_{QMLE}$ under $H_0$ and $H_1$, respectively. That is,
\begin{equation*} 
    \Bar{\theta}_{QMLE}=\left\{\begin{array}{ll}
\Hat{\theta}_{QMLE},  &H_0 \text { is true}\\
\Tilde{\theta}_{QMLE},  &H_1\text { is true}
\end{array}\right.
\end{equation*}

\textbf{Modifications on Theorem 1.} Here $\theta_0$ is not the true parameter of the true conditional intensity function. Instead, it is the maximizer of Quasi-log-likelihood, i.e. our approximation to the true log-likelihood function. By the definition of $\theta_0$ and stationarity of the process, the first result still in this theorem still holds:
$$\frac{\partial \EE \big[ \ell_1 (\theta|\cH_t)\big]}{\partial \theta}\bigg|_{\theta = \theta_0} = 0.$$

However, the second result does not hold unless our approximation is indeed a correct specification of the model. More specifically, in general, $$dN(t) = \lambda^*(t|\cH_t)dt \not =  \lambda(t|\cH_t)dt,$$where $\lambda^*$ is the correct parametric form and typically unknown in practice. 

Thus, we have $$\EE \bigg[\frac{\partial  \ell_1 (\theta|\cH_t)}{\partial \theta_i}\frac{\partial  \ell_1 (\theta|\cH_t)}{\partial \theta_j}\bigg]\Bigg|_{\theta = \theta_0} \not = - \EE \bigg[\frac{\partial^2  \ell_1 (\theta|\cH_t)}{\partial \theta_i \partial \theta_j}\bigg]\Bigg|_{\theta = \theta_0}.$$

Using our notation, this can be re-expressed as $A(\theta_0)  \not =  B(\theta_0)$.

\textbf{Modifications on Theorem 2.} The convergence in our case is much stronger. By following the proof in \citet{fox2016spatially}, the convergence in probability comes from Assumptions (C), where the convergence in the stochastic approximation is only in probability sense. However, we just show that the stochastic approximation holds for every sample path as long as $t>T_0$ based on our parameterization \eqref{piecewiseconst} that the Quasi-temporal triggering function is truncated, i.e. our convergence is in almost surely sense. Thus, we have:$$\Bar{\theta}_{QMLE} \overset{a.s.}{\to} \theta_0 \ \ \ \text{  as  }\ \ \  T \overset{}{\to} \infty.$$

\textbf{Modifications on Theorem 4.} Since $A(\theta_0)  \not = B(\theta_0)$, the convergence result should be$$\frac{1}{\sqrt{T}} \frac{\partial  \ell_1 (\theta|\cH_T)}{\partial \theta}\bigg|_{\theta = \theta_0} \overset{d}{\to} N (0,A(\theta_0)) \ \ \ \text{  as  } \ \ \ T \overset{}{\to} \infty.$$

This is because $$\EE \bigg[\frac{\partial  \ell_1 (\theta|\cH_1)}{\partial \theta} \frac{\partial  \ell_1 (\theta|\cH_1)}{\partial \theta^\intercal}\bigg] = A(\theta_0)  \not = B(\theta_0),$$ where the first equality comes from definition and stationarity of the process.

\textbf{Modifications on Theorem 5.} By the proof of this theorem one can reach this result:
$$\sqrt{T} (\Bar{\theta}_{QMLE} - \theta_0) \overset{d}{\to} N \Big(0,B^{-1}(\theta_0)A(\theta_0)B^{-1}(\theta_0)\Big) \ \ \ \text{  as  } \ \ \ T \overset{}{\to} \infty.$$

Again, since $A(\theta_0)  \not = B(\theta_0)$, the asymptotic covariance matrix is not $B^{-1}(\theta_0)$ and that's the modification here. Besides, the asymptotic $\chi^2$ distribution of log-likelihood
ratio does not hold because of the model mismatch.

Here, we complete the proof.
\end{proof}
%\vspace{-.1in}
\subsection{Proof of Theorem~\ref{asym_null}: asymptotic distribution under null hypothesis}\label{null}
%\vspace{-.1in}
This proof is highly involved. To help better understand this proof, we first provide a high level sketch on why our GS statistic follows a $\chi^2$ distribution.
\begin{proof}[Proof Sketch]  $\Hat{\theta}_{QMLE}$ solves the following problem $$\max_{\theta \in \Theta} \ell_1(\theta|\cH_T) \ \ \ \text{  s.t.  } \ \ \  h(\theta)=0.$$ By adding Lagrange Multiplier $\zeta_T$, we can derive that $\Hat{\theta}_{QMLE}$ satisfies:
\begin{equation}\label{lagrange}
    \nabla \ell_1(\Hat{\theta}_{QMLE}|\cH_T) + \zeta_T^\intercal\nabla h(\Hat{\theta}_{QMLE}) =  S_T(\Hat{\theta}_{QMLE}) + \zeta_T^\intercal\nabla h(\Hat{\theta}_{QMLE}) = 0.
\end{equation}
Following idea in \citet{boos1992generalized}, we can use Taylor expansion to expand $S_T(\theta_0)$ about $\Hat{\theta}_{QMLE}$ and $h(\Hat{\theta}_{QMLE})$ about $\theta_0$ (note that we have $h(\theta_0)=0$ under $H_0$):
\begin{equation*} \begin{split}
    S_T(\Hat{\theta}_{QMLE}) & = S_T(\theta_0) - B_T(\Hat{\theta}_{QMLE}) (\Hat{\theta}_{QMLE} - \theta_0) + o(1), \\
    0 = h(\Hat{\theta}_{QMLE}) & = h(\theta_0)+ \nabla h(\theta_0) (\Hat{\theta}_{QMLE} - \theta_0) + o(1).
\end{split}\end{equation*}

Note that by our notation $\nabla h(\theta) = H(\theta)$. Since $h(\theta)$ is linear in $\theta$, its gradient is a constant matrix and we can denote $H = \nabla h(\theta)$. 

Pre-multiply the first equation above by $H^\intercal \Big(H B_T^{-1}(\theta) H^\intercal\Big)^{-1}  H B_T^{-1}(\theta)\bigg|_{\theta=\Hat{\theta}_{QMLE}}$ 
\begin{equation*} \begin{split}
    & H^\intercal \Big(H B_T^{-1}(\theta) H^\intercal\Big)^{-1}  H B_T^{-1}(\theta)S_T(\theta_0)\bigg|_{\theta=\Hat{\theta}_{QMLE}} \\ 
  & \ \ \ \ \ \ \ \ \ \  =  B_T^{\frac{1}{2}}(\theta) \left( B_T^{-\frac{1}{2}}(\theta) H^\intercal \Big(H B_T^{-1}(\theta) H^\intercal\Big)^{-1}  H B_T^{-\frac{1}{2}}(\theta) \right) B_T^{-\frac{1}{2}}(\theta) S_T(\theta)\bigg|_{\theta=\Hat{\theta}_{QMLE}}  + o(1) .
\end{split}\end{equation*}
The matrix in the middle of RHS is a projection matrix for the column space of $B_T^{-\frac{1}{2}}(\Hat{\theta}_{QMLE}) H^\intercal$, and from \eqref{lagrange} we know $B_T^{-\frac{1}{2}}(\Hat{\theta}_{QMLE}) S_T(\Hat{\theta}_{QMLE})$ is already in this space. This means the RHS is exactly $S_T(\Hat{\theta}_{QMLE})$ and we will get:
$$S_T(\Hat{\theta}_{QMLE})= H^\intercal \Big(H B_T^{-1}(\theta_0) H^\intercal\Big)^{-1}  H B_T^{-1}(\theta_0) S_T(\theta_0) + o(1).$$Rewrite GS statistic as $$\Hat{GS}_T = \frac{1}{\sqrt{T}} S_T^\intercal(\Hat{\theta}_{QMLE}) \left(T  \Hat{\Sigma}^{-1}\right) \frac{1}{\sqrt{T}}S_T(\Hat{\theta}_{QMLE}).$$ By Lemma \ref{asym_qmle}, one can verify $S_T(\Hat{\theta}_{QMLE})/\sqrt{T}$ has a asymptotic normal distribution with $T \Hat{\Sigma}^{-1}$ being a consistent estimator of generalized inverse of its asymptotic covariance matrix. 

Since $H$ is of rank $r$, we verify that $\Hat{GS}_T \sim \chi^2_r$.\end{proof}

Next, we present a more rigorous proof following the method in \citet{white1982maximum}.

\begin{proof}
We first state some useful results:

By the almost surely convergence of QMLE (modifications of Theorems 2 and 5 in \citet{ogata1978asymptotic}), we have that\begin{equation*} \begin{split}
   & \frac{1}{T}A_T(\Hat{\theta}_{QMLE}) \overset{a.s.}{\to} A(\theta_0) \ \ \ \text{ as }  \ \ \ T\rightarrow \infty\\
       & \frac{1}{T} B_T(\Hat{\theta}_{QMLE}) \overset{a.s.}{\to}B(\theta_0) \ \ \ \text{ as }  \ \ \ T\rightarrow \infty.
\end{split}\end{equation*}

The modification of Theorem 1 in \citet{ogata1978asymptotic} can be re-expressed as $S(\theta_0) = 0.$

The modification of Theorem 4 in \citet{ogata1978asymptotic} can be re-expressed as follows$$ \frac{1}{\sqrt{T}}S_T(\theta_0)\overset{d}{\to} N \Big(0,A(\theta_0)\Big) \ \ \ \text{ as }  \ \ \ T\rightarrow \infty, $$where $S_T$ is the Quasi-score function (i.e. first order gradient of Quasi-log-likelihood function).

Under null hypothesis, the asymptotic $\chi^2$ distribution of GS statistic under model mismatch (e.g. Theorem 3.5. in \citet{white1982maximum} and Section 4.2. in \citet{boos1992generalized}) can be extended to temporal Hawkes process.

The QMLE actually solves the following optimization problem: $$ \max_{\theta \in \Theta_0}  \ell_0(\theta|\cH_T).$$ Since $\ell_1(\theta) = \ell_0 (\theta) \  (\forall \theta \in \Theta_0)$, equivalently it can be re-expressed as $$\max_{\theta \in \Theta_0}  \ell_1(\theta|\cH_T),$$ or $$  \max_{\theta \in \Theta}  \ell_1(\theta|\cH_T) \ \ \  \text{ s.t. } \ \ \ h(\theta)=0.$$

We can reformulate this by adding Lagrange Multiplier $\zeta_T$:
$$ \max_{\theta \in \Theta} \frac{1}{T} \ell_1(\theta|\cH_T) + \zeta_T^\intercal h(\theta).$$

Since $h$ as well as $\nabla h$ both has full row rank $r$, by Lagrange Multiplier Theorem (e.g. Theorem 42.9 in \citet{bartle1976elements}), we can guarantee the existence of $\zeta_T$, which satisfies:
\begin{align}
    \label{zeta_condi}&\frac{1}{T} \nabla \ell_1(\Hat{\theta}_{QMLE}|\cH_T) +  \Big(\nabla h(\Hat{\theta}_{QMLE})\Big)^\intercal  \zeta_T= 0,\\
    \nonumber &h(\Hat{\theta}_{QMLE})=0.
\end{align}
We denote $S_T(\theta) = \nabla \ell_1(\theta|\cH_T)$. By the mean-value theorem for random functions (Lemma 3 in \citet{jennrich1969asymptotic}), we have:
\begin{align}
    \label{expansion1}& S_T(\Hat{\theta}_{QMLE}) =  S_T(\theta_0) +  B_T(\Bar{\theta}) (\Hat{\theta}_{QMLE} - \theta_0),\\
    \label{expansion2}&0 = h(\Hat{\theta}_{QMLE})=h(\theta_0)+ \nabla h(\Tilde{\theta}) (\Hat{\theta}_{QMLE} - \theta_0) ,
\end{align}where $\Tilde{\theta}$ and $\Bar{\theta}$ lies on the segment joining $\Hat{\theta}_{QMLE}$ and $\theta_0$. Since $\Hat{\theta}_{QMLE}$ converges to $\theta_0$ almost surely, $\Tilde{\theta}$ and $\Bar{\theta}$ both converge to $\theta_0$ almost surely.
% we have:
% $$\frac{1}{T} S_T(\Tilde{\theta}) = \frac{1}{T} S_T(\theta_0) + B_T(\Bar{\theta}) (\Tilde{\theta} - \theta_0),$$where $\Tilde{\theta}$ is a tail equivalent to $\Hat{\theta}_{QMLE}$ lying in a convex compact neighbourhood of $\theta_0$ and $\Bar{\theta}$ lies on the segment between $\Tilde{\theta}$ and $\theta_0$. Here, tail equivalence means for each sample path $\omega$, we have $\lim_{T\rightarrow\infty} (\Tilde{\theta} - \Hat{\theta}_{QMLE}) = 0$ almost surely (in fact, $\Tilde{\theta} = \Hat{\theta}_{QMLE}$ holds for $T \geq \tau(\omega)$ ). Since $\Hat{\theta}_{QMLE}$ converges to $\theta_0$ almost surely, we have $\Tilde{\theta}$ and $\Bar{\theta}$ both converge to $\theta_0$ almost surely. 

% Similarly, we have $$0 = h(\Tilde{\theta})=h(\theta_0)+ \nabla h(\ddot{\theta}) (\Tilde{\theta} - \theta_0),$$where $\ddot{\theta}$ lies on the segment between $\Tilde{\theta}$ and $\theta_0$. Again, $\ddot{\theta}$ converges to $\theta_0$ almost surely.
% By the Taylor expansion, 

Under $H_0$: $\theta_0 \in \Theta_0$, we have $h(\theta_0) = 0$. Plug this back into the mean-value expansion \eqref{expansion2} we will get: \begin{equation}\label{expansion2-1}
\nabla h(\Tilde{\theta}) \sqrt{T}(\Hat{\theta}_{QMLE} - \theta_0) = 0.\end{equation}

Multiply \eqref{zeta_condi} by $\sqrt{T}$ and plug the mean-value expansion \eqref{expansion1} into it, we will get:
\begin{equation} \label{expansion1-1}
\begin{split}
    & \frac{1}{\sqrt{T}} S_T(\theta_0) + \frac{1}{T} B_T(\Bar{\theta}) \sqrt{T}(\Hat{\theta}_{QMLE} - \theta_0) + \sqrt{T} \Big(\nabla h(\Hat{\theta}_{QMLE})\Big)^\intercal  \zeta_T = 0,
\end{split}
\end{equation}

Since $B_T(\Bar{\theta})/T  \overset{a.s.}{\to}  B(\theta_0)$, the non-singularity of $B_T(\Bar{\theta})$ directly follows Assumption (B6) in \citet{ogata1978asymptotic} for sufficiently large $T$. Pre-multiplying \eqref{expansion1-1} by $\nabla h(\Tilde{\theta}) B_T^{-1}(\Bar{\theta})$ and plug \eqref{expansion2-1} into it, we will get:
\begin{equation*} \begin{split}
    0=& \nabla h(\Tilde{\theta}) B_T^{-1}(\Bar{\theta}) \left( \frac{1}{\sqrt{T}} S_T(\theta_0) + \frac{1}{T} B_T(\Bar{\theta}) \sqrt{T}(\Hat{\theta}_{QMLE} - \theta_0) + \sqrt{T} \Big(\nabla h(\Hat{\theta}_{QMLE})\Big)^\intercal  \zeta_T \right) \\
    =&  \nabla h(\Tilde{\theta}) B_T^{-1}(\Bar{\theta})\frac{1}{\sqrt{T}} S_T(\theta_0) + \nabla h(\Tilde{\theta}) B_T^{-1}(\Bar{\theta}) \Big(\nabla h(\Hat{\theta}_{QMLE})\Big)^\intercal   \sqrt{T} \zeta_T.
\end{split}\end{equation*}

Note that for our testing problem, since $h(\theta)$ is linear in $\theta$, $\nabla h(\theta)$ does not depend on $\theta$ and has full row rank $r$. We denote this by $H$. It is easy to verify that $H B_T^{-1}(\Bar{\theta}) H^\intercal$ is non-singular for sufficiently large $T$. Thus, pre-multiply $(H B_T^{-1}(\Bar{\theta}) H^\intercal)^{-1}$ and rearrange the terms, we will get:
\begin{equation*} \begin{split}
\sqrt{T} \zeta_T= - \Big(H B_T^{-1}(\Bar{\theta}) H^\intercal\Big)^{-1} H B_T^{-1}(\Bar{\theta})\frac{1}{\sqrt{T}} S_T(\theta_0).
\end{split}\end{equation*}
Note that we have shown that $ S_T(\theta_0)/\sqrt{T}$ is asymptotically normally distributed with covariance matrix $A(\theta_0)$, thus we will have
\begin{equation} \label{asym_zeta}
\sqrt{T} \zeta_T \overset{d}{\to} N \bigg(0,\Big(H B^{-1}(\theta_0) H^\intercal\Big)^{-1} H B^{-1}(\theta_0)A(\theta_0)B^{-1}(\theta_0)H^\intercal\Big(H B^{-1}(\theta_0) H^\intercal\Big)^{-1}\bigg). \end{equation}
We denote this covariance matrix by $Q(\theta_0)$.

Denote\begin{align}
\sqrt{T} \Tilde{\zeta}_T(\theta)= - \Big(H B_T^{-1}(\theta) H^\intercal\Big)^{-1} H B_T^{-1}(\theta)\frac{1}{\sqrt{T}} S_T(\theta).\label{score_zeta}
\end{align}
By 2c.4(x.a) in \citet{rao1973linear}, we will have $$\sqrt{T} \zeta_T-\sqrt{T} \Tilde{\zeta}_T(\theta_0) \overset{p}{\to} 0.$$

Meanwhile, by pre-multiplying \eqref{zeta_condi} by $\Big(H B_T^{-1}(\Hat{\theta}_{QMLE}) H^\intercal\Big)^{-1} H B_T^{-1}(\Hat{\theta}_{QMLE})$ (again the non-singularity holds for sufficiently large $T$), we will have $$\sqrt{T} \zeta_T = \sqrt{T} \Tilde{\zeta}_T(\Hat{\theta}_{QMLE}).$$  Thus, by \eqref{asym_zeta}, we have when $T\rightarrow \infty,$
$$ \sqrt{T} \Tilde{\zeta}_T(\Hat{\theta}_{QMLE})\overset{d}{\to} N \Big(0,Q(\theta_0)\Big). $$

We can easily re-write GS statistic $\Hat{GS}_T$ as a quadratic form of score function $S_T(\Hat{\theta}_{QMLE})$. By the notation we just defined in \eqref{score_zeta} we will have:
\begin{equation*} \begin{split}
       \Hat{GS}_T = \sqrt{T} \Tilde{\zeta}_T^\intercal(\theta) H B^{-1}(\theta) H^\intercal\bigg( H B^{-1}(\theta)\frac{A_T(\theta)}{T}B^{-1}(\theta_0)H^\intercal\bigg)^{-1} H B^{-1}(\theta) H^\intercal \sqrt{T} \Tilde{\zeta}_T(\theta)\Bigg|_{\theta=\Hat{\theta}_{QMLE}},
\end{split}\end{equation*}where the matrix in the middle $$H B^{-1}(\theta) H^\intercal\bigg( H B^{-1}(\theta)\frac{A_T(\theta)}{T}B^{-1}(\theta_0)H^\intercal\bigg)^{-1} H B^{-1}(\theta) H^\intercal\Bigg|_{\theta=\Hat{\theta}_{QMLE}}$$ is a consistent estimator of $Q(\theta_0)$, since $\Hat{\theta}_{QMLE}$ converges to $\theta_0$ almost surely. 

By Lemma 3.3 in \citet{10.2307/1913132}, we can verify the asymptotic $\chi^2$ distribution of our GS statistic.
\end{proof}
%\vspace{-.1in}
\subsection{Proof of Theorem~\ref{asym_power}: asymptotic power under alternative hypothesis}
%\vspace{-.1in}
\begin{proof} We make use of the Generalized Wald (GW) test statistic here, which is asymptotically equivalent to GS statistic under both $H_0$ and $H_1$. More specifically, by 2c.4(xiv) in \citet{rao1973linear} (or Theorem 1 in 13.6 in \citet{engle1984wald}), $$\Hat{GS}_T - \Hat{GW}_T\overset{p}{\to} 0,$$where $\Hat{GW}_T$ is the GW test statistic. We define it as follows:
\begin{equation} \label{GWstats} \begin{split}
 \Hat{GW}_T &= \ h(\theta)^\intercal  \Big(H(\theta)B_T^{-1}(\theta)A_T(\theta)B_T^{-1}(\theta)H(\theta)^\intercal\Big)^{-1} h(\theta)\bigg|_{\theta=\Tilde{\theta}_{QMLE}},
\end{split}\end{equation}where $\Tilde{\theta}_{QMLE}$ is QMLE under $H_1$. 

As we have mentioned above, $h(\theta)$ is linear in $\theta$, thus its first order gradient is a constant matrix, i.e. $H(\theta) = H$. More specifically, $h(\theta) = H \theta $. Then it is not hard to verify the asymptotic normal distribution of $h(\Tilde{\theta}_{QMLE})$ based on asymptotically normality of $\Tilde{\theta}_{QMLE}$. That is
\begin{equation*} \begin{split}
   \sqrt{T} \Big(h(\Tilde{\theta}_{QMLE}) - h(\theta_0)\Big) \overset{d}{\to} N \Big(0,H B^{-1}(\theta_0)A(\theta_0)B^{-1}(\theta_0)H^\intercal\Big) \ \ \ \text{  as  } \ \ \ T \overset{}{\to} \infty. 
\end{split}\end{equation*}Then the noncentral $\chi^2$ distribution of $\Hat{GW}_T$ as well as $\Hat{GS}_T$ directly follow. 

Since $$\theta_0 ^\intercal H^\intercal H \theta_0 = \Big(\alpha^{(1)}-\alpha^{(2)}\Big)^2 +\sum_{k=1}^{n_0}\Big(g_k^{(1)}-g_{k}^{(2)}\Big)^2 = \norm{\phi^{(1)}-\phi^{(2)}}_2^2,$$ and $H$ is of rank $r$, the noncentrality parameter is $T \norm{\phi^{(1)}-\phi^{(2)}}_2^2$ and the degree of freedom is $r$. Thus, we get that the asymptotic power function is Marcum-Q-function. 
\end{proof}
% \subsection{Another Proof of Consistency of Generalized Score Test}\label{power}
\begin{proof}[Another proof of consistency of GS test] We can re-express GW test statistic as:
\begin{equation*} \begin{split}
 \Hat{GW}_T &= T \ h(\theta)^\intercal  \left(H(\theta)  \Big(\frac{B_T(\theta)}{T}\Big)^{-1} \frac{A_T(\theta)}{T} \Big(\frac{B_T(\theta)}{T}\Big)^{-1} H(\theta)^\intercal\right)^{-1} h(\theta)\Bigg|_{\theta=\Tilde{\theta}_{QMLE}}.
\end{split}\end{equation*}

From Lemma \ref{asym_qmle} which we just prove, we have (i) $h(\Tilde{\theta}_{QMLE}) \rightarrow h(\theta_0) \not= 0$ almost surely, where the last inequality comes form $H_1: \theta_0 \not \in \Theta_0$; and (ii) $A_T(\Tilde{\theta}_{QMLE})/T$, $B_T(\Tilde{\theta}_{QMLE})/T$ converges to $A(\theta_0)$, $B(\theta_0)$ almost surely. Thus, we can verify $$\Hat{GW}_T \rightarrow \infty \ \ \  \text{ as } \ \ \  T\rightarrow  \infty.$$
Thus, we have $$\Hat{GS}_T \rightarrow \infty \ \ \  \text{ as } \ \ \  T\rightarrow  \infty,$$which indicates the unit asymptotic power of the proposed GS test, i.e. this test is consistent.
\end{proof}

%\vspace{-.1in}

\section{Numerical experiments}\label{add_exp}

\subsection{Testing details}

The testing procedures are detailed in Algorithm~\ref{algo_goodness}. We specify the data sequence sets we use, the initialization and other experiment configurations in Algorithm~\ref{algo_goodness} here for all experiments above.

\textbf{Validation of asymptotic properties in Section~\ref{expsec:validation}}: (a) For each $\alpha \in \{1.5,2,2.5,3,3.5\}$, generate $L$ data sequences as $D_1$ and another $L$ data sequences as $D_2$; (b) Generate $L$ data sequences from $\alpha=1$ as $D_1$ and another $L$ data sequences as $D_2$ from $\alpha=4$; (c) Use the first pair of data sequence set in (a) (corresponding to $\alpha=1.5$) as positive sample and data sequence set in (b) as the negative sample. For experiments in Sections~\ref{expsec:diff_n0} and \ref{expsec:comparison}, data generation mechanisms for $D_1$ and $D_2$ are the same.

The experiment configurations (initialization) are as follows: $L=1,000$, $n_0=14$ and endpoints for those bins are $(0,.04,.08,.12,.16,.2,.26,.32,.38,.45,.55,.65,.75,1,2)$ for all experiments. (a) $N=200$, $K = 20$; (b) $N \in \{50,150,\dots,850\}$, $K = 5$; (c) $N\in\{25,50,100\}$, $K = 150$.

For the experiment on how $n_0$ influences our proposed test, the endpoints for bins with $n_0 = 2,3,4,7, 14, 28$ are \\
$(0, .6, 2)$, 
$(0,.2,.6,2)$, $(0,.1,.2,.6,2)$, $(0, .08, .16, .2, .32, .45, .65, 2)$, $(0,
 .04,
 .08,
 .12,
 .16,
 .2,
 .26,
 .32,
 .38,
 .45,
 .55,
 .65,
 .75,
 $ $1,
 2)$ and $(0,
 .02,
 .04,
 .06,
 .08,
 .1,
 .12,
 .14,
 .16,
 .18,
 .2,
 .23,
 .26,
 .29,
 .32,
 .35,
 .38,
 .41,
 .45,
 .5,
 .55,
 .6,
 .65,
 $ $.7,
 .75,
 .8,
 1,
 1.5,
 2)$, respectively. We use $L=1,000$ sequences in computing $\hat GS_T$.

\textbf{Goodness-of-fit in Section~\ref{expsec:GOF}}: $D_1$ is chosen to be the testing data and $D_2$ is generated from the model fitted on the training data. The endpoints of bins are

(i) $(0,
 .02,
 .04,
 .06,
 .08,
 .1,
 .12,
 .14,
 .16,
 .18,
 .2,
 .25,
 .3,
 .35,
 .4,
 .5)$ for \texttt{Exp} and \texttt{Matern} data;
 
 (ii) $(0,
 .02,
 .04,
 .06,
 .08,
 .1,
 .12,
 .14,
 .16,
 .18,
 .2,
 .5,
 .6,
 .8,
 1)$ for \texttt{MIMIC} data; 
 
 (iii) $(0,
 .05,
 .1,
 .15,
 .2,
 .25,
 .3,
 .35,
 .4,
 .45,
 .5,
 .6,
 .8,
 1)$ for \texttt{MEME} data.

For 911 call data, $L=364$ and we use the first 200 sequences to as the training data to fit the model and the rest 164 sequences as $D_1$. Then we generate 164 data sequences as $D_2$ to perform the testing procedure. We choose $N=20$, $K=1$ and  use $(0, .02, .04, .06, .08, .1, .12, .14, .16, .18, .2, .5, 1)$ as endpoints for bins.

\subsection{Additional experiments}

\textbf{Validation of our proposed method as an model free approach.
} We use different synthetic data to validate our theoretical results. Here, the triggering function used to generate synthetic data is power function (which is commonly used in seismology) : $\phi(t) =  \alpha (P-1)c^{P-1}(t+c)^{-P}\mathbf{1}_{\{t> 0\}}$ with parameters $\mu= 20, \alpha = 0.2, C = 2, P = 13,14,\dots,17$. The experiment configurations are as follows: $L=1,000$, $n_0=12$ and endpoints for those bins are $(0, .04, .08, .12, .16, .2, .24, .28, .32, .36, .4, .7, 2)$ for all experiments. (a) $N=200$, $K = 20$; (b) $N \in \{50,150,\dots,850\}$, $K = 5$; (c) $N\in\{50,100,200\}$, $K = 150$. See the results in Figure~\ref{exp1_2}.

\begin{figure}[htp]
  \centering
\subfigure{\includegraphics[scale=0.45]{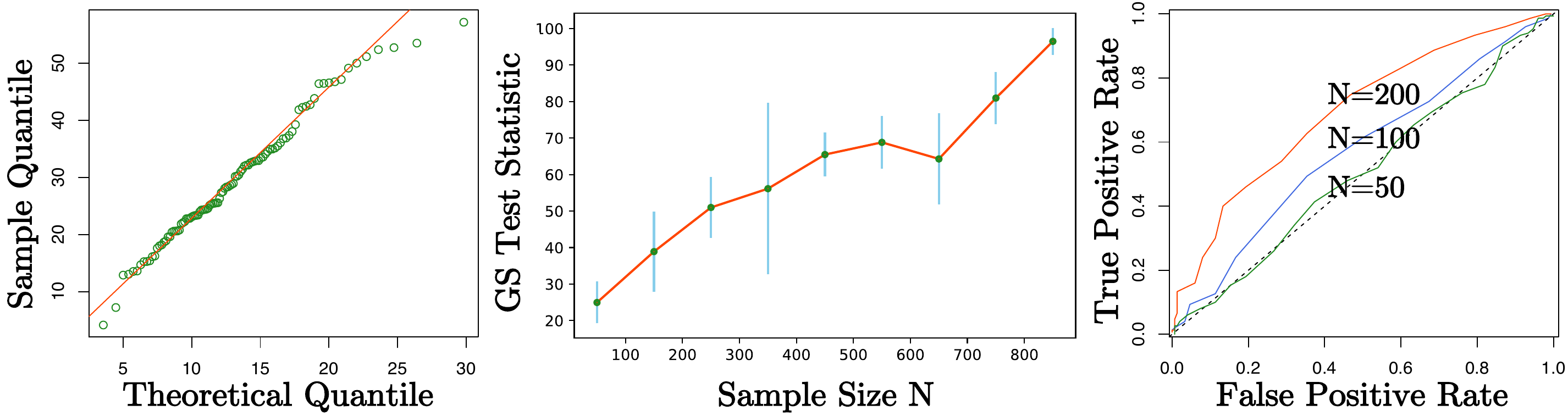}}
        \vspace{-.15in}
    \caption{ Simulation results: (a) Quantiles of calculated scores against theoretical quantiles of $\chi^2_{n_0+1}$ distribution under $H_0$; (b) mean and variance of scores with increasing $N$ under $H_1$; (c) ROC curve for different $N$.}\label{exp1_2}
    
\end{figure}

\textbf{Comparison with Ripley's K function.} See Figure~\ref{exp1_comparison}.
\begin{figure}[!htp]
  \centering
  \vspace{-.1in}
  \subfigure{\includegraphics[clip, trim=.05cm .03cm .05cm .03cm, width=0.153\textwidth]{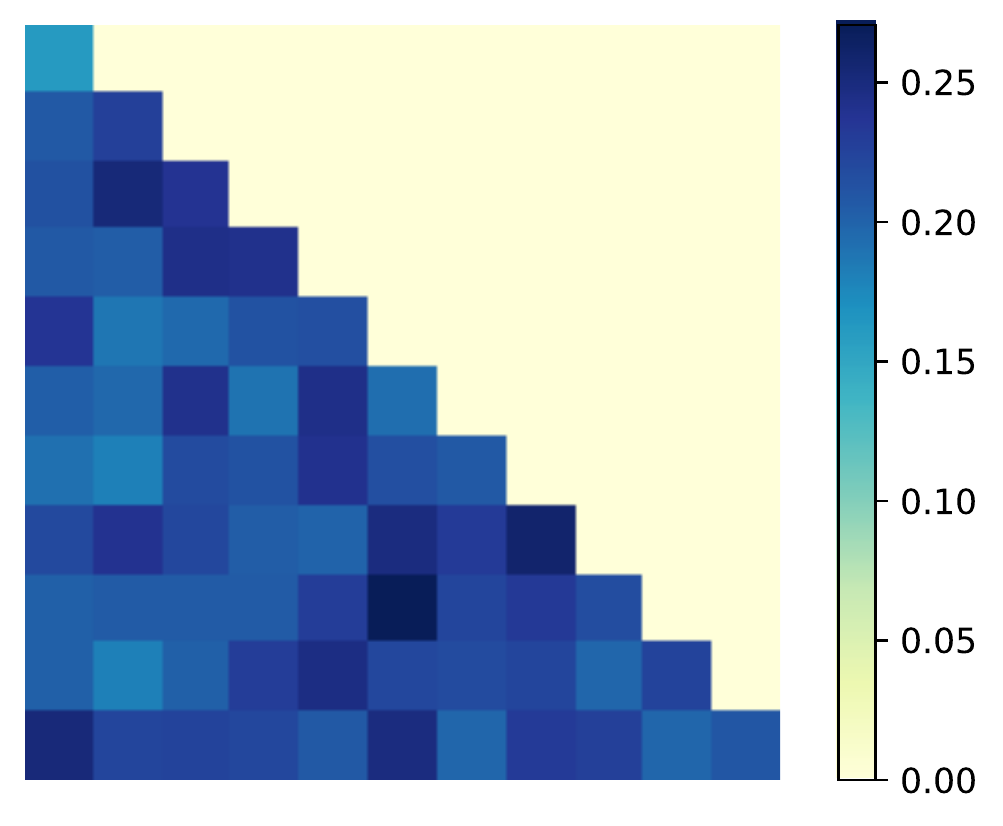}}  \subfigure{\includegraphics[clip, trim=.05cm .03cm .05cm .03cm, width=0.153\textwidth]{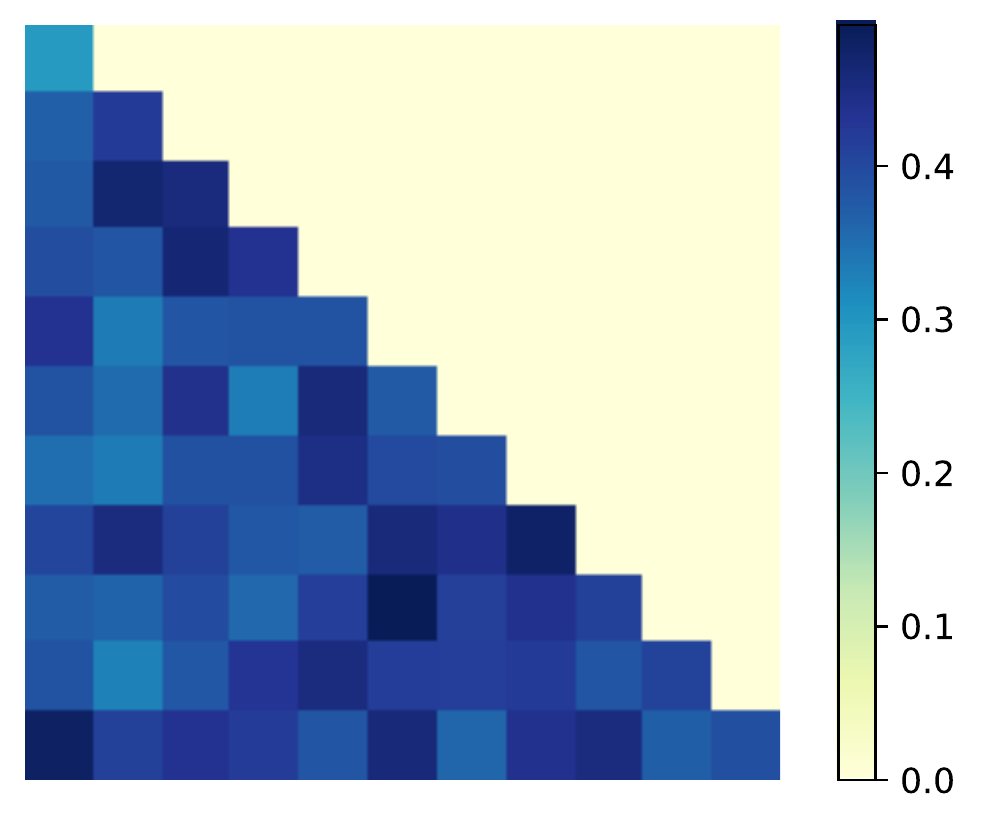}}  \subfigure{\includegraphics[clip, trim=.05cm .03cm .05cm .03cm, width=0.153\textwidth]{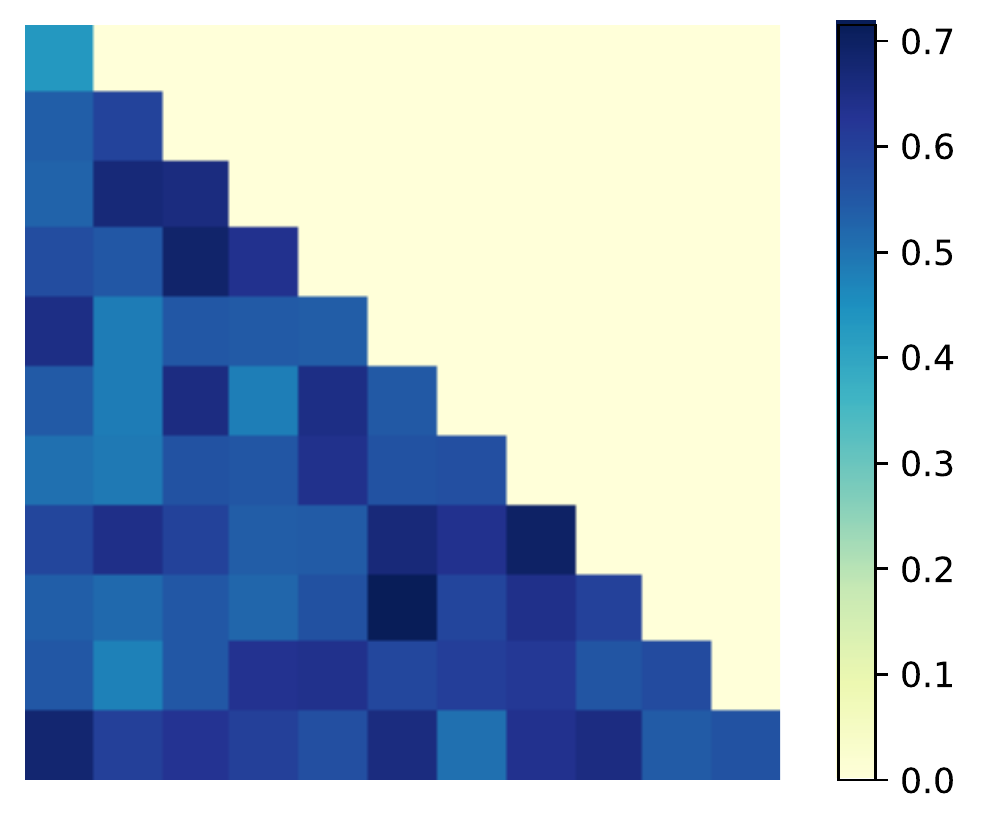}}  \subfigure{\includegraphics[clip, trim=.05cm .03cm .05cm .03cm, width=0.153\textwidth]{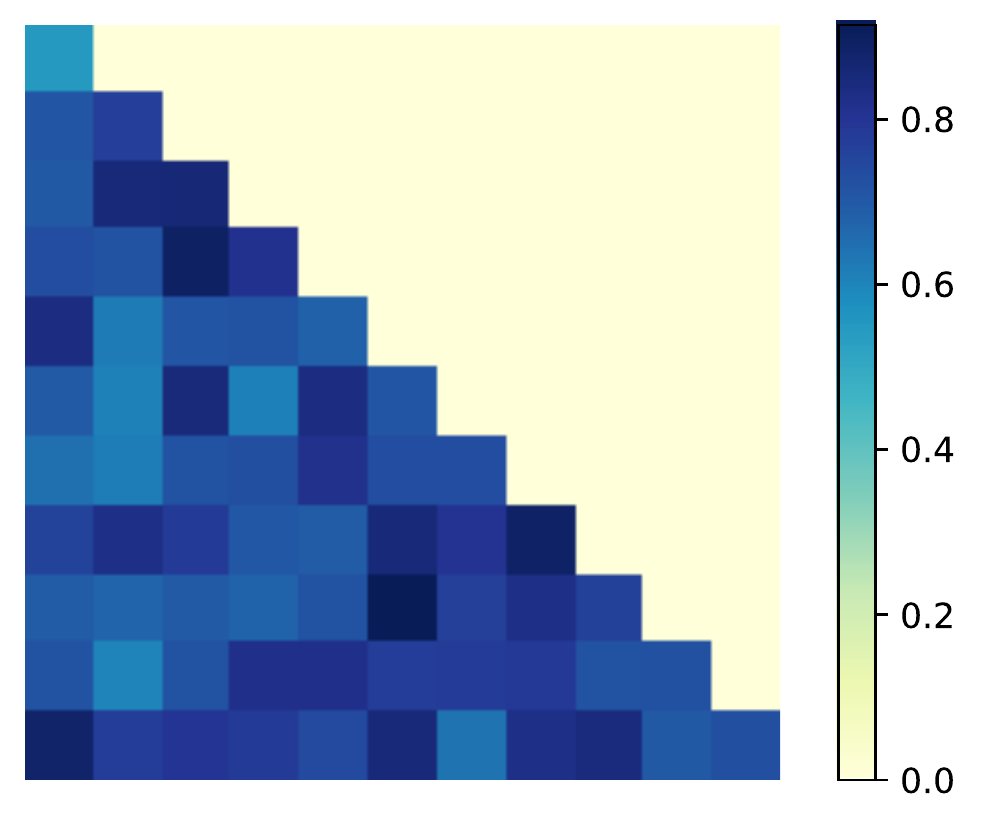}}  \subfigure{\includegraphics[clip, trim=.05cm .03cm .05cm .03cm, width=0.153\textwidth]{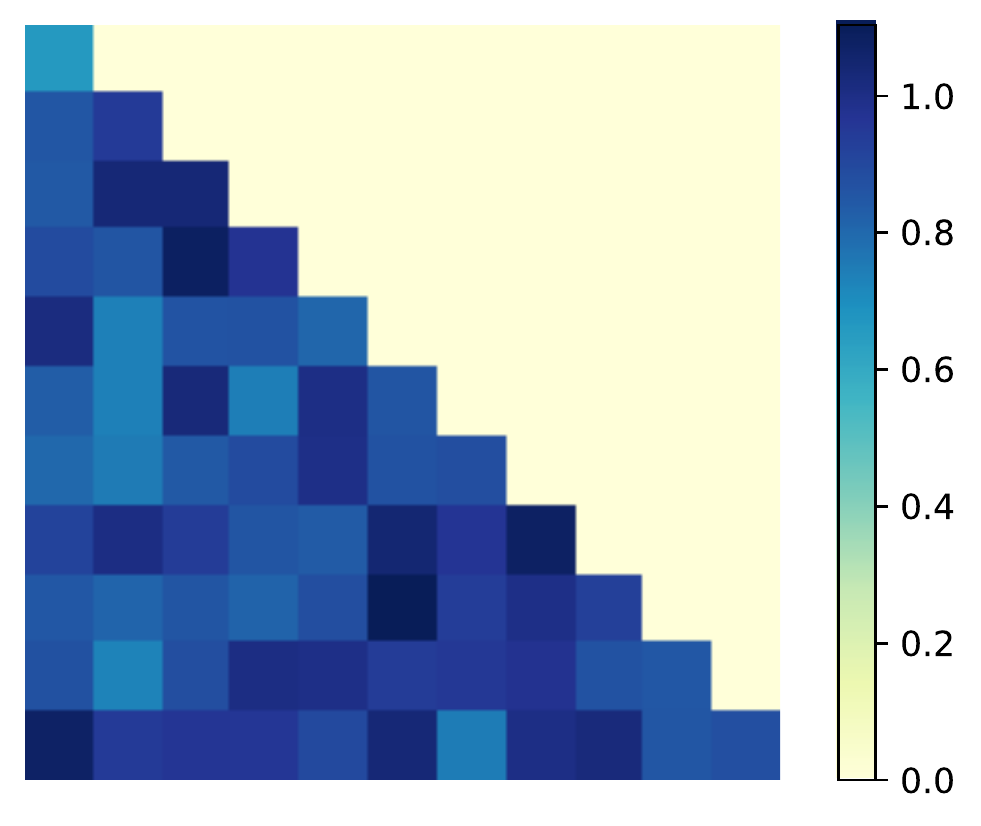}}  
  
    \subfigure{\includegraphics[clip, trim=.05cm .03cm .05cm .03cm, width=0.153\textwidth]{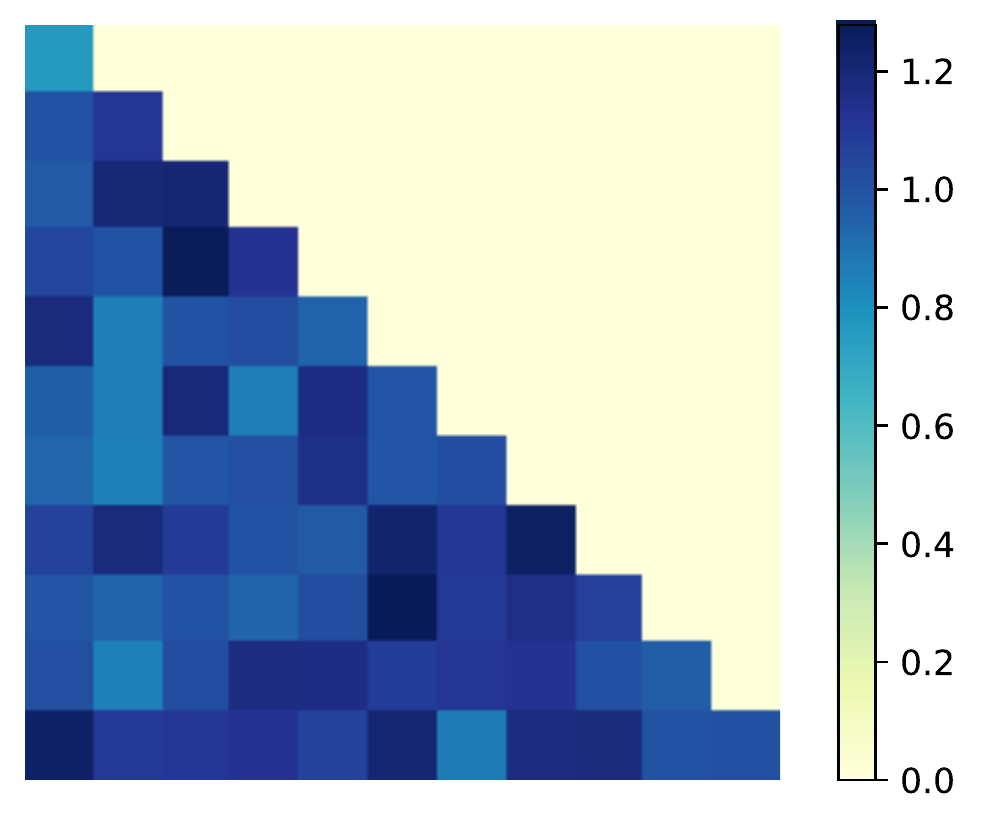}}  \subfigure{\includegraphics[clip, trim=.05cm .03cm .05cm .03cm, width=0.153\textwidth]{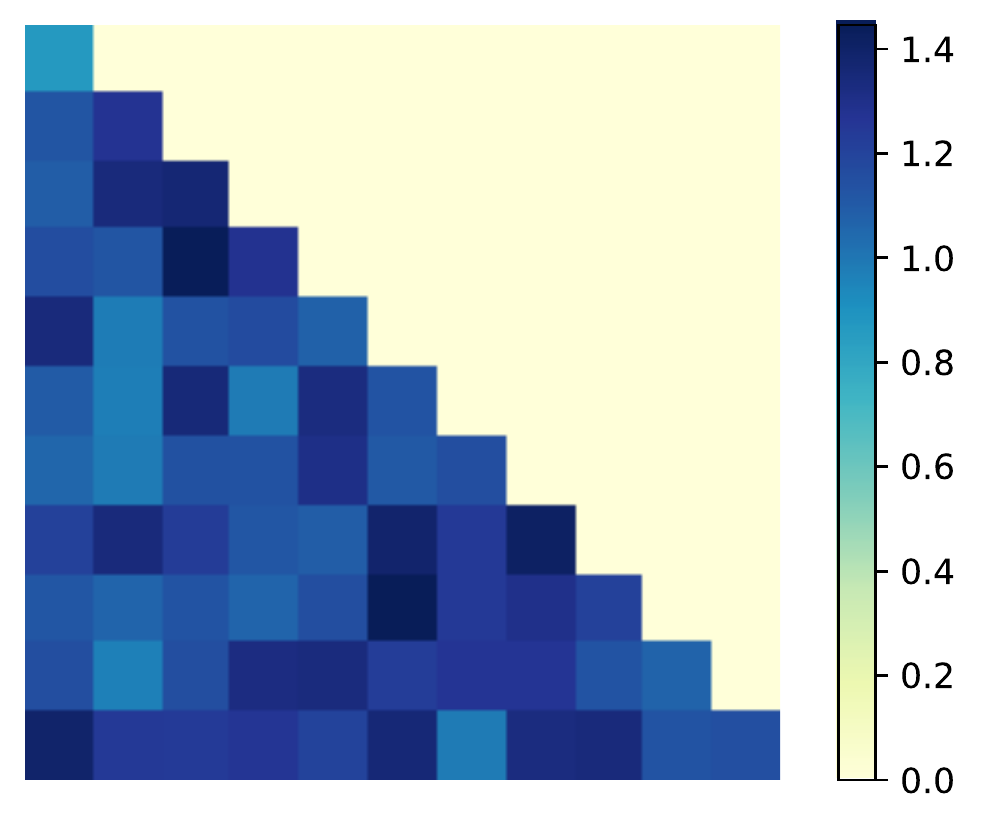}}  \subfigure{\includegraphics[clip, trim=.05cm .03cm .05cm .03cm, width=0.153\textwidth]{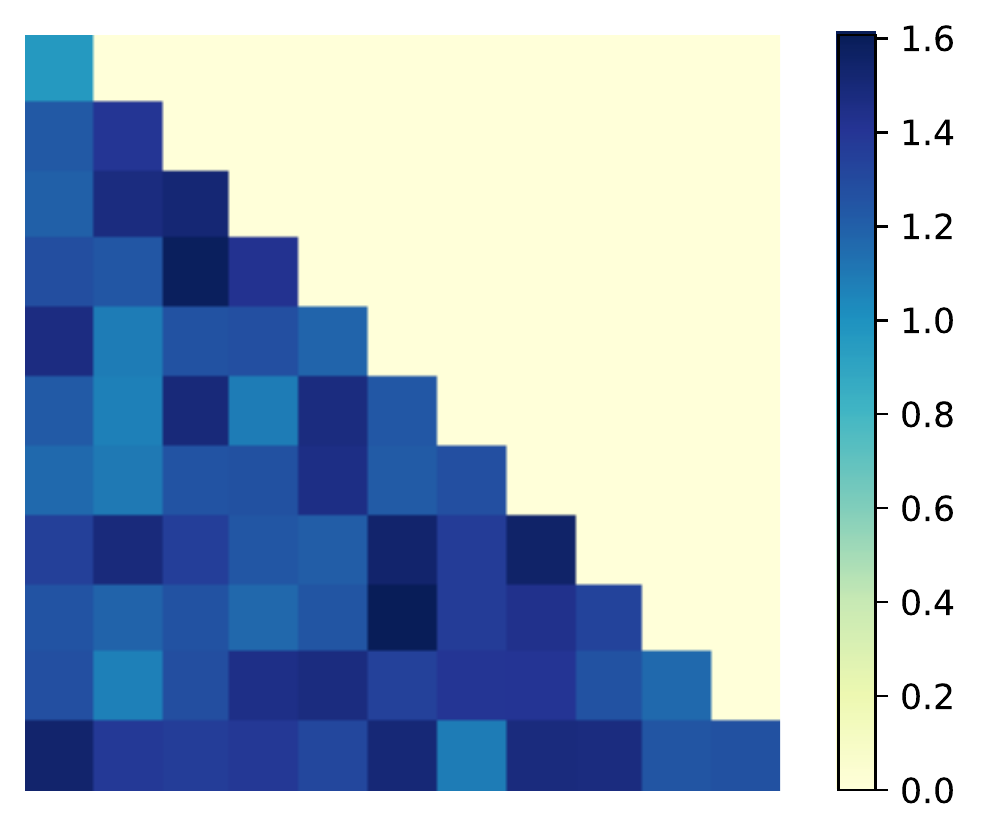}}  \subfigure{\includegraphics[clip, trim=.05cm .03cm .05cm .03cm, width=0.153\textwidth]{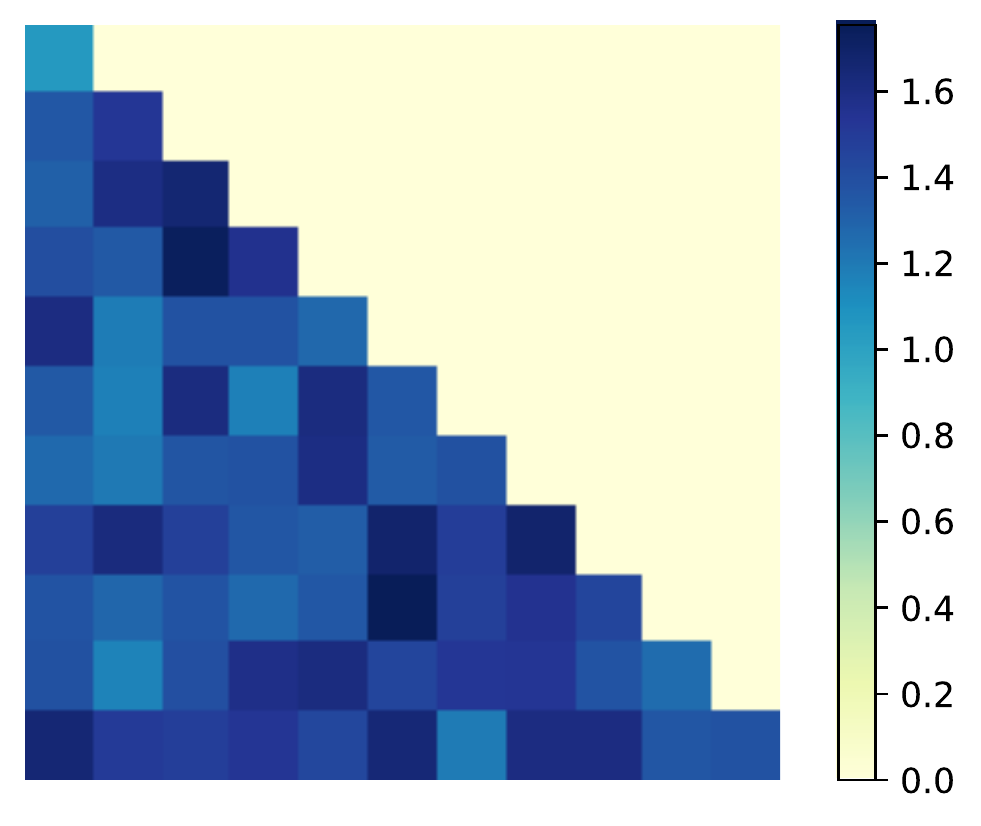}}  \subfigure{\includegraphics[clip, trim=.05cm .03cm .05cm .03cm, width=0.153\textwidth]{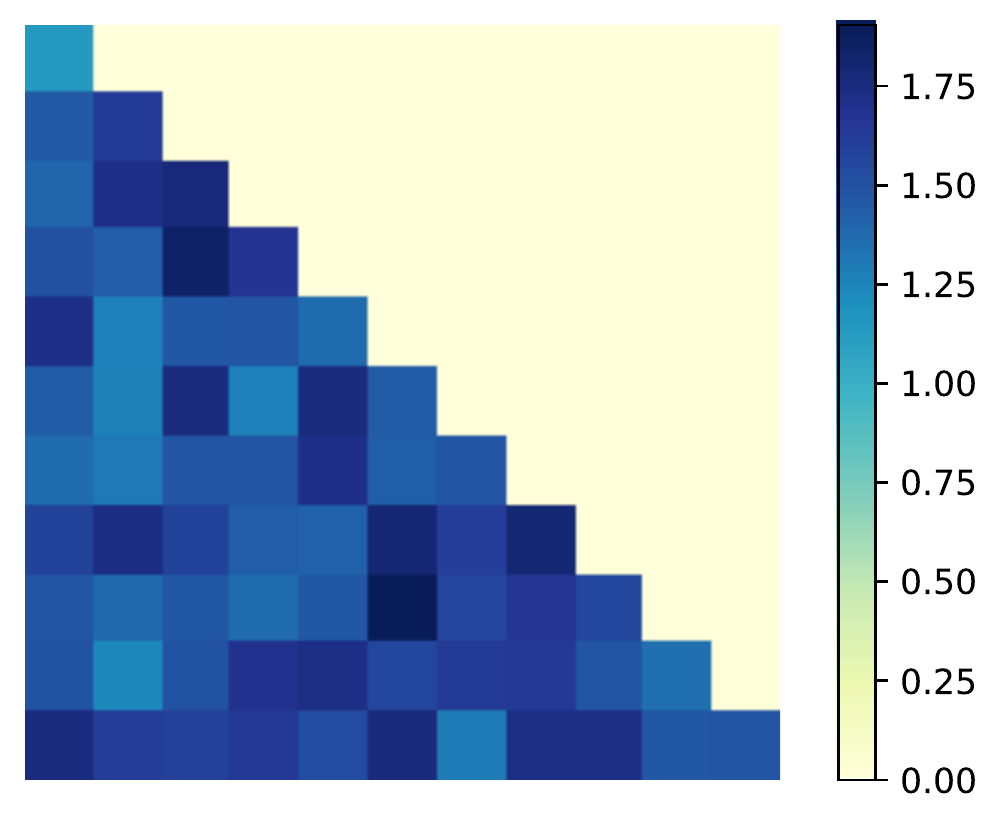}} 
  \vspace{-0.15in}

\caption{Heatmap of estimated Ripley's K function value $\hat K(t)$ for $t=1,\dots,5$ (top), $t=6,\dots,10$ (bottom). For each pixel, the data sequence $D_1$ and $D_2$ are generated from the same distribution as in Figure~\ref{exp1_3} (a).}\label{exp1_comparison}
\vspace{-.05in}
\end{figure}

\textbf{Algorithmic behavior of \texttt{Exp GD} method on \texttt{Exp} data.} In our experiment, we saw a very interesting phenomenon --- no matter where we initialize $\hat \alpha$, using GD to maximize log-likelihood under correct model specification would yield very biased estimate. 

As illustrated in Figure~\ref{fig:overfitting}, we observe that when $\hat \alpha$ is around the ground-truth 1, the log-likelihood is very large. But it keeps growing larger when $\hat \alpha$ keeps decreasing. The same is also true for $\hat \beta$. We can see that even though we got very large log-likelihood, the estimate is very biased. Clearly, overfitting occurs here --- we only gain very little log-likelihood increment but the estimates are getting futther away from the ground-truth. Therefore, using log-likelihood as GOF would be questionable.

\begin{figure}[htp]
  \centering
\subfigure{\includegraphics[scale=0.35]{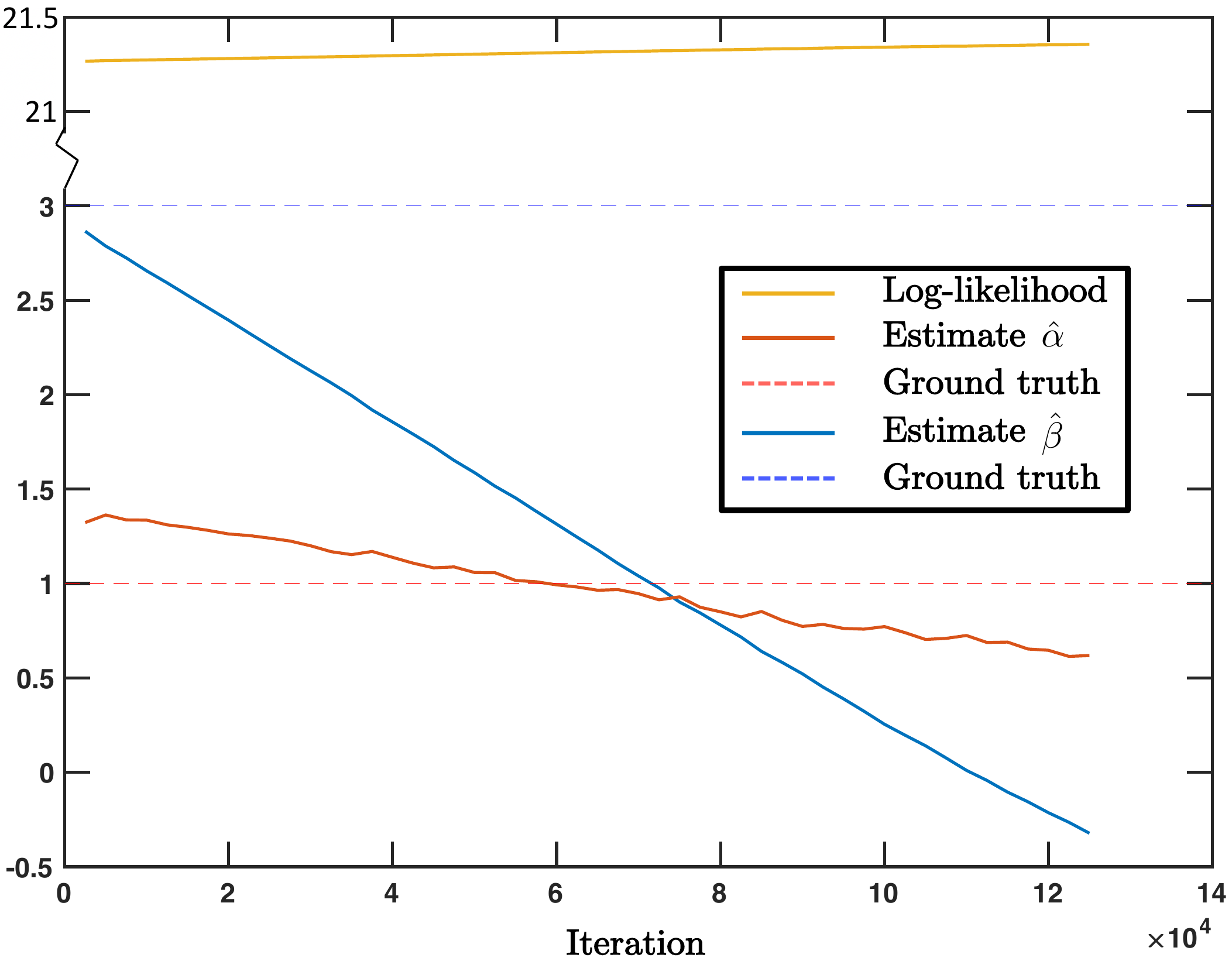}}
        %\vspace{-.15in}
\caption{Algorithmic behavior illustration of \texttt{Exp GD} method on \texttt{Exp} data.}\label{fig:overfitting}
    %%\vspace{-.15in}
\end{figure}

\textbf{Probability weighted histogram estimation under $H_1$.} As one may see from our proof in Appendix~\ref{all_asym},  GS and GW tests are asymptotically equivalent. One would ask why we choose GS test over GW test. 
% It is easy to get the PWHE under full model and derive it as an EM algorithm. Thus, it outputs $\Tilde{\theta}_{QMLE}$. The asymptotic results for GW test directly follows from its asymptotic equivalency to GS test. We defer the details on the algorithm into Appendix \ref{PWHE_full}.
The reason is two-fold. Firstly, it is not computationally efficient, since using GW test involves estimating $r$ more parameters. Secondly and most importantly, its power is far less than GS test. That's because, in empirical study, the QMLE $\Tilde{\theta}_{QMLE}$ does maximize the full model Quasi-likelihood but fails to differentiate two different triggering components, which makes $\norm{h(\Tilde{\theta}_{QMLE})}_2$ much smaller than $\norm{h(\theta_0)}_2$. We will further illustrate this by performing the estimation of the full model (Algorithm~\ref{algo2} in next section) and visualizing the estimation of triggering function as follows:

\begin{figure}[!htp]
  \centering
\subfigure{\includegraphics[scale=0.2]{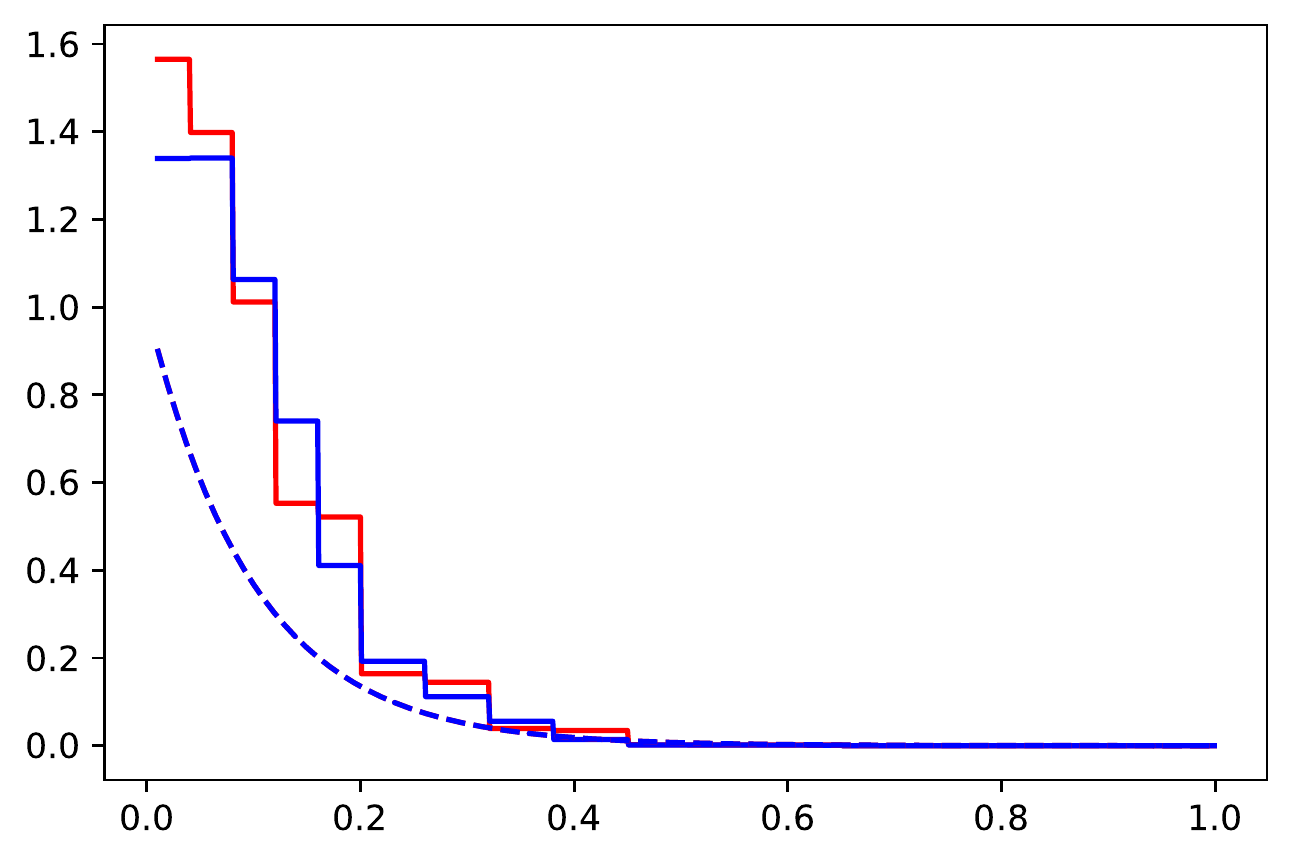}}
\subfigure{\includegraphics[scale=0.2]{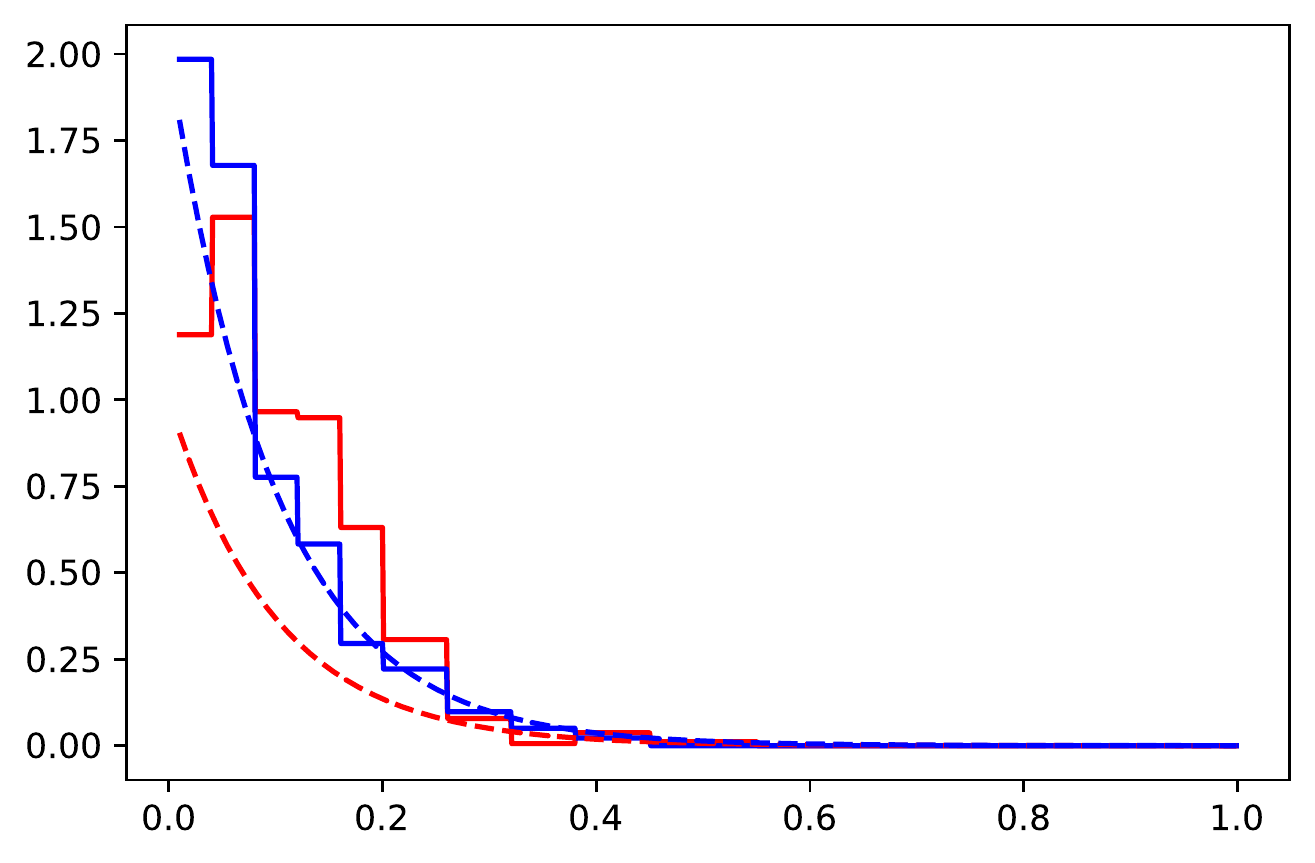}}
\subfigure{\includegraphics[scale=0.2]{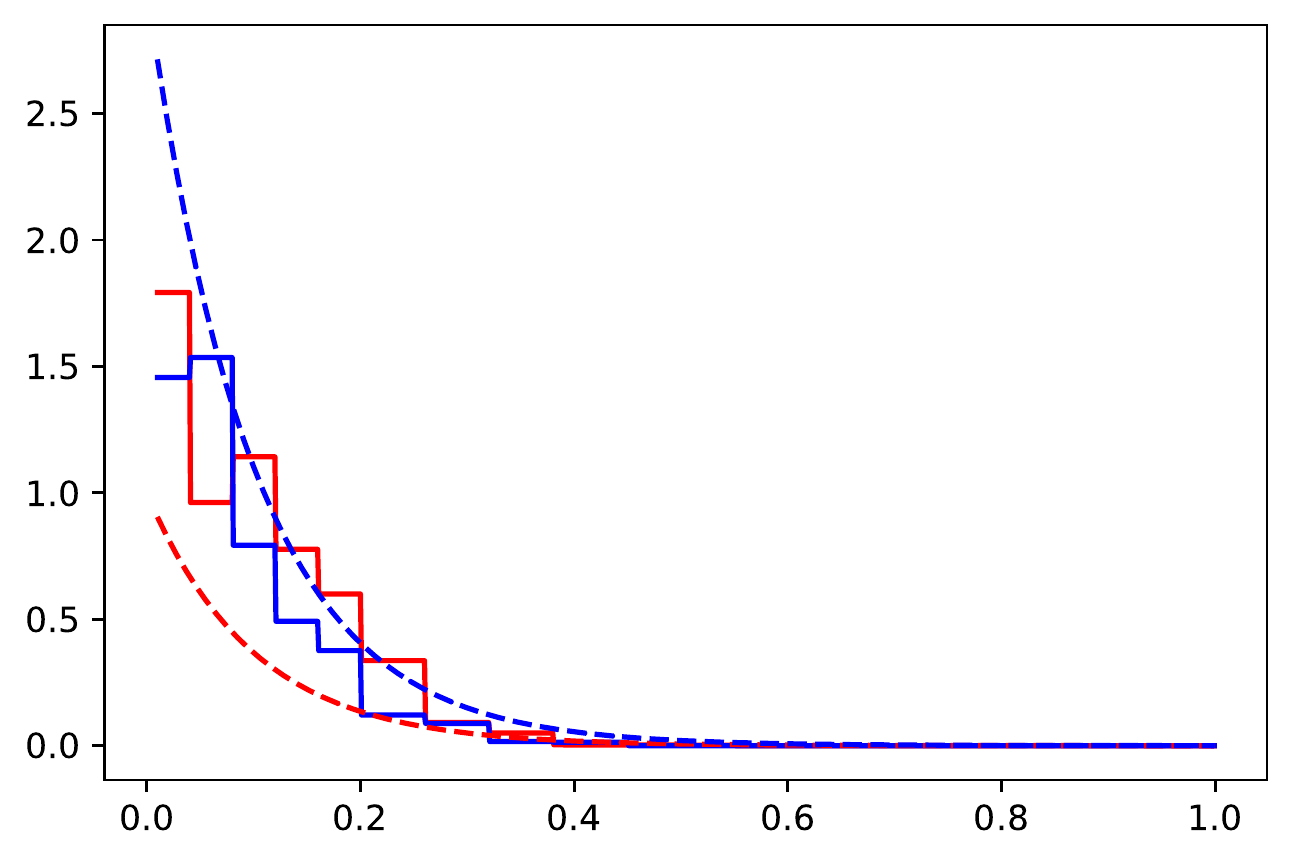}}
\subfigure{\includegraphics[scale=0.2]{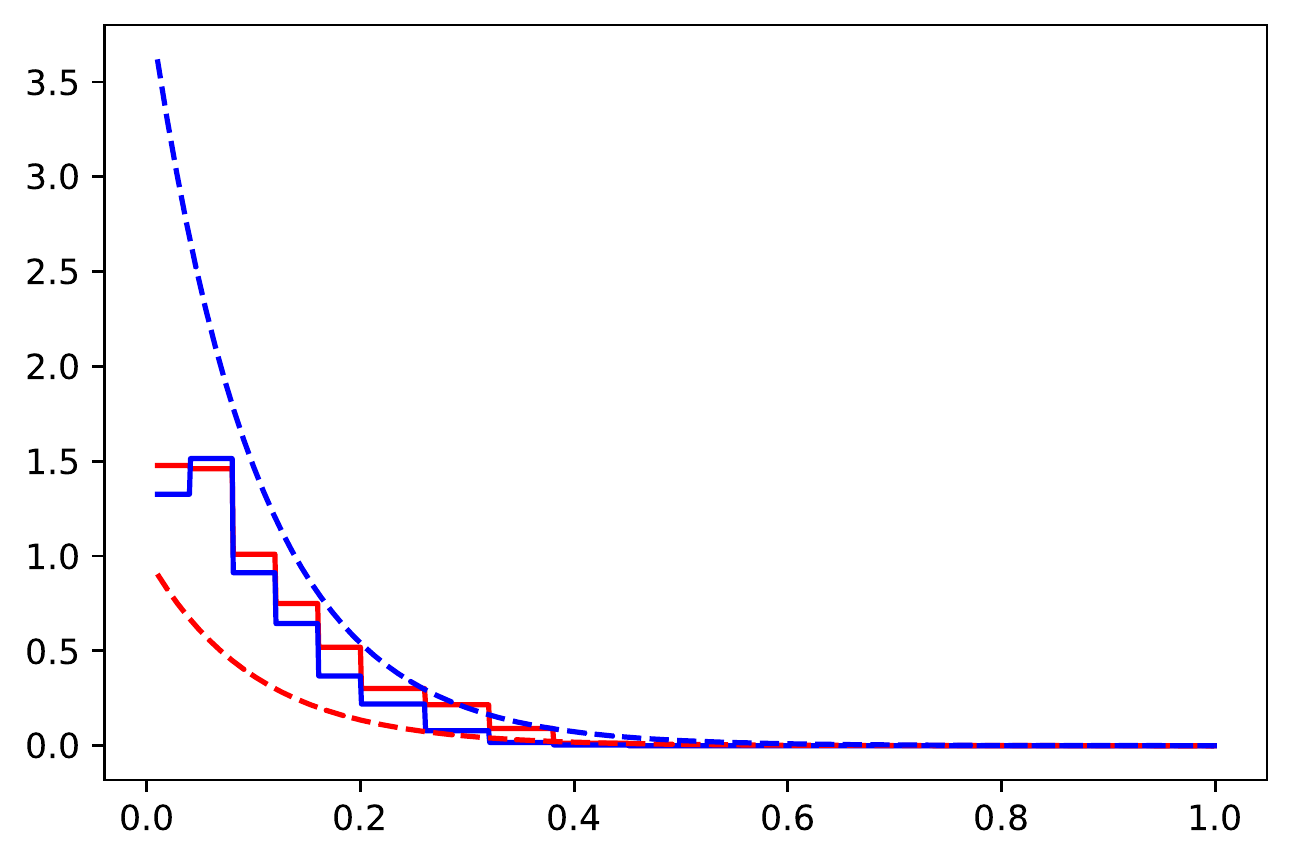}}
\subfigure{\includegraphics[scale=0.2]{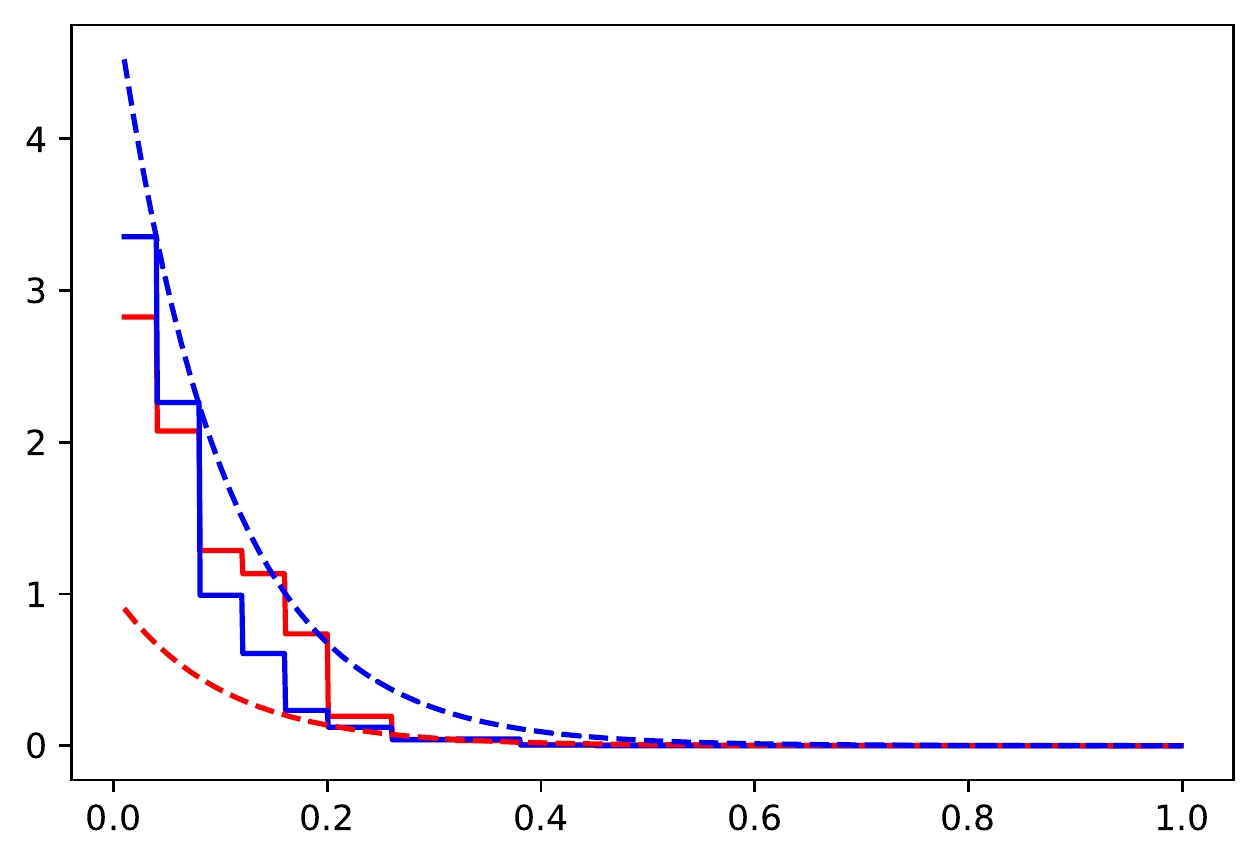}}
 
\subfigure{\includegraphics[scale=0.2]{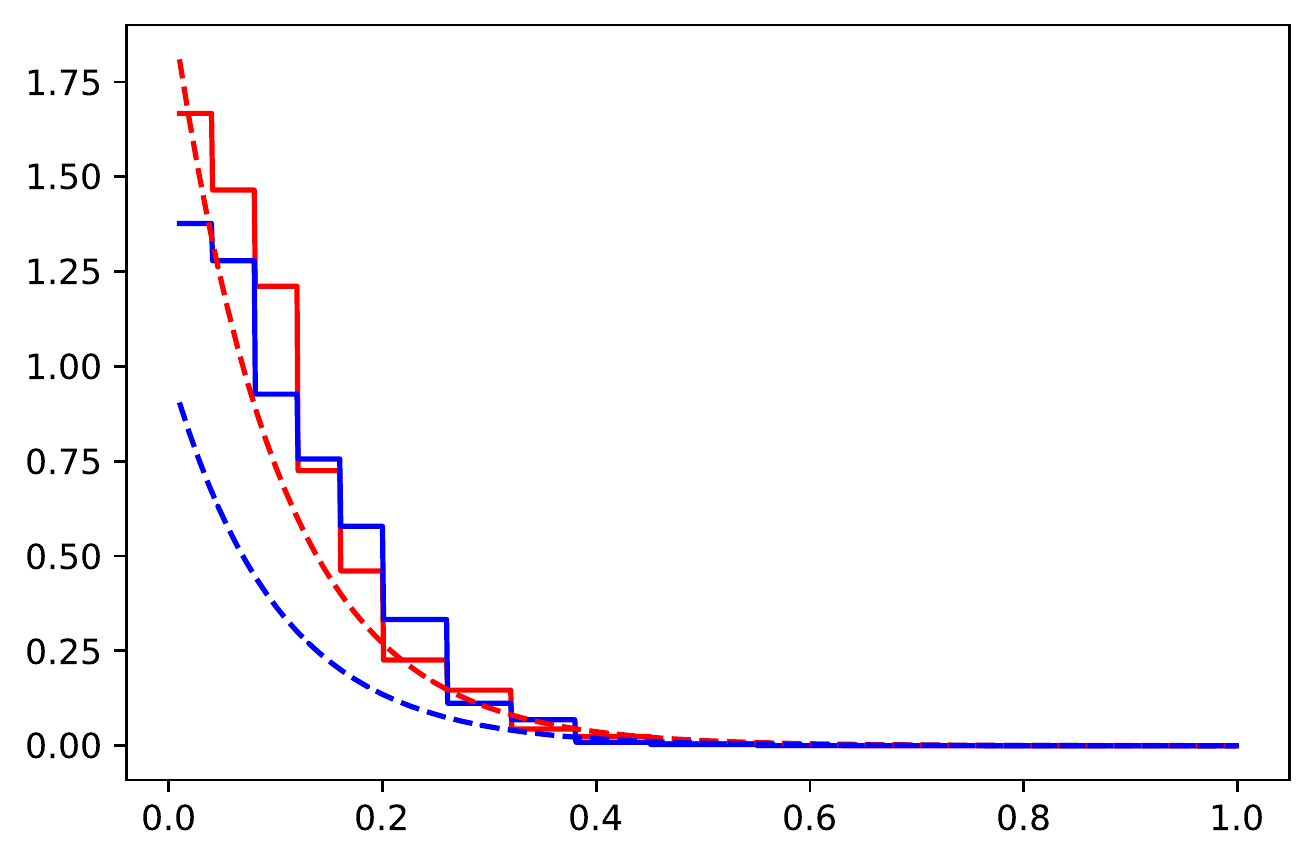}}
\subfigure{\includegraphics[scale=0.2]{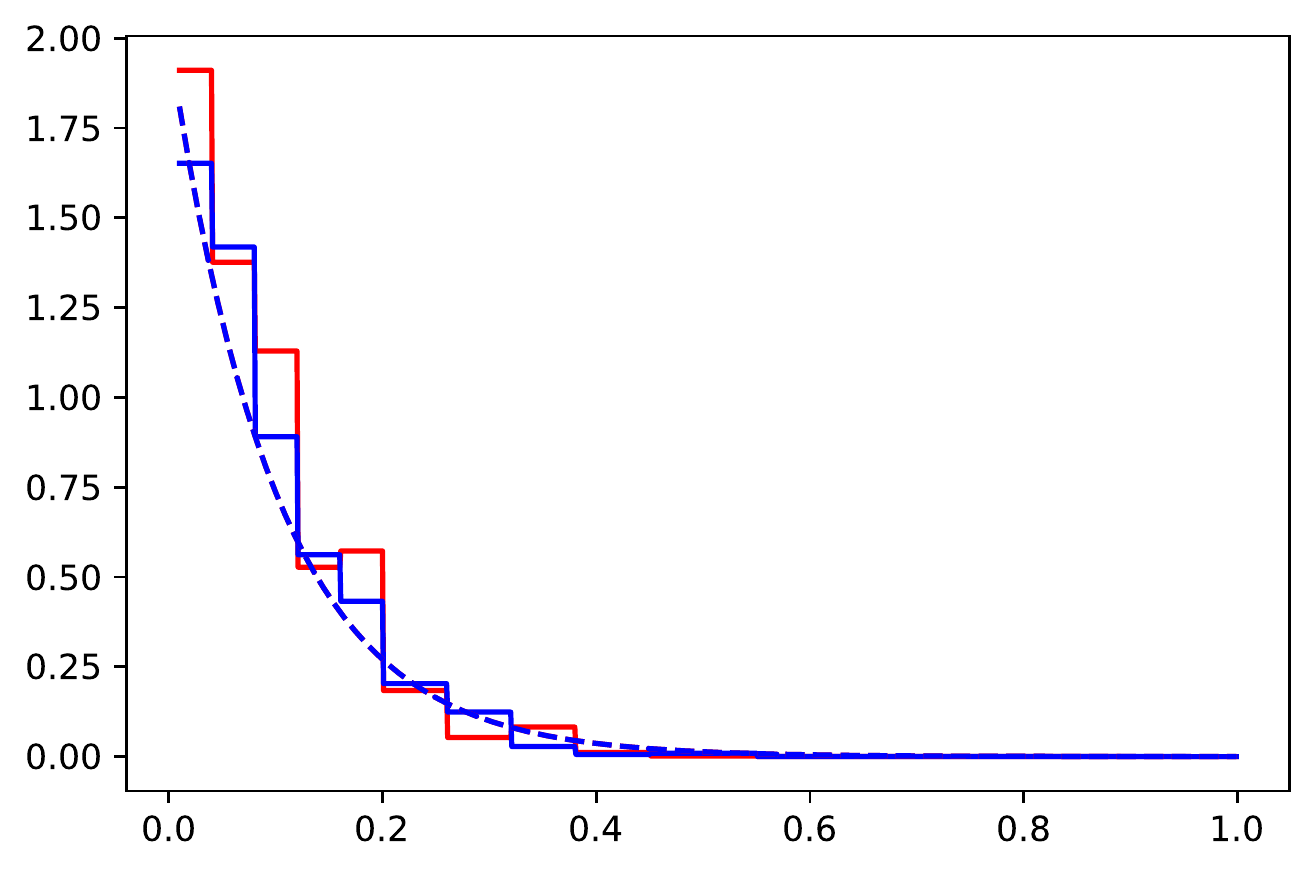}}
\subfigure{\includegraphics[scale=0.2]{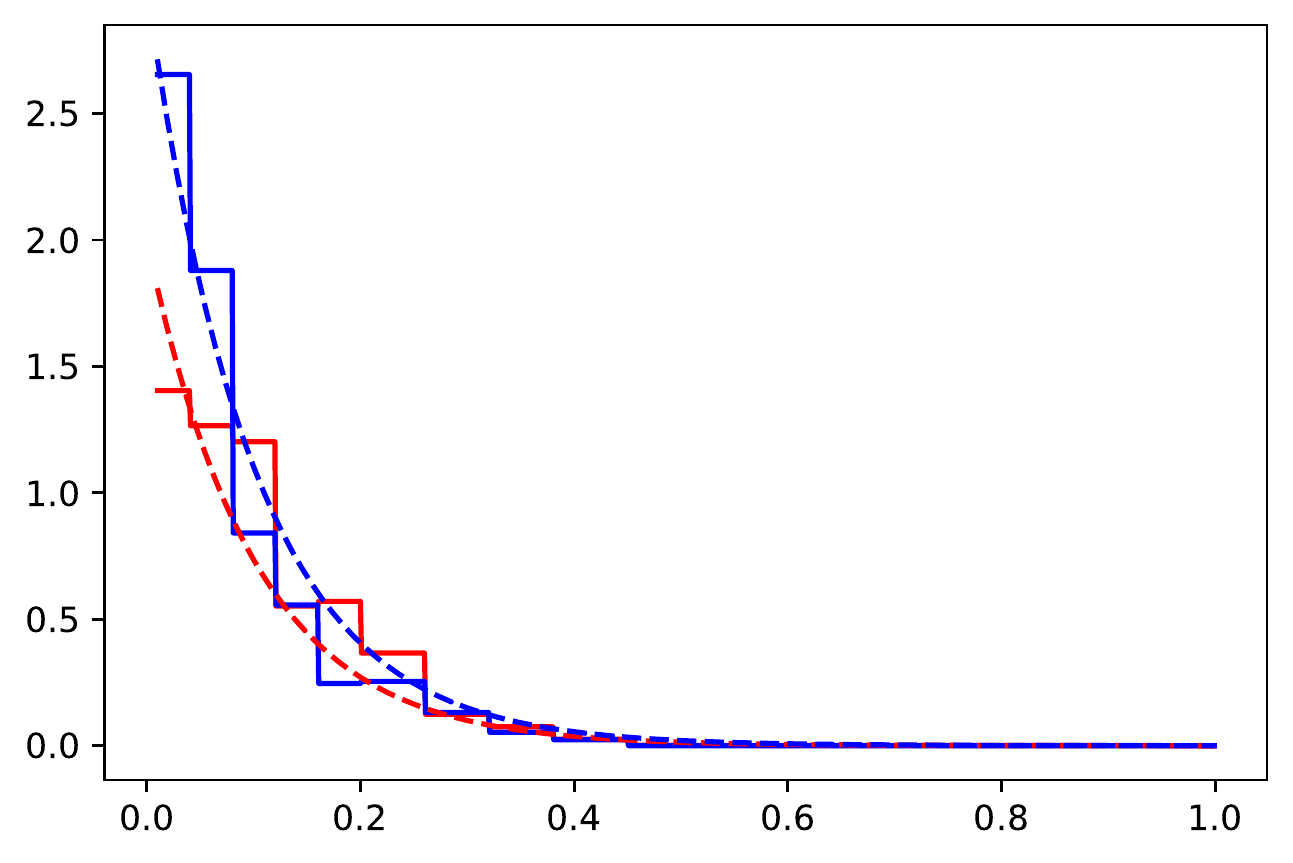}}
\subfigure{\includegraphics[scale=0.2]{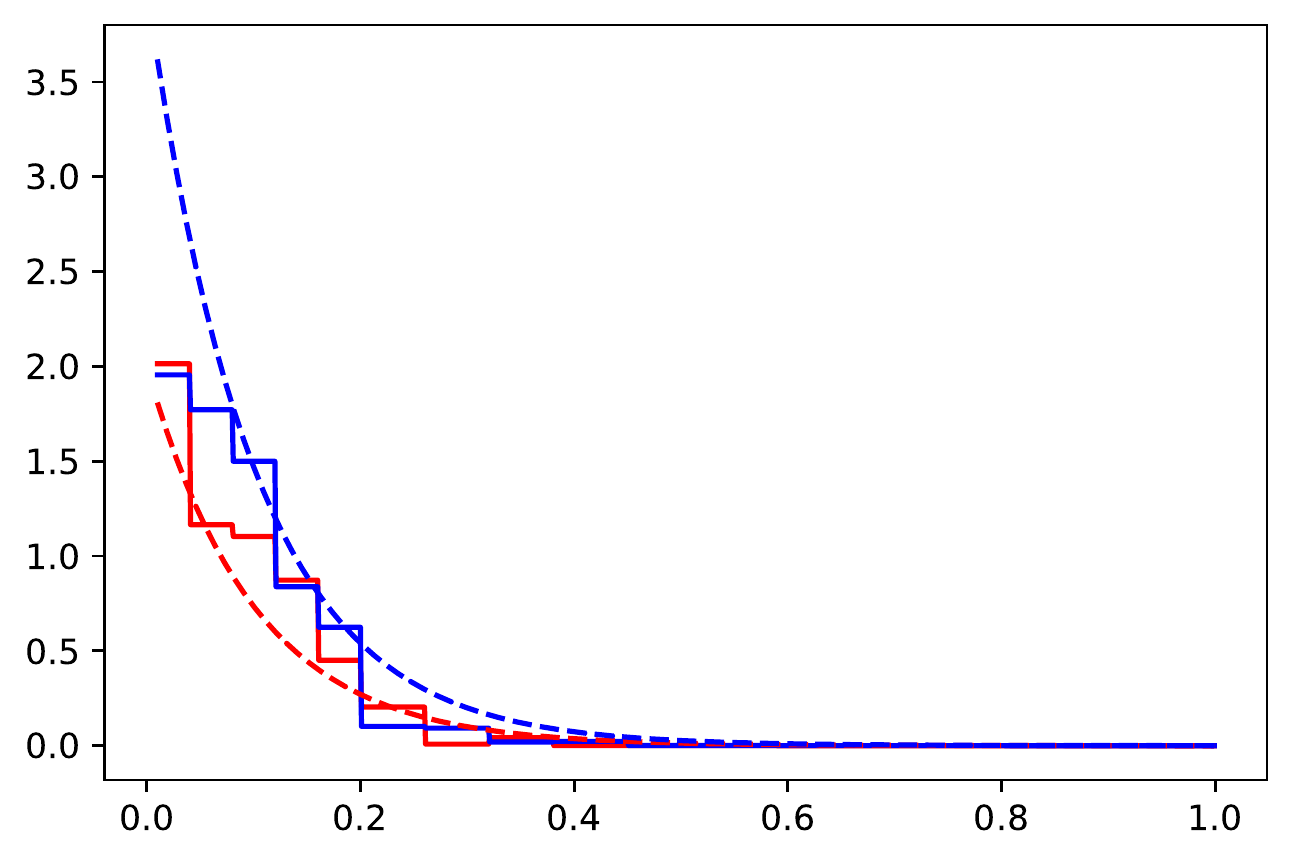}}
\subfigure{\includegraphics[scale=0.2]{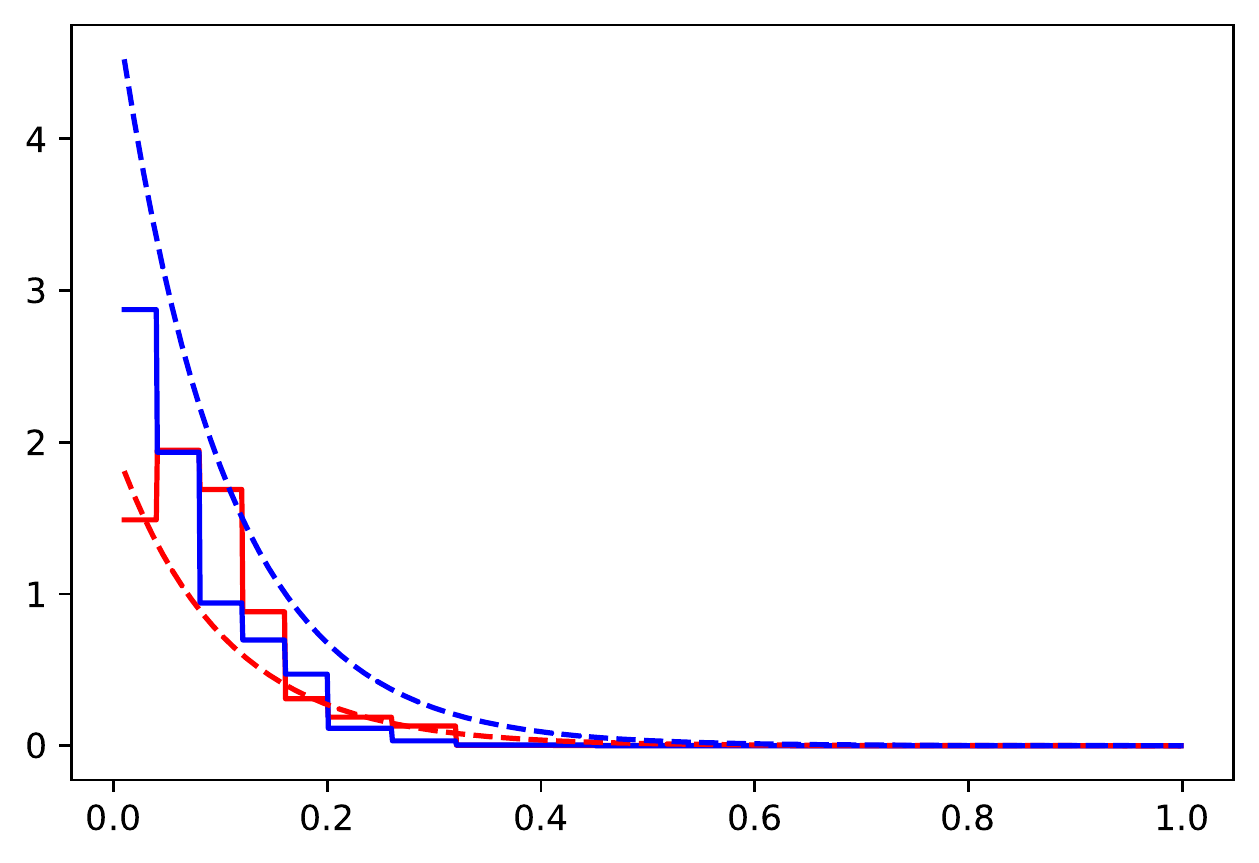}}
 
\subfigure{\includegraphics[scale=0.2]{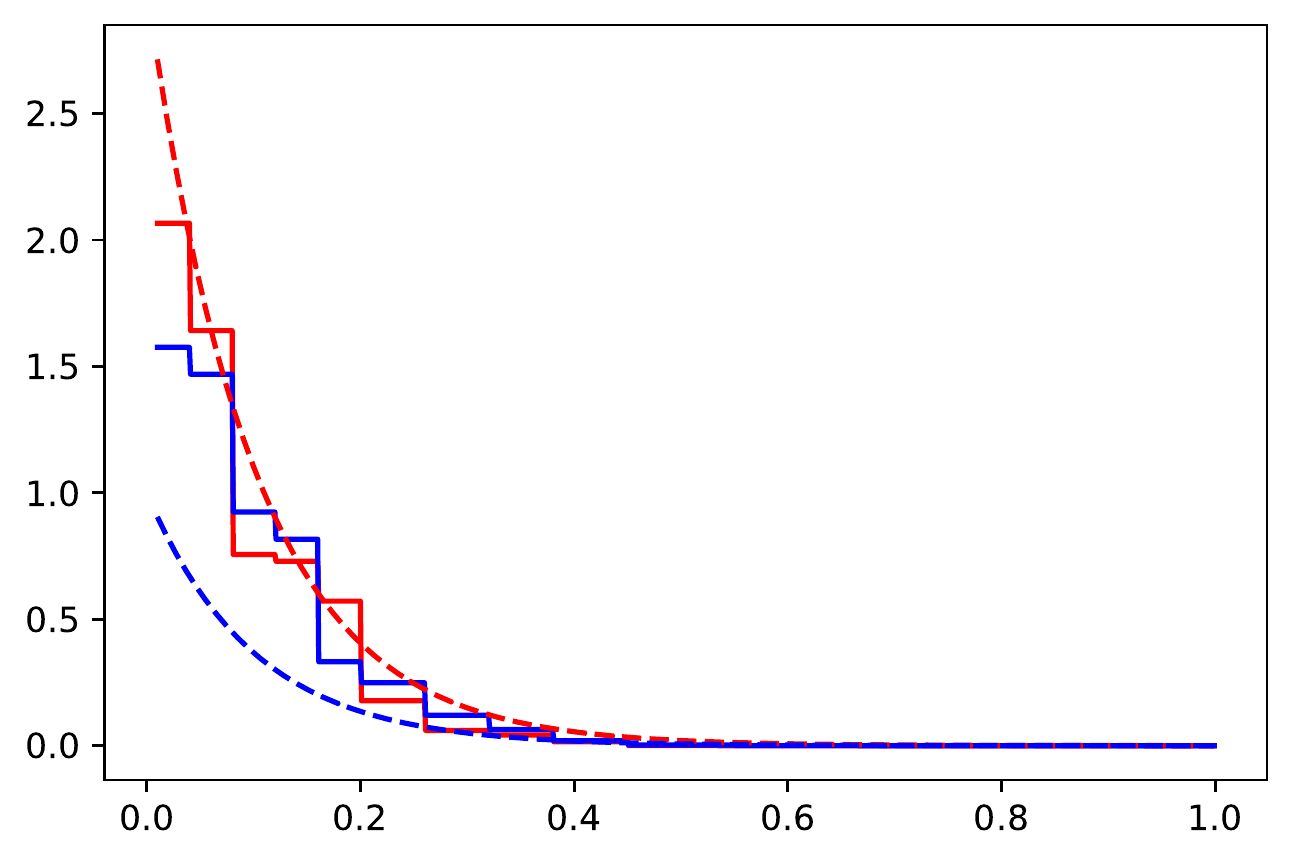}}
\subfigure{\includegraphics[scale=0.2]{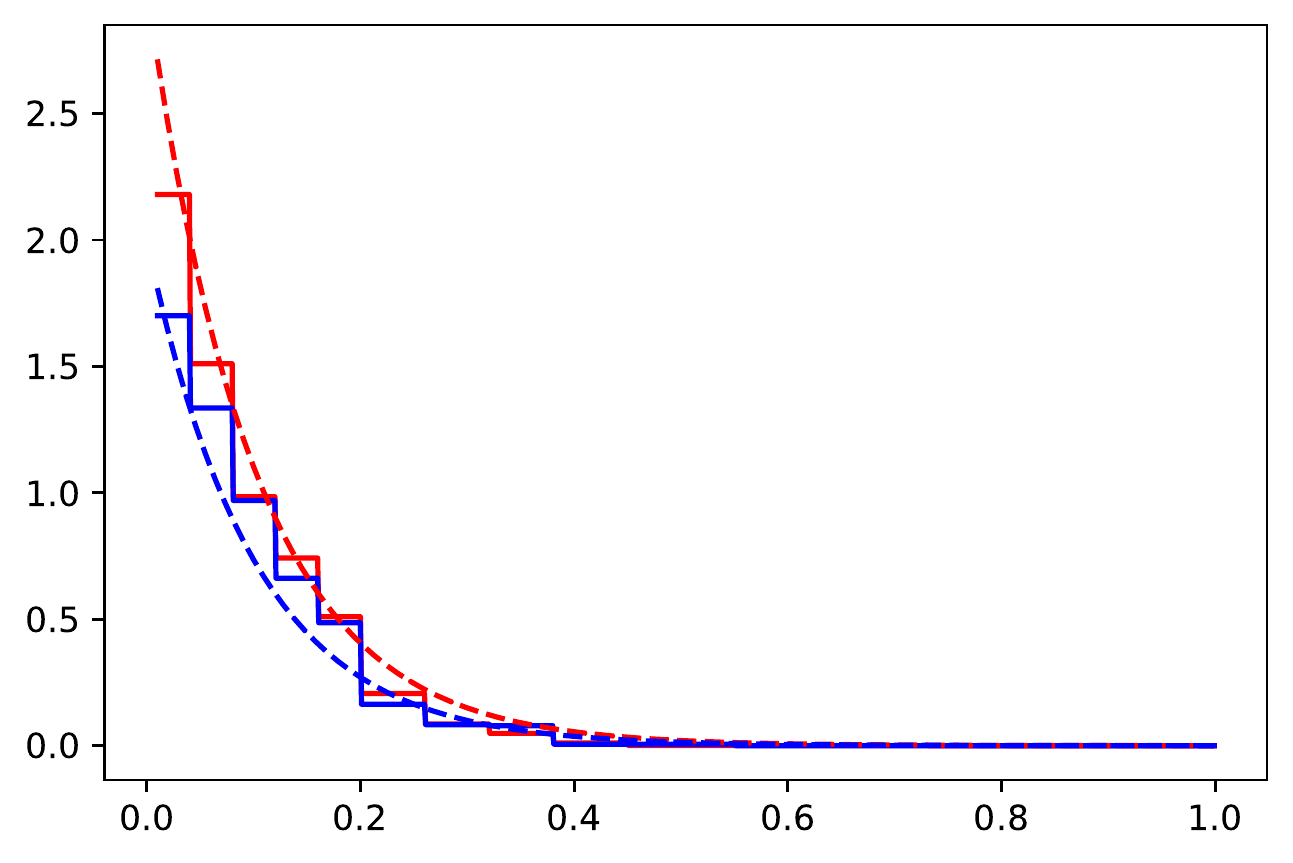}}
\subfigure{\includegraphics[scale=0.2]{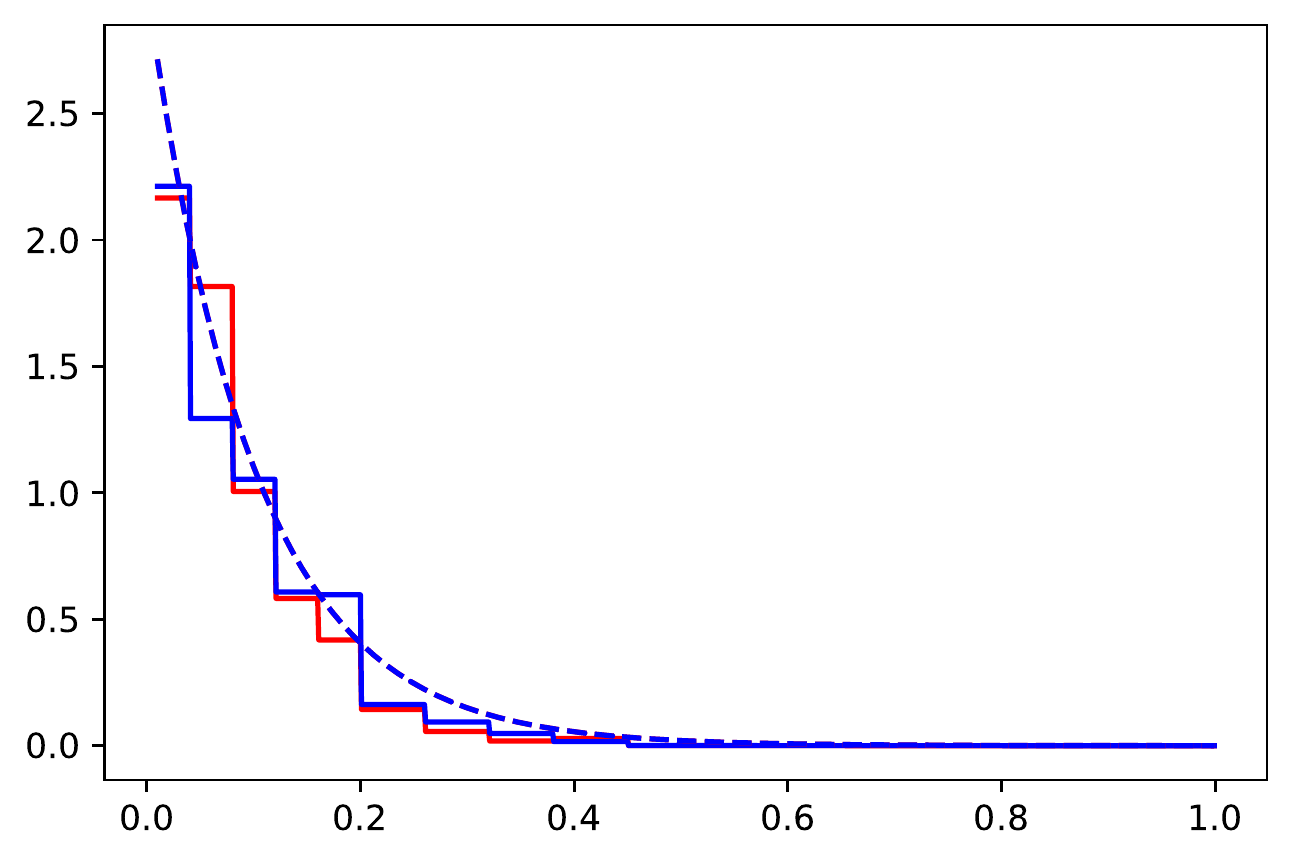}}
\subfigure{\includegraphics[scale=0.2]{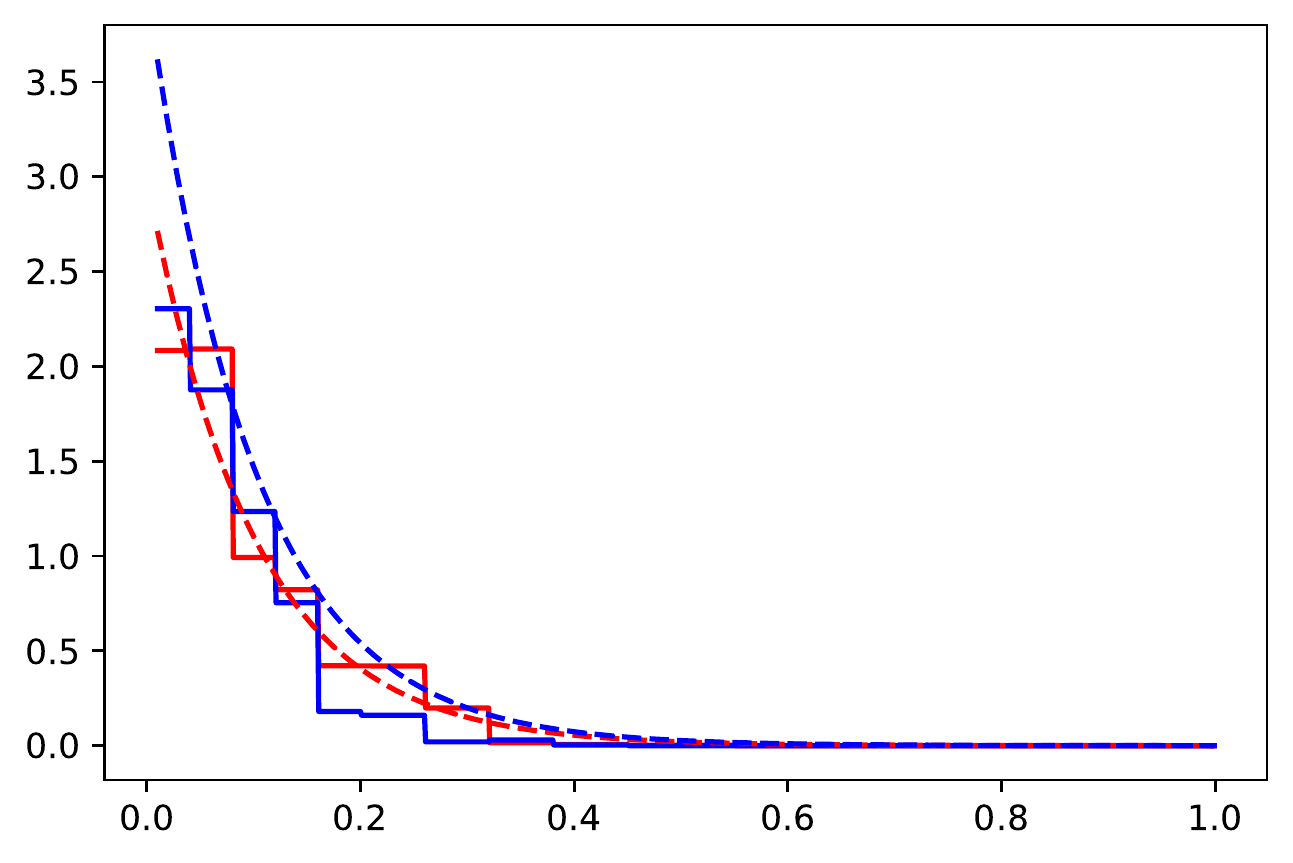}}
\subfigure{\includegraphics[scale=0.2]{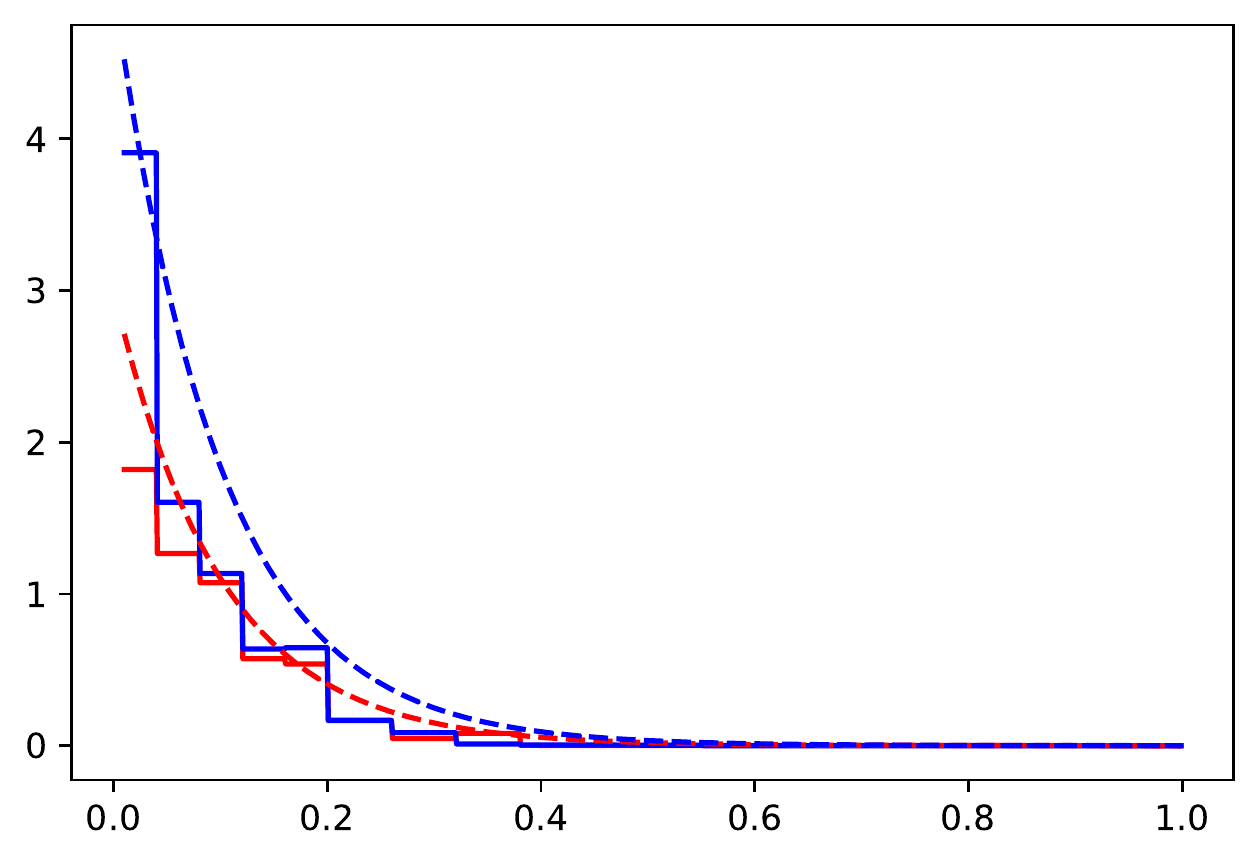}}
 
\subfigure{\includegraphics[scale=0.2]{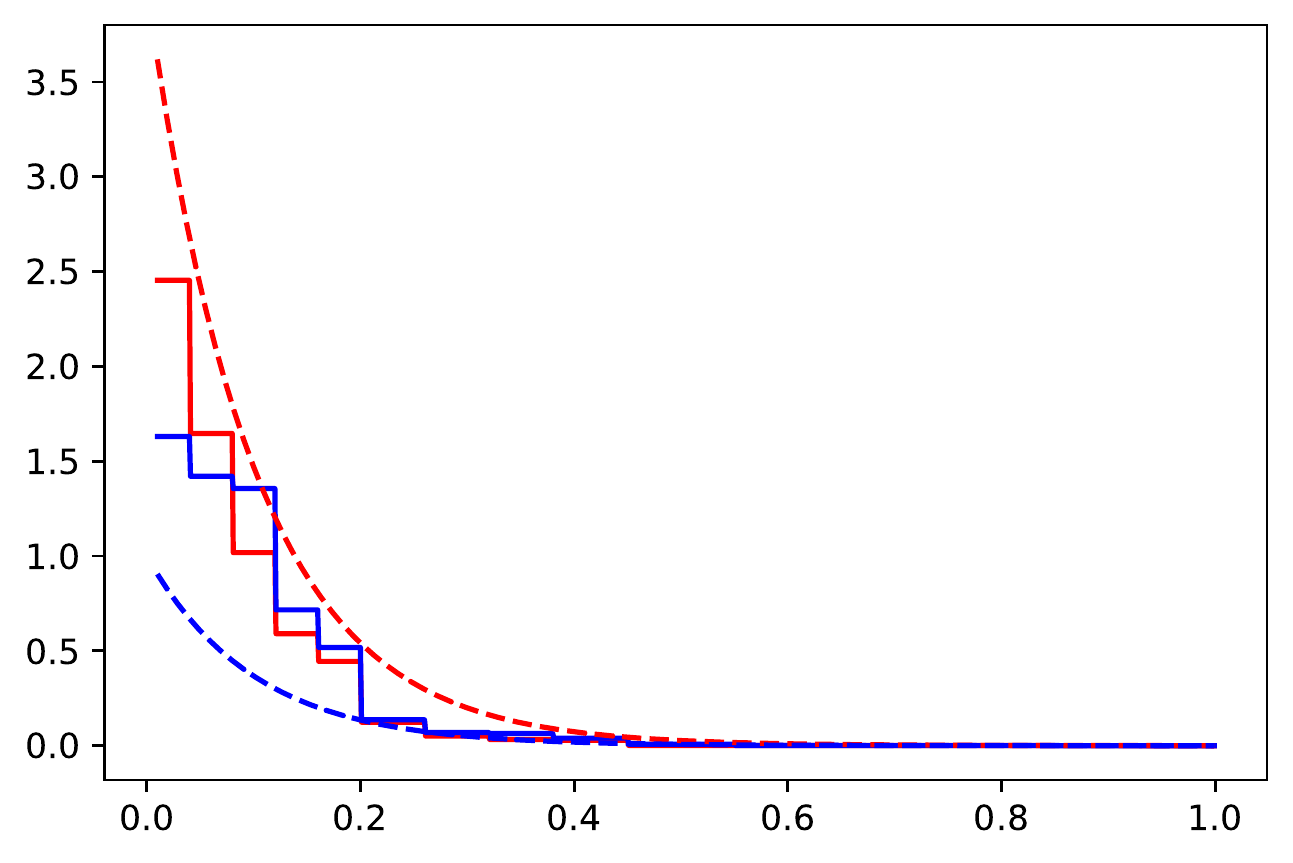}}
\subfigure{\includegraphics[scale=0.2]{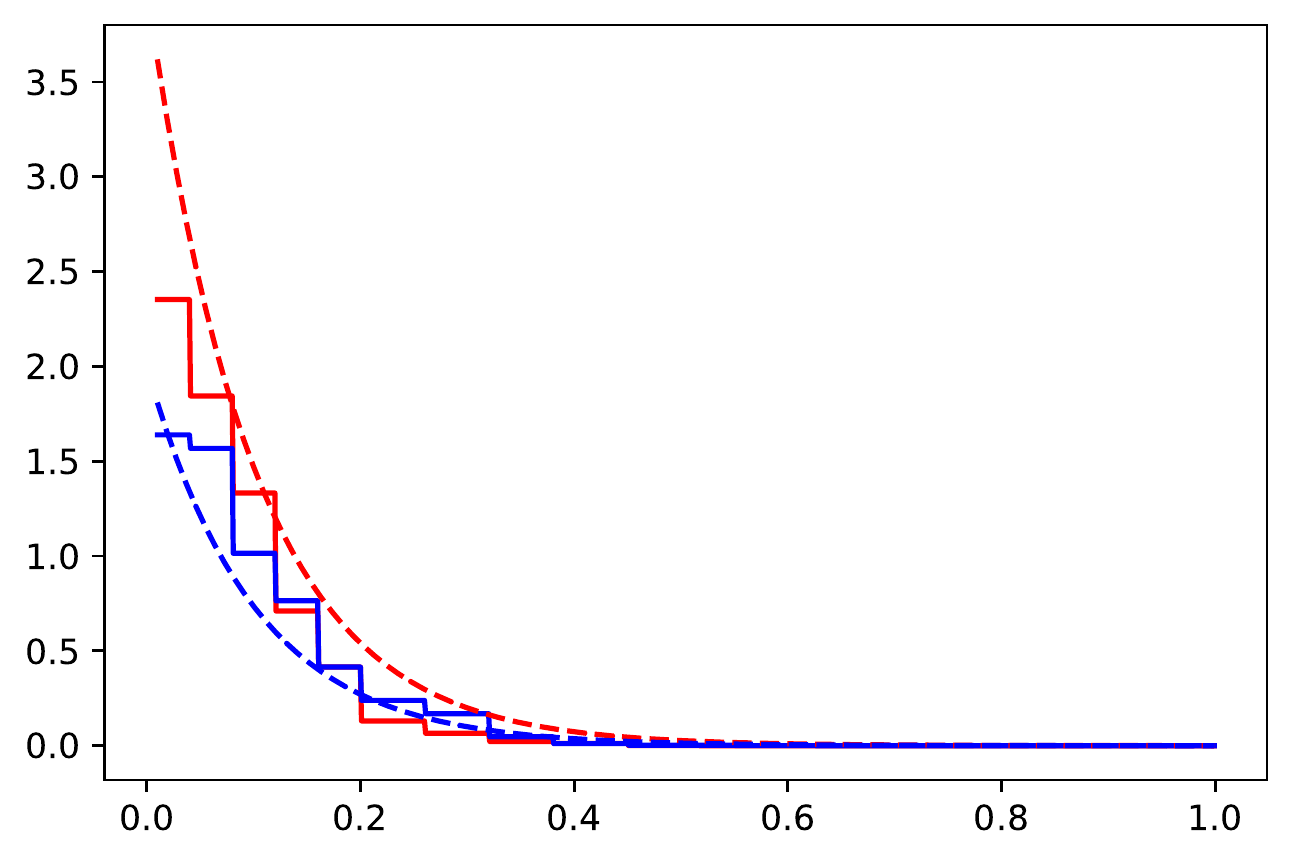}}
\subfigure{\includegraphics[scale=0.2]{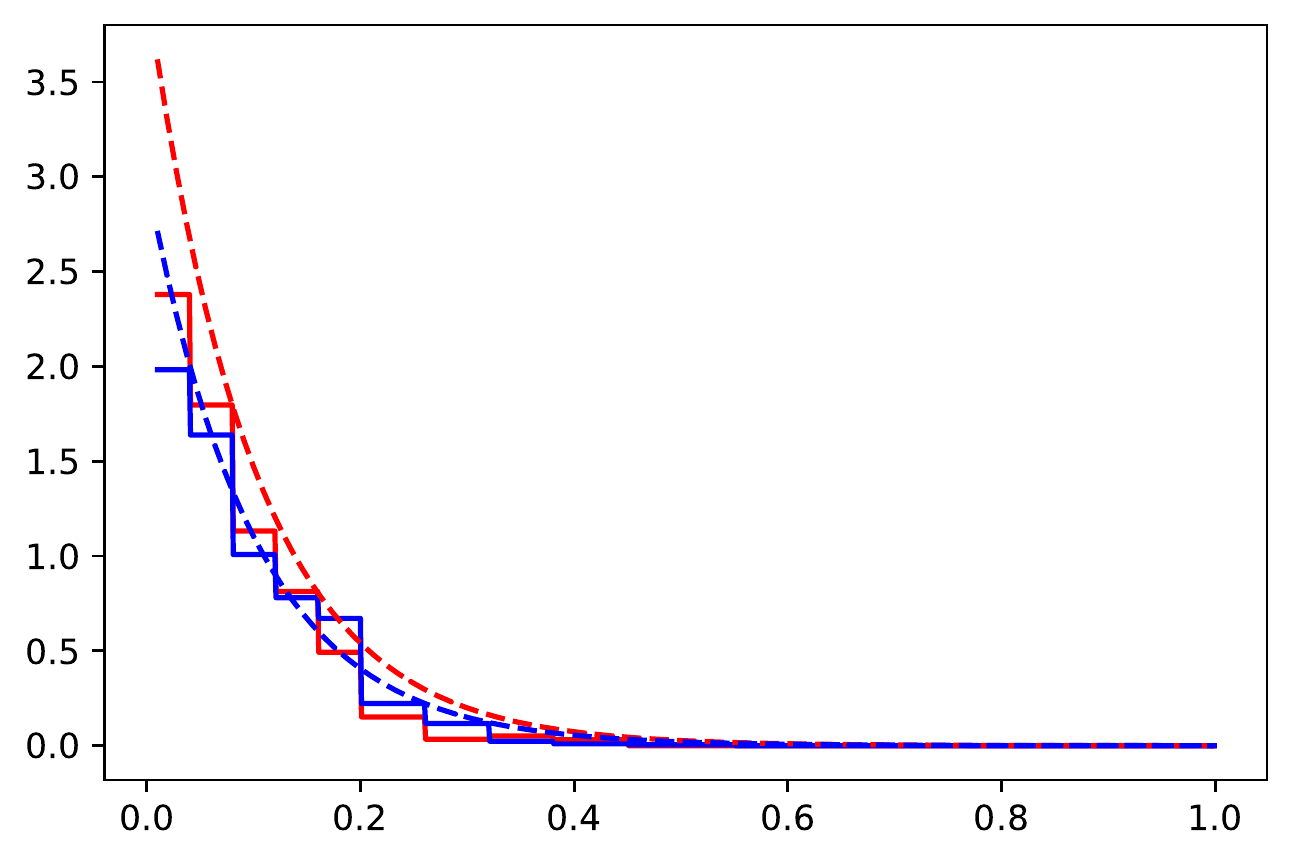}}
\subfigure{\includegraphics[scale=0.2]{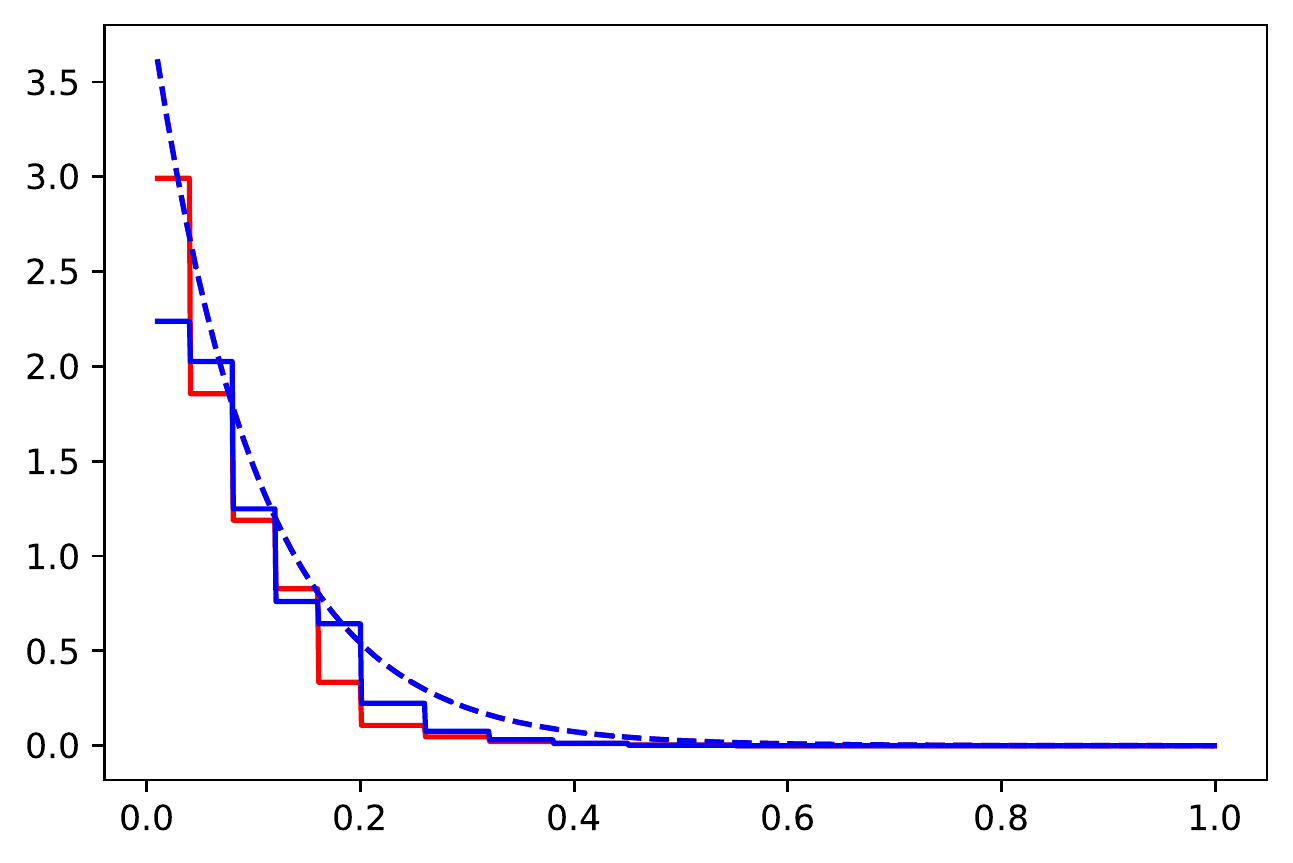}}
\subfigure{\includegraphics[scale=0.2]{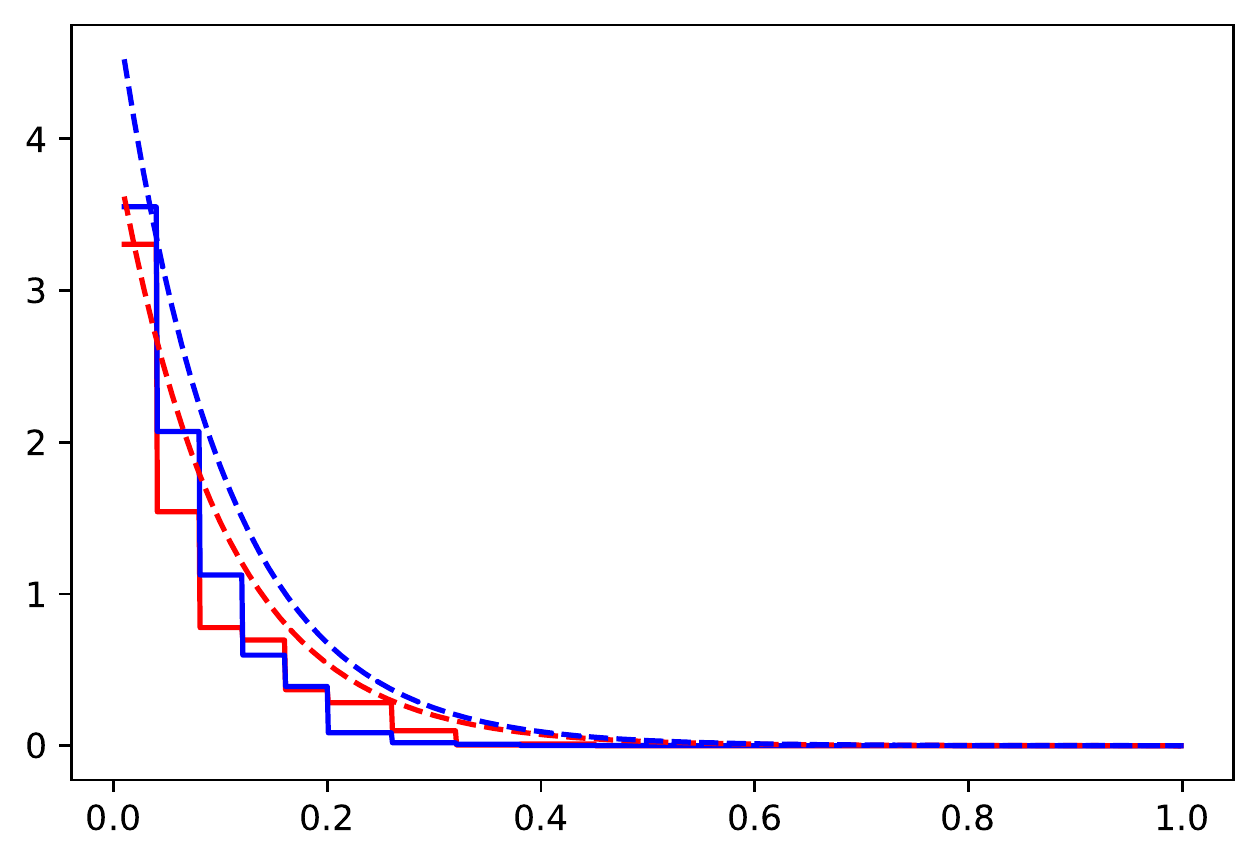}}
 
\subfigure{\includegraphics[scale=0.2]{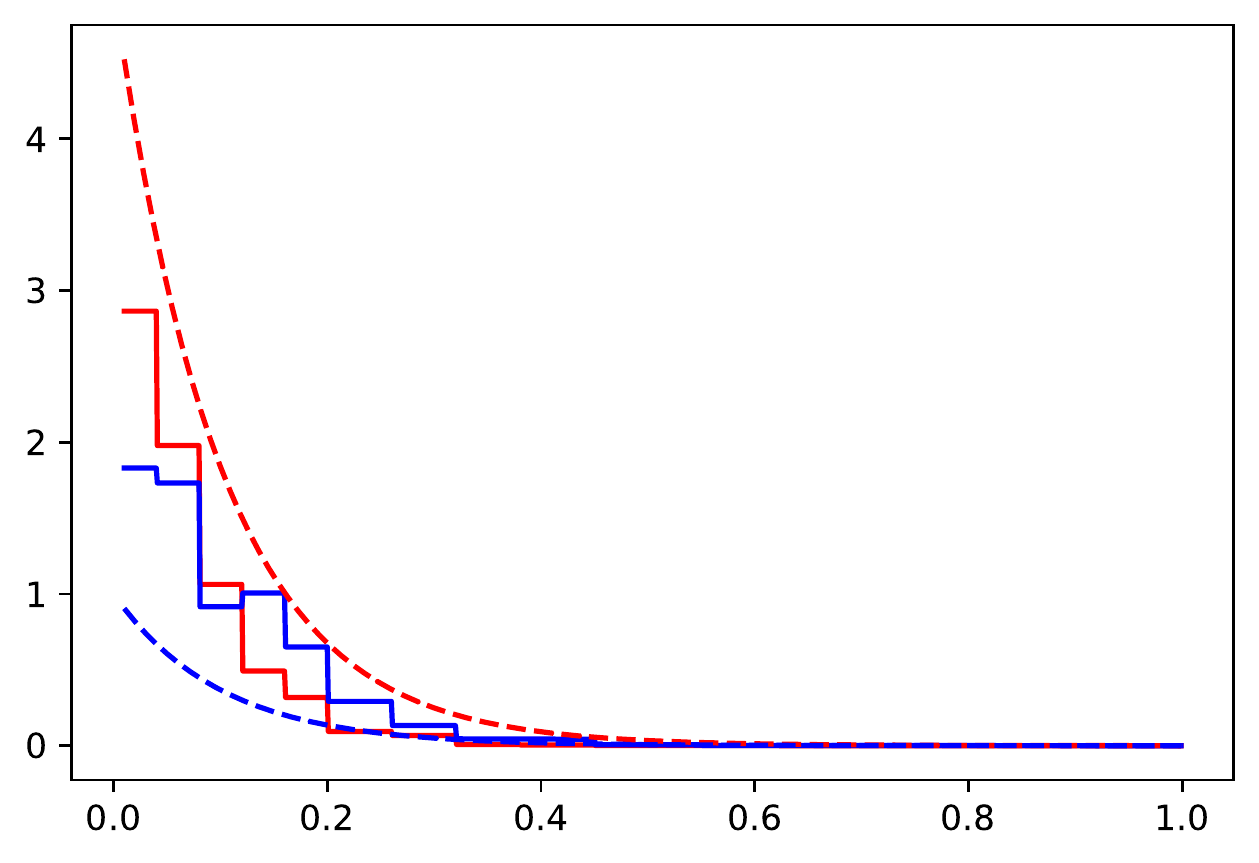}}
\subfigure{\includegraphics[scale=0.2]{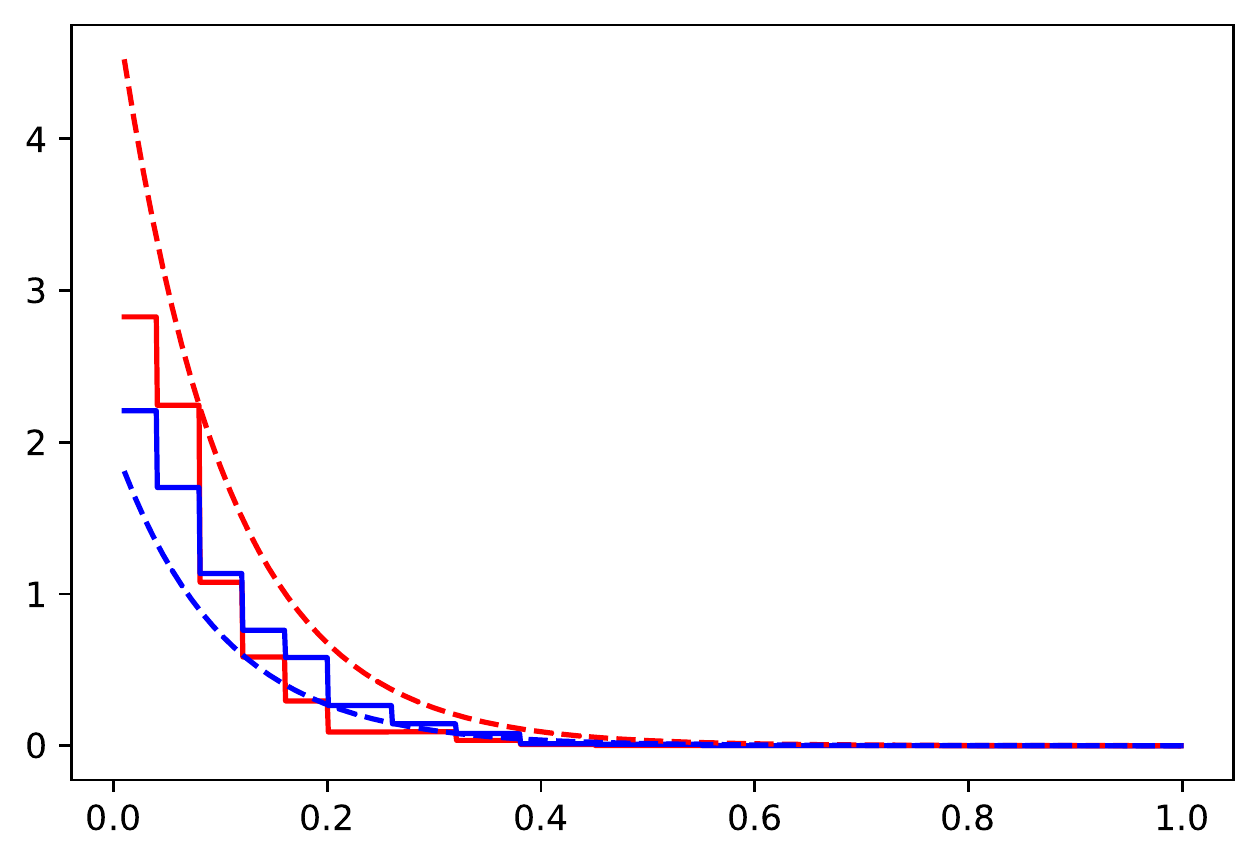}}
\subfigure{\includegraphics[scale=0.2]{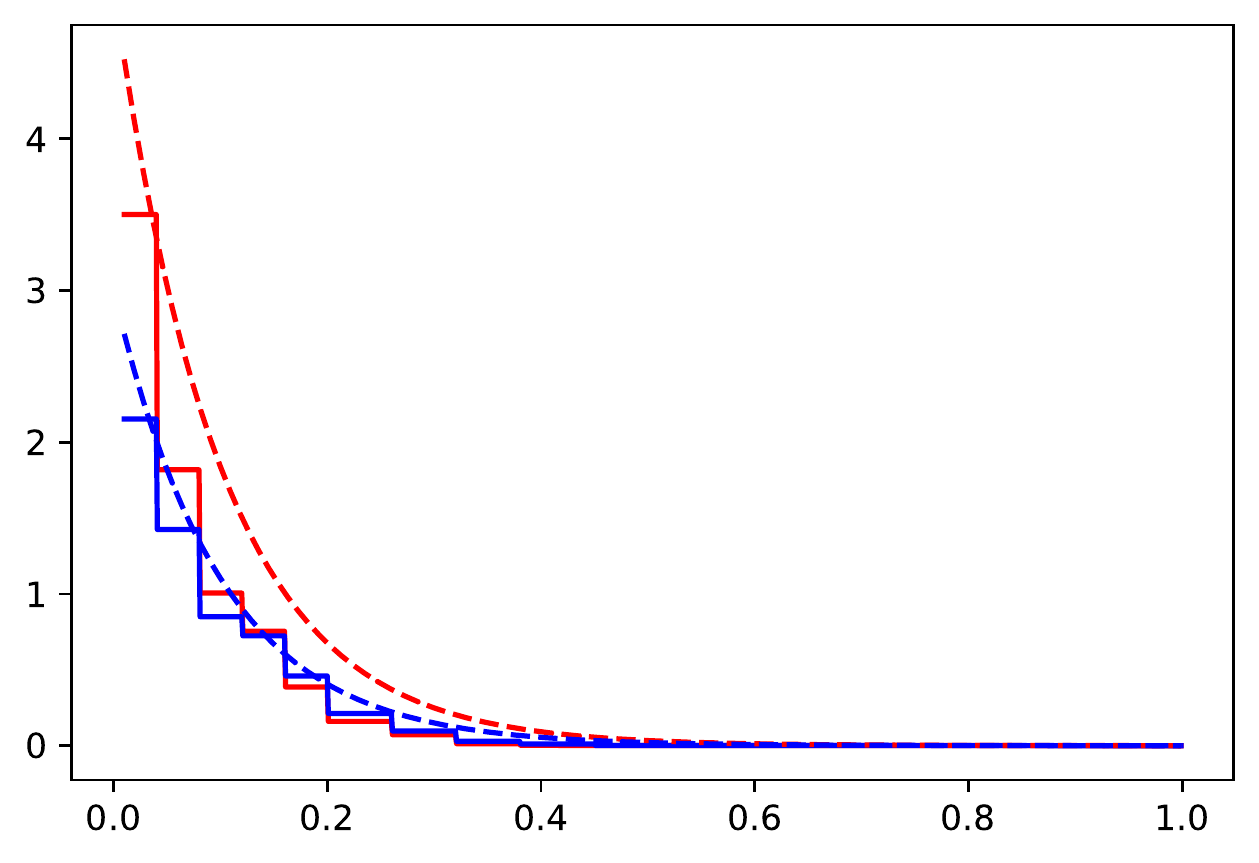}}
\subfigure{\includegraphics[scale=0.2]{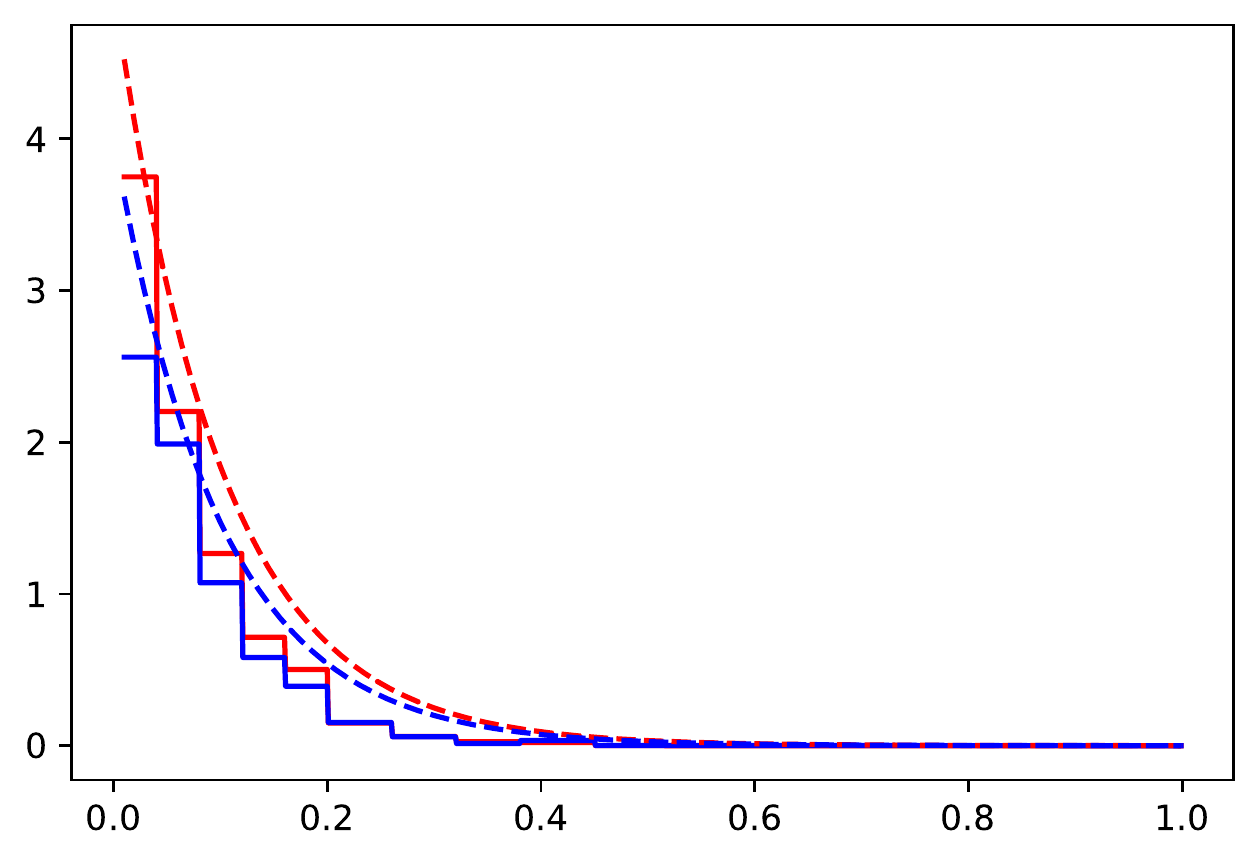}}
\subfigure{\includegraphics[scale=0.2]{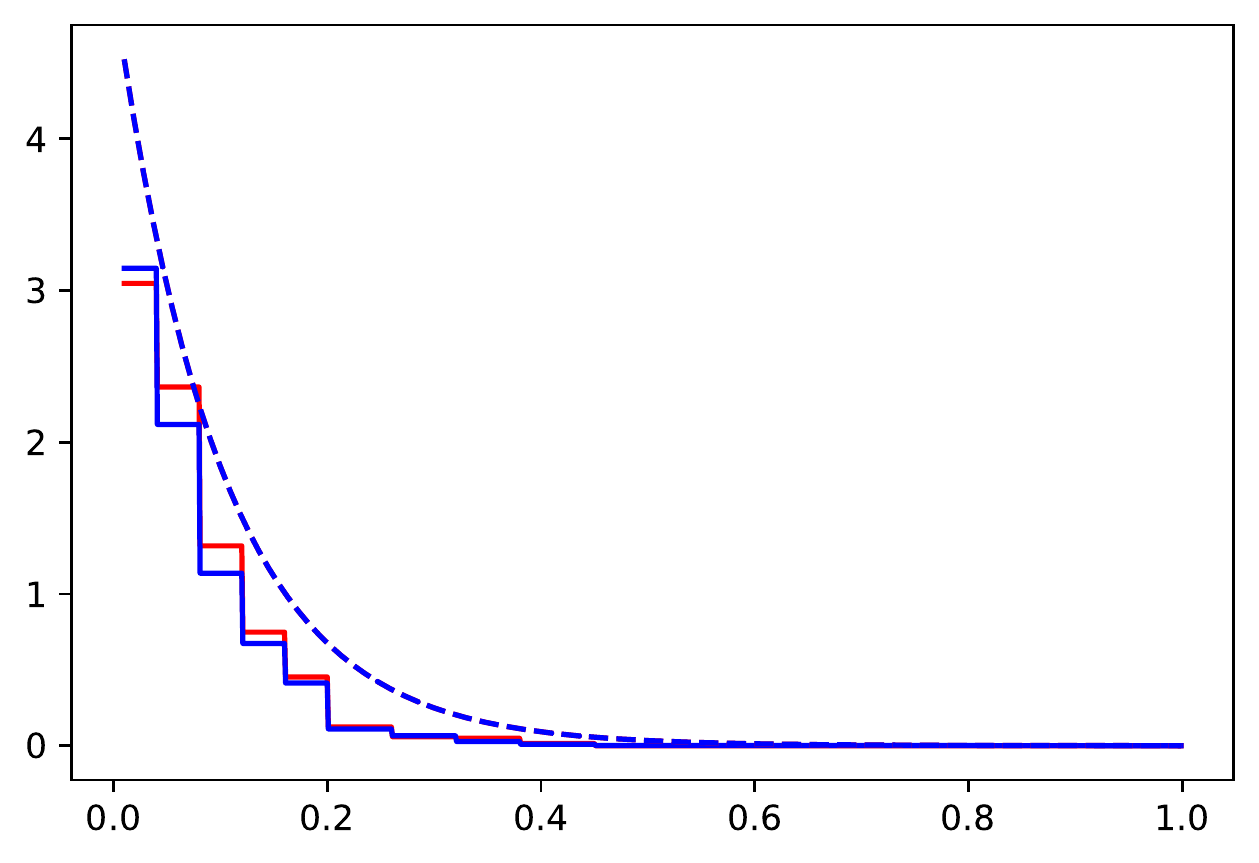}}

        %\vspace{-.1in}
    \caption{Histogram estimation under $H_1$: $\mu= 20, \alpha = 2, \beta = 1,2,3,4,5$. The solid line is the true triggering function whereas the dashed line is the estimated one (blue for $D_1$ and red for $D_2$).}
\end{figure}

We can see a very interesting pattern: when the true triggering functions for $D_1$ and $D_2$ are different (off-diagonal panels), histogram estimation tends to yields two piecewise constant triggering function lie between those two different true ones. This means the estimated difference between two triggering functions are much smaller than the truth, or rather $\norm{h(\Tilde{\theta}_{QMLE})}_2$ will be much smaller than it should be. By the form of GW test statistic \eqref{GWstats}, we can see the power of GW test statistic is highly dependent on $\norm{h(\Tilde{\theta}_{QMLE})}_2$ and histogram estimation will make the resulting GW test less powerful.
%\vspace{-.1in}
\section{Some useful functions}\label{specialfunc}
%\vspace{-.1in}
\subsection{Marcum-Q-function}
%\vspace{-.1in}
In statistics, the Marcum-Q-function $Q_{M}$ is defined as
$$Q_{M}(a, b)=\int_{b}^{\infty} x\left(\frac{x}{a}\right)^{M-1} \exp \left(-\frac{x^{2}+a^{2}}{2}\right) I_{M-1}(a x) d x,$$
or
$$Q_{M}(a, b)=\exp \left(-\frac{a^{2}+b^{2}}{2}\right) \sum_{k=1-M}^{\infty}\left(\frac{a}{b}\right)^{k} I_{k}(a b),$$with modified Bessel function $I_{M-1}(\cdot)$ of order $M-1$. \citet{10.2307/2332731} proved the following approximation formula:
$$Q_{k/2}(\sqrt{\lambda}, \sqrt{x}) \approx 1-  \Phi\left\{\frac{\left(\frac{x}{k+\lambda}\right)^{1 / 3}-\left(1-\frac{2}{9 f}\right)}{\sqrt{\frac{2}{9 f}}}\right\},$$ 
where $f=\frac{(k+\lambda)^{2}}{k+2 \lambda}=k+\frac{\lambda^{2}}{k+2 \lambda}$ and $\Phi(\cdot)$ is CDF of standard Gaussian random variable. We can easily verify that $Q_{k/2}(\sqrt{\lambda}, \sqrt{x}) \rightarrow 1$ as $\lambda \rightarrow \infty$. Also, this is illustrated in Figure~\ref{fig_power}. What's more, by the Theorem 1 in \citet{sun2010monotonicity}, Marcum-Q-function $Q_{M}(a,b)$ is monotonically increasing w.r.t. $a$.
%\vspace{-.1in}
\subsection{Matérn covariance function}
%\vspace{-.1in}
The Matérn covariance between two points separated by $d$ distance units is defined as
$$C_{\rho,\nu}(d)=\sigma^{2} \frac{2^{1-\nu}}{\Gamma(\nu)}\left(\sqrt{2 \nu} \frac{d}{\rho}\right)^{\nu} K_{\nu}\left(\sqrt{2 \nu} \frac{d}{\rho}\right).$$
where  $\Gamma(\cdot)$ is the gamma function, $K_{\nu }(\cdot)$ is the modified Bessel function of the second kind, and $\rho$ and $\nu$ are non-negative parameters of the covariance.
%\vspace{-.1in}
\section{Probability weighted histogram estimation under alternative hypothesis}\label{PWHE_full}
%\vspace{-.1in}
Under $H_1$, the triggering mechanism is more complex compared to univariate Hawkes Process, since each event can be either from the background,  direct offspring from an individual ancestor in Hawkes Process $1$ or Hawkes Process $2$ and the triggering effects of events in two different processes are different. 

We denote the branching structure matrix $P_{(z)(z')} \in \RR^{(N_1+N_2)\times(N_1+N_2)}$ ($z, z' \in \{1,2\}$). The element in $i-$th row and $j-$th column is defined to be the probability that event $i$ in process $z$  is triggered by event $j$  in process $z'$ if either $i\not= j$ or $z\not=z'$ (case 1) or the probability that event $i$ in process $ z$ is a background event if $i=j$ and $z=z'$ (case 2). That is,
\begin{equation*} 
    \Big(P_{(z)(z')}\Big)_{i j}=\left\{\begin{array}{ll}
\text {probability that event } i \text { in process } z \text {  is triggered by event } j  \text { in process } z',  &\text {case}\  1 \\
\text {probability that event } i \text { in process } z \text { is a background event, }  &\text {case}\  2
\end{array}\right.
\end{equation*}
Note that the probability is zero when $t_i^{(z)}\leq t_j^{(z')}$, which means the event that happens earlier in the process cannot be triggered by those which happen later.

As discussed above, we focus on differentiating difference in triggering effect. Thus we estimate the sum of two background intensities from all background events:
\begin{equation}\label{algo2_background}
    \mu^{(v)}=\frac{1}{T } \sum_{z=1}^2 \sum_{i=1}^{N_z} \Big(P_{(z)(z)}^{(v)}\Big)_{i i}.
\end{equation}

For the triggering components, we estimate the magnitude for process $z \ (z=1,2)$ using events from aggregated data triggered by process $z$ and estimate the temporal triggering density function from those events which fall into the corresponding bin. 
% For $z =1,2$ and $ k=1, \ldots, n_{0}$:
% \begin{align*}
%     \alpha_z^{(v)} =& \frac{\sum_{i=1}^{N_z}  \sum_{j=1}^{i-1} \Big(P_{(z)(z)}^{(v)}\Big)_{i j} + \sum_{i=1}^{N_{z'}}\sum_{j=1}^{N_z} \Big(P_{(z')(z)}^{(v)}\Big)_{i j}\mathbf{1}_{\{t_{i}^{(z')}>t_{j}^{(z)}\}}}{N_z}, \\
% g_{z,k}^{(v)}=&\frac{\sum_{i=1}^{N_z}  \sum_{j=1}^{i-1} \Big(P_{(z)(z)}^{(v)}\Big)_{i j}\mathbf{1}_{B_{k}}\left(t_{i}^{(z)}-t_{j}^{(z)}\right) + \sum_{i=1}^{N_{z'}}\sum_{j=1}^{N_z} \Big(P_{(z')(z)}^{(v)}\Big)_{i j}\mathbf{1}_{B_{k}}\left(t_{i}^{(z')}-t_{j}^{(z)}\right)\mathbf{1}_{\{t_{i}^{(z')}>t_{j}^{(z)}\}}}{\Delta t_{k} \left(\sum_{i=1}^{N_z} \sum_{j=1}^{i-1} \Big(P_{(z)(z)}^{(v)}\Big)_{i j} + \sum_{i=1}^{N_{z'}}\sum_{j=1}^{N_z} \Big(P_{(z')(z)}^{(v)}\Big)_{i j}\mathbf{1}_{\{t_{i}^{(z')}>t_{j}^{(z)}\}}\right)}.
% \end{align*}
Note that we have$$\Big(P_{(z')(z)}^{(v)}\Big)_{i j} = \Big(P_{(z')(z)}^{(v)}\Big)_{i j}\mathbf{1}_{\{t_{i}^{(z')}>t_{j}^{(z)}\}}.$$  This is because the probability will be zero if $t_{i}^{(z')}\leq t_{j}^{(z)}$ as discussed above. Thus, for $z =1,2$ and $ k=1, \ldots, n_{0}$, the estimators can be expressed as
\begin{align}
    \alpha_z^{(v)} &= \frac{\sum_{i=1}^{N_z}  \sum_{j=1}^{i-1} \Big(P_{(z)(z)}^{(v)}\Big)_{i j} + \sum_{i=1}^{N_{z'}}\sum_{j=1}^{N_z} \Big(P_{(z')(z)}^{(v)}\Big)_{i j}}{N_z}, \label{algo2_trigger_mag}\\
g_{z,k}^{(v)}&=\frac{\sum_{i=1}^{N_z}  \sum_{j=1}^{i-1} \Big(P_{(z)(z)}^{(v)}\Big)_{i j}\mathbf{1}_{B_{k}}\left(t_{i}^{(z)}-t_{j}^{(z)}\right) + \sum_{i=1}^{N_{z'}}\sum_{j=1}^{N_z} \Big(P_{(z')(z)}^{(v)}\Big)_{i j}\mathbf{1}_{B_{k}}\left(t_{i}^{(z')}-t_{j}^{(z)}\right)}{\Delta t_{k} \left(\sum_{i=1}^{N_z} \sum_{j=1}^{i-1} \Big(P_{(z)(z)}^{(v)}\Big)_{i j} + \sum_{i=1}^{N_{z'}}\sum_{j=1}^{N_z} \Big(P_{(z')(z)}^{(v)}\Big)_{i j}\right)}.\label{algo2_trigger_time}
\end{align}
And the updates for the branching probabilities are similar, for $z=1,2$, $z\not=z'$ and $i=1,2,\dots,N_z$:
\begin{align}
	&\lambda_{z,i}^{(v)} = \mu^{(v)} + \sum_{j=1}^{i-1} \alpha_z^{(v)} g_z^{(v)}\left(t_{i}^{(z)}-t_{j}^{(z)}\right)  + \sum_{j=1}^{N_{z'}}  \alpha_{z'}^{(v)} g_{z'}^{(v)}\left(t_{i}^{(z)} - t_{j}^{(z')}\right) \mathbf{1}_{\{t_{i}^{(z)}>t_{j}^{(z')}\}} \nonumber
	\\
	& \Big(P_{(z)(z)}^{(v+1)}\Big)_{i i}=\frac{\mu^{(v)}}{\lambda_{z,i}^{(v)}} \label{algo2_background_prob}\\
	&\Big(P_{(z)(z)}^{(v+1)}\Big)_{i j}=\frac{ \alpha_{z}^{(v)} g_z^{(v)}\left(t_{i}^{(z)}-t_{j}^{(z)}\right) }{\lambda_{z,i}^{(v)}} \ \ (\text{for} \ i>j) \label{algo2_trigger_prob} \\
&\Big(P_{(z)(z')}^{(v+1)}\Big)_{i j}=\frac{\alpha_{z'}^{(v)} g_{z'}^{(v)}\left(t_{i}^{(z)} - t_{j}^{(z')}\right)\mathbf{1}_{\{t_{i}^{(z)}>t_{j}^{(z')}\}}}{\lambda_{z,i}^{(v)}\label{algo2_trigger_cross_prob}}  
\end{align}
Here we summarize the algorithm as follows:
\begin{algorithm}[H]
    \caption{Probability Weighted Histogram Estimation of Quasi-log-likelihood under $H_1$}\label{algo2}
    \begin{algorithmic}

	\STATE{\textbf{Initialize}: choose stopping critical value $\epsilon$ (e.g. $10^{-3}$), initialize $P_{(z)(z')}^{(0)}$ and set $\Big(P_{(z)(z')}^{(-1)}\Big)_{ij} = \epsilon + \Big(P_{(z)(z')}^{(0)}\Big)_{ij}$ and iteration index $v=0$.}
	\WHILE{$\max _{t_i^{(z)}> t_j^{(z')}}\left|\Big(P_{(z)(z')}^{(v)}\Big)_{i j}-\Big(P_{(z)(z')}^{(v-1)}\Big)_{i j}\right|<\varepsilon$}
	\STATE{1. Estimate background rate $\mu $ as in \eqref{algo2_background}.}
	\STATE{2. Estimate triggering components $\alpha_z$, $g_z(t)$ as in \eqref{algo2_trigger_mag} and \eqref{algo2_trigger_time}.}
	\STATE{3. Update probabilities $\Big(P_{(z)(z')}^{(v+1)}\Big)_{i j}$'s as in \eqref{algo2_background_prob}, \eqref{algo2_trigger_prob} and \eqref{algo2_trigger_cross_prob}.}
	\STATE{4. $v = v + 1$}
	\ENDWHILE
    \end{algorithmic}
\end{algorithm}

 We follow the derivation of EM-type algorithm in Appendix. \ref{algo1EM} and derive that Probability Weighted Histogram Estimation under the full model is again an EM-type algorithm. Similar to the proof framework above, we first use integral approximation of \citet{schoenberg2013facilitated} to approximate the Quasi-log-likelihood function and then lower bound it using Jensen's inequality:

\begin{equation*} \begin{split}  \Tilde{\ell}(\theta)& \approx
\sum_{z=1}^{2}  \sum_{i=1}^{N_z}\Bigg [ \Big(P_{(z)(z)}^{ }\Big)_{i i} \log \mu  + \sum_{j<i} \Big(P_{(z)(z)}^{ }\Big)_{i j}\left(\log \alpha_{z} + \log \left(\sum_{k=1}^{n_{0}} g_{z,k} \mathbf{1}_{B_{k}}\left(t_{i}^{(z)}-t_{j}^{(z)}\right)\right)\right)\\ 
&+ \sum_{j=1}^{N_{z'}} \Big(P_{(z)(z')}^{ }\Big)_{i j}\Bigg( \log \alpha_{z'} + \log \left(\sum_{k=1}^{n_{0}} g_{z',k}  \mathbf{1}_{B_{k}}\left(t_{i}^{(z)}-t_{j}^{(z')}\right)\right)\Bigg]-T \mu - N_1 \alpha_1 -N_2 \alpha_2\\
&    - \sum_{z=1}^{2} \sum_{i=1}^{N_z} \left(\sum_{i\geq j} \Big(P_{(z)(z)}^{ }\Big)_{i j} \log(\Big(P_{(z)(z)}^{ }\Big)_{i j})+\sum_{j=1}^{N_{z'}} \Big(P_{(z)(z')}^{ }\Big)_{i j} \log(\Big(P_{(z)(z')}^{ }\Big)_{i j})\mathbf{1}_{\{t_{i}^{(z)}>t_{j}^{(z')}\}} \right),
\end{split}\end{equation*}where $z'\not=z$. Note that the term $- N_1 \alpha_1 -N_2 \alpha_2$ comes from integral approximation. Add Lagrange multipliers and we will get the following objective function:

\begin{equation*} \begin{split}  \Tilde{L}(\theta)& =
\sum_{z=1}^{2}  \sum_{i=1}^{N_z} \Bigg [\Big(P_{(z)(z)}^{ }\Big)_{i i} \log \mu + \sum_{j<i} \Big(P_{(z)(z)}^{ }\Big)_{i j}\left(\log \alpha_{z} + \log \left(\sum_{k=1}^{n_{0}} g_{z,k} \mathbf{1}_{B_{k}}\left(t_{i}^{(z)}-t_{j}^{(z)}\right)\right) \right)\\ 
&+ \sum_{j=1}^{N_{z'}} \Big(P_{(z)(z')}^{ }\Big)_{i j}\Bigg( \log \alpha_{z'} + \log \left(\sum_{k=1}^{n_{0}} g_{z',k} \mathbf{1}_{B_{k}}\left(t_{i}^{(z)}-t_{j}^{(z')}\right)\right)\Bigg ]-T \mu - N_1 \alpha_1 -N_2 \alpha_2 \\
&    - \sum_{z=1}^{2} \sum_{i=1}^{N_z} \left(\sum_{i\geq j} \Big(P_{(z)(z)}^{ }\Big)_{i j} \log(\Big(P_{(z)(z)}^{ }\Big)_{i j})+\sum_{j=1}^{N_{z'}} \Big(P_{(z)(z')}^{ }\Big)_{i j} \log(\Big(P_{(z)(z')}^{ }\Big)_{i j})\mathbf{1}_{\{t_{i}^{(z)}>t_{j}^{(z')}\}} \right)\\
&-\sum_{z=1}^{2} \left [c_{z,1}\left(\sum_{k=1}^{n_{0}} g_{z,k} \Delta t_{k}-1\right)-\sum_{i=1}^{N_z} c_{z,3}^{(i)}\left(\sum_{i\geq j} \Big(P_{(z)(z)}^{ }\Big)_{i j}+ \sum_{j=1}^{N_{z'}}\Big(P_{(z)(z')}^{ }\Big)_{i j} \mathbf{1}_{\{t_{i}^{(z)}>t_{j}^{(z')}\}} - 1\right)\right ].
\end{split}\end{equation*}Then, by taking first order derivatives and setting them to zero we can validate Algorithm~\ref{algo2} as an EM-type algorithm.

\end{appendices}
\end{document}